\documentclass{article}
\usepackage{amssymb}
\usepackage{amsmath}
\usepackage{arxiv}
\usepackage[utf8]{inputenc}
\usepackage[T1]{fontenc}
\usepackage{url}
\usepackage{booktabs}
\usepackage{amsfonts}
\usepackage{nicefrac}
\usepackage{microtype}
\usepackage{color}
\usepackage{lipsum}
\usepackage{amsthm}
\usepackage{indentfirst}
\usepackage{graphicx}
\usepackage{epstopdf}
\usepackage{fourier}
\usepackage{bm}
\usepackage{mathtools}
\usepackage{latexsym,enumerate}
\usepackage{multicol}
\usepackage[skip=2pt]{caption}
\usepackage[font=small,skip=0pt]{subcaption}
\usepackage{array}
\usepackage{ascii}
\usepackage{algorithm,algorithmic}
\usepackage{mathtools}
\usepackage{arydshln}
\usepackage{soul}
\usepackage[pdfstartview=FitH, CJKbookmarks=true,
bookmarksnumbered=true, bookmarksopen=true,
colorlinks,
linkcolor=green,
anchorcolor=blue, citecolor=blue
]{hyperref}

\newcommand{\comment}[1]{}

\newcommand{\BEA}{\begin{eqnarray}}
\newcommand{\EEA}{\end{eqnarray}}

\newtheorem{thm}{Theorem}[section]
\newtheorem{prop}[thm]{Proposition}
\newtheorem{example}[thm]{Example}
\newtheorem{lemma}[thm]{Lemma}

\newtheorem{definition}[thm]{Definition}

\newtheorem{rem}[thm]{Remark}

\newcommand{\PreserveBackslash}[1]{\let\temp=\\#1\let\\=\temp}
\newcolumntype{C}[1]{>{\PreserveBackslash\centering}p{#1}}
\newcolumntype{R}[1]{>{\PreserveBackslash\raggedleft}p{#1}}
\newcolumntype{L}[1]{>{\PreserveBackslash\raggedright}p{#1}}
\newcommand{\stkout}[1]{\ifmmode\text{\sout{\ensuremath{#1}}}\else\sout{#1}\fi}

\begin{document}

\title{Two-step Generalized RBF-Generated Finite Difference Method  on Manifolds}
\author{ Rongji Li \\
School of Information Science and Technology, ShanghaiTech University,
Shanghai 201210, China\\
\texttt{lirj2022@shanghaitech.edu.cn} \And Haichuan Di \\
School of Information Science and Technology, ShanghaiTech University,
Shanghai 201210, China\\
\texttt{dihch2024@shanghaitech.edu.cn} \And Shixiao Willing Jiang \\
Institute of Mathematical Sciences, ShanghaiTech University, Shanghai
201210, China\\
\texttt{jiangshx@shanghaitech.edu.cn} }
\date{\today }
\maketitle

\begin{abstract}
Solving partial differential equations (PDEs) on manifolds  has broad applications in various fields. In this paper, we develop a two-step generalized radial basis function-generated finite difference (gRBF-FD) method for solving PDEs on manifolds without boundaries, identified by randomly sampled point cloud data. The gRBF-FD is based on polyharmonic spline kernels and multivariate polynomials (PHS+Poly) defined over the tangent space in a local Monge coordinate system. The first step is to regress the local target function using a generalized moving least-squares (GMLS)  while the second step is to compensate for the residual  using a PHS interpolation. Our gRBF-FD method has the same interpolant form with the standard RBF-FD but differs in interpolation coefficients.
Our approach utilizes a specific weight function in both the GMLS and PHS steps and implements an automatic tuning strategy for the stencil size $K$ (i.e., the number of nearest neighbors) at each point. These strategies are designed to produce a Laplacian matrix with a specific coefficient structure, thereby enhancing stability and reducing the solution error.
We establish an error bound for the operator approximation in terms of the so-called local stencil diameter as well as in terms of the number of data. We further demonstrate the high accuracy of gRBF-FD through numerical tests on various smooth manifolds from 1D to 4D.
\end{abstract}


\lhead{} \newpage

\section{Introduction}

Partial differential equations (PDEs) on manifolds have a wide range of applications in physics, biology, and engineering problems. For example, the applications include phase separation on biomembranes \cite{elliott2010modeling},liquid crystal alignment on deformable interfaces \cite{nitschke2020liquid}, morphogen transport on cell surfaces \cite{barreira2011surface}, image processing on curved domains \cite{bertalmio2000image} and surface reconstruction \cite{zhao2001fast}, just to list a few. These broad applications motivate the development of accurate and efficient  numerical solvers. In this paper, we focus on a type of mesh-free approach based on local approximations using radial basis functions (RBFs) and multivariate polynomials for solving PDEs on manifolds.


The early application of global RBF collocation methods for solving PDEs can be traced back to Kansa's method  in 1990 \cite{kansa1990multiquadrics}.
Following this success, various RBF collocation methods were subsequently  developed for solving PDEs in Euclidean spaces and on manifolds \cite{franke1998solving,fasshauer1996solving,piret2012orthogonal,fuselier2013high,Natasha2015Solving,harlim2023radial}.
With the dawn of the 21st century,  the local radial basis function-generated finite difference (RBF-FD) method was introduced by several groups  as a replacement of global RBFs \cite{shu2003local, tolstykh2003using, wright2003radial,cecil2004numerical, Wendland2005Scat,shankar2015radial,lehto2017radial,flyer2016role,shaw2019radial,jones2023generalized}.
RBF-FD, as a local method, utilizes both radial basis functions and multivariate polynomials to approximate functions locally.


The principal advantages of both global RBF and local RBF-FD methods include  meshlessness, geometric flexibility, dimension independence, high-order accuracy, and implementation simplicity \cite{Delengov2018RBF,flyer2013radial,flyer2011radial}.
However, RBF-FD is superior to global RBF for the following reasons.
(1) While powerful for small problems (e.g., the number of points $N <10^4$), global RBF methods do not scale well to larger cases due to their expensive computational costs induced by dense matrix systems \cite{glaubitz2024energy}. However, RBF-FD employs sparse matrices enabling efficient scaling to relatively large problems (e.g., $10^4 <N <10^6$) while maintaining high-order accuracy of solutions.
(2) While global RBFs face fundamental challenges such as ill-conditioned matrices and shape parameter tuning, RBF-FD, in contrast, achieves high accuracy by leveraging local processing to handle these issues efficiently during local interpolation and differentiation (see \cite{fuselier2013high} and ref therein).
(3) The global nature of RBFs makes local refinement inefficient, while RBF-FD naturally supports adaptive node placement via compact stencils without global reconfiguration \cite{le2023guidelines}.
(4) Solving dense and large-scale systems in global RBFs challenges parallelization, whereas RBF-FD's sparse structure readily accommodates modern parallel computing frameworks \cite{bollig2012solution,tillenius2015scalable}.
(5) Global RBFs need specialized techniques like ghost points to handle boundaries \cite{fedoseyev2002improved, larsson2003numerical,chen2020novel,ma2021ghost}, whereas RBF-FD maintains high accuracy without special techniques for boundaries as demonstrated in references \cite{bayona2017role,bayona2019role}.


Other mesh-free collocation  methods for solving PDEs include generalized moving least-squares (GMLS) \cite{liang2013solving,gross2020meshfree,li2024generalized}, graph-based approaches \cite{li2016convergent,li2017point,gh2019,jiang2023ghost,yan2023kernel}, and generalized finite difference method (GFDM) \cite{chen2017meshfree,suchde2019meshfree,jiang2024generalized,halada2025overview}, just to name a few. Some characteristics of these methods are listed in \cite{jiang2024generalized} for practical applications.
In particular, GMLS method was early found as a powerful tool for surface fitting \cite{lancaster1981surfaces} and was later extended to the solution of partial differential equations \cite{nayroles1992generalizing,belytschko1994element}.
The theoretical framework of GMLS was established in \cite{Wendland2005Scat,mirzaei2012generalized}
for linear functional approximations.
We will review the GMLS method on manifolds, which is based on local polynomial regression in the tangent space
\cite{liang2013solving,gross2020meshfree,jones2023generalized,li2024generalized}. In our numerical experiments, we will compare our RBF-FD results with those from GMLS for  approximating differential operators and solving PDEs on manifolds.


In this work, we consider a two-step generalized RBF-generated finite difference (gRBF-FD) approach for solving screened Poisson equations on manifolds without boundaries, identified by randomly sampled point cloud data. The basic idea of gRBF-FD is to regress the local function using GMLS followed by fitting the residual using an RBF interpolation.
The first GMLS regression step is to capture the smooth, leading Taylor's expansion component of the target function while the second RBF interpolation step is to compensate for the expansion's remainder  contained within the residual.
Compared to standard RBF-FD, our gRBF-FD method employs the same  interpolant form but differs in the interpolation coefficients. In addition, for problems on manifolds, we follow the tangent-plane  method in \cite{shaw2019radial,wright2023mgm,jones2023generalized}, that is, local polyharmonic spline kernels and
polynomials  (PHS+Poly) are both defined over the projected tangent-space coordinates. It is natural for the tangent-plane PHS+Poly method to approximate the local representation of a function which is defined over the tangent space (see Sections \ref{sec:3.1} and \ref{sec:3.2}). Moreover, the tangent-plane method reduces the computational cost as seen from Remark \ref{rem:35}.


Other innovations of this work can be summarized as follows.\newline
\noindent \textbullet\ To stabilize the approximation, we employ a specific weight function \cite{liang2012geometric,liang2013solving,li2024generalized} in both GMLS regression and PHS interpolation steps to achieve a Laplacian matrix  that is nearly diagonally dominant. Here, we say a matrix is nearly diagonally dominant if  the absolute value of each diagonal entry is relatively large compared to those of the off-diagonal entries (see Definition \ref{def:4.3}). Numerical results demonstrate that the usage of the weight function ensures the stability of the Laplacian matrix.\newline
\noindent \textbullet\ The choice of the stencil size $K$ for each point is important for randomly sampled data. We propose an automatic strategy to tune the stencil size $K$ at each base point, delivering a numerically stable approximation of the discrete Laplacian matrix with a small operator approximation error. Indeed, we aim to select an appropriate $K$ for each point that navigates a trade-off: an overly small $K$ leads to instability, while an unnecessarily large $K$ increases the operator approximation error.\newline
\noindent \textbullet\ Theoretically, we show the approximation error bound of our gRBF-FD method in terms of the so-called stencil diameter for randomly sampled data. We further show that the approximation error rate of RBF-FD is the same as that of GMLS, which is  { of order $(\frac{\log N}{N})^{(l-1)/d}$}. Here, $N$ is the number of data {points}, $l$ is the polynomial degree, and $d$ is
the intrinsic dimension.
Numerically, we observe that by enhancing the approximation space, PHS functions in gRBF-FD help reduce both approximation  and solution errors compared to using GMLS.

The paper is organized as follows. In Section \ref{sec:pre}, we review the GMLS and the tangent-plane RBF-FD approaches on manifolds. In Section \ref{sec:gRBF-FD}, we present the two-step gRBF-FD method for approximating the Laplace–Beltrami operator and also point out its relation to the standard RBF-FD.
In Section \ref{sec:tunek}, we detail our numerical implementation, including normalization of the tangent-plane coordinates, the specific weight function used in two steps, and the auto-tuning strategy for the stencil size $K$. In Section \ref{sec:numerical},  we report supporting numerical results.
We conclude the paper with a summary and several open problems in Section \ref{sec:conclusion}.
We include some proofs in Appendix \ref{app:A} and Appendix \ref{app:B}.


\section{Preliminaries}

\label{sec:pre}

We consider a $d$-dimensional manifold $M$ embedded in the ambient Euclidean space $\mathbb{R}^n$. For a point $\mathbf{x} \in M$, the tangent space is denoted by $T_{\mathbf{x}}M$ with orthonormal basis $\left\{\boldsymbol{t}_1(\mathbf{x}),\ldots,\boldsymbol{t}_d(\mathbf{x})\right\}$. Let $f : M \to \mathbb{R}$ be a smooth function. We are given a set of $N$ distinct points $\mathbf{X}_{M}=\left\{ \mathbf{x}_{i}\right\}_{i=1}^{N}\subset M$  and the corresponding function values $\mathbf{f} = (f(\mathbf{x}_1),\ldots, f(\mathbf{x}_N))^\top$. We briefly review the GMLS  and RBF-FD methods in Sections \ref{sec:2_1} and \ref{sec:PRFD}, respectively, for approximating functions on manifolds using the set of scattered data points $\mathbf{X}_{M}$ and the function values $\mathbf{f}$.

\subsection{GMLS regression of functions on manifolds}

\label{sec:2_1} We first briefly review the basic idea of the GMLS approach,
which utilizes a set of local polynomials to approximate functions on
manifolds \cite{liang2013solving,gross2020meshfree,jiang2024generalized}.
For an arbitrary base point $\mathbf{x}_{0}\in \mathbf{X}_{M}\subset M$, its
$K$-nearest neighbors in $\mathbf{X}_{M}$ are denoted by $S_{\mathbf{x}%
_{0}}=\left\{ \mathbf{x}_{0,k}\right\} _{k=1}^{K}\subset \mathbf{X}_{M}$.
The set $S_{\mathbf{x}_{0}}$ is referred to as a "stencil" (see e.g., \cite%
{flyer2016role,lehto2017radial}). By definition,
 $\mathbf{x}_{0,1}=\mathbf{x}%
_{0}$ is the base point. Let $\mathbb{P}_{\mathbf{x}_{0}}^{l,d}$ be the
space of local \textquotedblleft intrinsic\textquotedblright\ polynomials
with degree up to $l$ in $d$ variables at the base point ${\mathbf{x}_{0}}$,
i.e.,
\begin{equation}
\mathbb{P}_{\mathbf{x}_{0}}^{l,d}=\mathrm{span}\Big\{p_{\boldsymbol{\alpha }%
}\left( \boldsymbol{\theta }\left( \mathbf{x}\right) \right)
:=\prod\limits_{i=1}^{d}\theta _{i}^{\alpha _{i}}\left( \mathbf{x}\right) %
\Big|\ 0\leq \left\vert \boldsymbol{\alpha }\right\vert \leq l\ \Big\},
\label{eq:thet12}
\end{equation}%
where $\boldsymbol{\theta }\left( \mathbf{x}\right) $ is a projected
coordinate system of the local\ tangent space $T_{\mathbf{x}_{0}}M$ with
\begin{equation}
\boldsymbol{\theta }\left( \mathbf{x}\right) =(\theta _{1}(\mathbf{x}%
),\ldots ,\theta _{d}(\mathbf{x})),\text{ \ \ }\theta _{i}(\mathbf{x})=%
\boldsymbol{t}_{i}({\mathbf{x}_{0}})\cdot (\mathbf{x}-{\mathbf{x}_{0}}),%
\text{ \ \ }i=1,\ldots ,d.  \label{eqn:Thex}
\end{equation}%
Here, $\boldsymbol{\alpha }=(\alpha _{1},\alpha _{2},\ldots ,\alpha _{d})\in
\mathbb{N}^{d}$ is the multi-index notation with $\left\vert \boldsymbol{%
\alpha }\right\vert =\alpha _{1}+\cdots +\alpha _{d}$. By definition, the
dimension of the space $\mathbb{P}_{\mathbf{x}_{0}}^{l,d}$ is $m=\left(
\begin{array}{c}
l+d \\
d%
\end{array}%
\right) $.

Given $K>m$ and $\mathbf{f}_{{\mathbf{x}_{0}}}:=(f({\mathbf{x}_{0,1}}),...,f(%
\mathbf{x}_{0,K}))^{\top }$, we can define an operator $\mathcal{I}_{p}:%
\mathbf{f}_{{\mathbf{x}_{0}}}\in \mathbb{R}^{K}\rightarrow \mathcal{I}_{p}%
\mathbf{f}_{{\mathbf{x}_{0}}}\in \mathbb{P}_{\mathbf{x}_{0}}^{l,d}$ such
that $\mathcal{I}_{p}\mathbf{f}_{{\mathbf{x}_{0}}}$ is the optimal solution
of the following moving least-squares problem:
\begin{equation}
\underset{\hat{f}\in \mathbb{P}_{\mathbf{x}_{0}}^{l,d}}{\min }\frac{1}{2}%
\sum_{i,j=1}^{K}\lambda _{ij}\left( f({\mathbf{x}_{0,i}})-\hat{f}({\mathbf{x}%
_{0,i}})\right) \left( f({\mathbf{x}_{0,j}})-\hat{f}({\mathbf{x}_{0,j}}%
)\right) ,  \label{eqn:int_LS}
\end{equation}%
where $\left[ \lambda _{ij}\right] _{i,j=1}^{K}$ is a symmetric
positive-definite matrix. In most literature, $\left[ \lambda _{ij}\right] _{i,j=1}^{K}$ is taken to be a diagonal
matrix in which the diagonal element $\lambda _{kk}$ is a positive weight
function of the distance $\Vert {\mathbf{x}_{0,k}-\mathbf{x}_{0}}%
\Vert$ (see e.g., \cite%
{Wendland2005Scat,lipman2009stable,liang2013solving,gross2020meshfree,jones2023generalized,li2024generalized}). GMLS is a regression technique for approximating functions by
solving the above least-squares problem with a weighted $\ell^{2}$-norm. The
solution to the least-squares problem (\ref{eqn:int_LS}) can be represented
as
\begin{equation}
\hat{f}(\mathbf{x}):=(\mathcal{I}_{p}\mathbf{f}_{{\mathbf{x}_{0}}})(\mathbf{x%
})=\sum_{j=1}^{m}b_{j}p_{\boldsymbol{\alpha }(j)}\left( \boldsymbol{\theta }%
\left( \mathbf{x}\right) \right) ,\text{ \ \ }\mathbf{x}\in M,
\label{eqn:gmls}
\end{equation}%
where each $b_{j}$ is the coefficient associated with the corresponding $j$th
polynomial basis function $p_{\boldsymbol{\alpha }(j)}\left( \boldsymbol{\theta}(\mathbf{x})%
\right) $. Denote the concatenated coefficients $\mathbf{b}%
=(b_{1},...,b_{m})^{\top }$. Then the moving least-squares problem (\ref%
{eqn:int_LS})\ can be formulated in matrix-vector form as
\begin{equation}
\underset{{\mathbf{b\in }}\mathbb{R}^{m}}{\min }\frac{1}{2}\left( \mathbf{f}%
_{\mathbf{x}_{0}}-\boldsymbol{P\mathbf{b}}\right) ^{\top }\boldsymbol{%
\Lambda }\left( \mathbf{f}_{\mathbf{x}_{0}}-\boldsymbol{P\mathbf{b}}\right) ,
\label{eqn:LS_optm}
\end{equation}%
where $\boldsymbol{\Lambda }=\left[ \lambda _{ij}\right] _{i,j=1}^{K}$ is
the $K\times K$\ positive-definite matrix and $\boldsymbol{P}$ is the $%
K\times m$\ Vandermonde-type matrix with components $\boldsymbol{P}_{kj}=p_{%
\boldsymbol{\alpha }(j)}\left( \boldsymbol{\theta }\left( \mathbf{x}{_{0,k}}%
\right) \right) $ for $1\leq k\leq K,1\leq j\leq m$. Thus, these expansion
coefficients $\mathbf{b}=(b_{1},...,b_{m})^{\top }$\ satisfy the normal
equation,
\begin{equation}
(\boldsymbol{P}^{\top }\boldsymbol{\Lambda }\boldsymbol{P})\mathbf{b}=%
\boldsymbol{P}^{\top }\boldsymbol{\Lambda }\mathbf{f}_{{\mathbf{x}_{0}}}.
\label{eqn:Pij}
\end{equation}%
The above normal equation  admits a unique solution if $\boldsymbol{P}$  has full column rank  ($\mathrm{rank}(\boldsymbol{P})=m$).



\subsection{Local RBF-FD approximation of functions on manifolds using the
tangent plane method}

\label{sec:PRFD}

%

In most previous studies on RBF-FD methods, the RBF interpolant
was defined over the Euclidean radial distance in  $\mathbb{R}^n$
to approximate differential operators on manifolds (see e.g., \cite%
{shankar2015radial,lehto2017radial,petras2018rbf, alvarez2021local}).
Recently, in order to approximate the Laplace-Beltrami operator on manifolds, authors
of \cite{shaw2019radial,wright2023mgm,jones2023generalized} first consider a PHS+Poly interpolant of the function $f$ to
the scattered data in the stencil $S_{\mathbf{x}_0}$ using the so-called tangent-plane RBF-FD method as follows:%
\begin{equation}
(\mathcal{I}_{\phi p}\mathbf{f}_{{\mathbf{x}_{0}}})\left( \mathbf{x}\right)
:=\sum_{k=1}^{K}a_{k}\phi \left( \Vert \boldsymbol{\theta }\left( \mathbf{x}%
\right) -\boldsymbol{\theta }\left( \mathbf{x}_{0,k}\right) \Vert \right)
+\sum_{j=1}^{m}b_{j}p_{\boldsymbol{\alpha }(j)}\left( \boldsymbol{\theta }%
\left( \mathbf{x}\right) \right) ,\ \ \ \mathbf{x}\in M,  \label{eqn:If2}
\end{equation}%
where $\phi $ is the radial function, $%
\Vert \cdot \Vert $ is the standard Euclidean norm in $\mathbb{R}^{n}$ and $%
p_{\boldsymbol{\alpha }(j)}$ are multivariate polynomials. Here, both the PHS radial function $\phi $  and the
polynomial $p_{\boldsymbol{\alpha }(j)}$ (equation (\ref{eq:thet12})) are defined
over the projected tangent-plane coordinate $\boldsymbol{\theta }\left(
\mathbf{x}\right) $ in (\ref{eqn:Thex}).  Notice that this
is similar to GMLS on manifolds, in which the polynomials are also defined
over the projected coordinate $\boldsymbol{\theta }\left( \mathbf{x}\right) $
as seen in (\ref{eqn:gmls}). By definition, we have $\boldsymbol{\theta }\left(
\mathbf{x}_{0,1}\right) =\boldsymbol{\theta }\left( \mathbf{x}_{0}\right) =%
\mathbf{0}$ to be center of the projected coordinate system. One
significant advantage of the tangent-plane method, as noted in \cite%
{jones2023generalized}, is that it defines interpolation functions locally
on the tangent space of the manifold, which can simplify the calculation of
their derivatives. We will discuss this advantage later in Remark \ref{rem:35}. 

Common choices of the radial function $\phi $ include Polyharmonic spline
(PHS) function \cite{flyer2016role,jones2023generalized}, Gaussian function
\cite{fasshauer2012stable}, inverse quadratic function \cite%
{harlim2023radial}, inverse multiquadric function \cite{fuselier2013high}
and Mat\'{e}rn class function \cite{fuselier2013high}. For most bell-shaped
positive-definite RBFs, including the Gaussian and Mat\'{e}rn classes, they
contain a shape parameter that controls the flatness of the RBF. In general,
the shape parameter needs appropriate tuning in a certain range;
specifically, a small value achieves better accuracy but also results in an
ill-conditioned interpolation matrix. When the shape parameter is
properly chosen, the local RBF-FD can achieve the convergence results (see
e.g., \cite{shankar2015radial,lehto2017radial,petras2018rbf, alvarez2021local}%
).

On the other hand, the $U$-shaped  PHS function is attractive due to its
absence of the shape parameter. For RBF-FD methods, the PHS
function has been used together with polynomials in many studies (see e.g.,
\cite{flyer2016role,bayona2017role,bayona2019role,jones2023generalized}). In
this article, we  focus on the following shape-parameter-free PHS
function,
\begin{equation}
\text{Polyharmonic spline (PHS)}:\qquad \phi (r)=r^{2\kappa +1},
\label{eqn:phs}
\end{equation}%
where the variable $r$ is the radial distance and the PHS parameter $%
\kappa $  $\in \mathbb{N}$ controls the smoothness of $\phi $.
 While some theories suggest using  $\phi (r)=r^{2\kappa +1}$ in odd dimension, and
$\phi (r)=r^{2\kappa}\log{r}$ in even dimension, that distinction may not be necessary \cite{flyer2016role}.
In the RBF-FD interpolant (\ref%
{eqn:If2}), the radial distance $r$ is taken to be the
 norm in $\mathbb{R}^{d}$, i.e., $\ r=\Vert \boldsymbol{\theta}(\mathbf{x})-\boldsymbol{\theta}({\mathbf{x}_{0,k}})\Vert $.


Given data $f(\mathbf{x}_{0,k})$\ at node $\mathbf{x}_{0,k}\ (k=1,\ldots ,K)$%
,\ the expansion coefficients are obtained by enforcing $K$ interpolation conditions together with $m$ additional moment conditions ensuring uniqueness:%
\begin{equation}
\begin{array}{ll}
(\mathcal{I}_{\phi p}\mathbf{f}_{{\mathbf{x}_{0}}})\left( \mathbf{x}%
_{0,k}\right) =f(\mathbf{x}_{0,k}), & \text{for}\ k=1,\ldots ,K, \\
\sum_{k=1}^{K}a_{k}p_{\boldsymbol{\alpha }(j)}\left( \boldsymbol{\theta }%
\left( \mathbf{x}_{0,k}\right) \right) =0, & \text{for}\ j=1,\cdots ,m.%
\end{array}
\label{eqn:mcond}
\end{equation}%
Denote the coefficients as $\mathbf{a}=\left[ a_{1},a_{2},\cdots ,a_{K}%
\right] ^{\top }$\ and $\mathbf{b}=\left[ b_{1},b_{2},\cdots ,b_{m}\right]
^{\top }$. In matrix form,  the $K+m$ conditions above can be written as%
\begin{equation}
\begin{bmatrix}
\boldsymbol{\Phi } & \boldsymbol{P} \\
\boldsymbol{P}^{\top } & \mathbf{0}%
\end{bmatrix}%
\begin{bmatrix}
\mathbf{a} \\
\mathbf{b}%
\end{bmatrix}%
=%
\begin{bmatrix}
\mathbf{f}_{\mathbf{x}_{0}} \\
\mathbf{0}%
\end{bmatrix}%
,  \label{eqn:mat_PRFD}
\end{equation}%
where $\boldsymbol{P}\in \mathbb{R}^{K\times m}$ is defined in (\ref{eqn:Pij}%
) and the symmetric $\boldsymbol{\Phi }\in \mathbb{R}^{K\times K}$ is
defined as
\begin{equation}
\boldsymbol{\Phi }=%
\begin{bmatrix}
\phi \left( \Vert \boldsymbol{\theta }\left( \mathbf{x}_{0,1}\right) -%
\boldsymbol{\theta }\left( \mathbf{x}_{0,1}\right) \Vert \right) & \phi
\left( \Vert \boldsymbol{\theta }\left( \mathbf{x}_{0,1}\right) -\boldsymbol{%
\theta }\left( \mathbf{x}_{0,2}\right) \Vert \right) & \cdots & \phi \left(
\Vert \boldsymbol{\theta }\left( \mathbf{x}_{0,1}\right) -\boldsymbol{\theta
}\left( \mathbf{x}_{0,K}\right) \Vert \right) \\
\phi \left( \Vert \boldsymbol{\theta }\left( \mathbf{x}_{0,2}\right) -%
\boldsymbol{\theta }\left( \mathbf{x}_{0,1}\right) \Vert \right) & \phi
\left( \Vert \boldsymbol{\theta }\left( \mathbf{x}_{0,2}\right) -\boldsymbol{%
\theta }\left( \mathbf{x}_{0,2}\right) \Vert \right) & \cdots & \phi \left(
\Vert \boldsymbol{\theta }\left( \mathbf{x}_{0,2}\right) -\boldsymbol{\theta
}\left( \mathbf{x}_{0,K}\right) \Vert \right) \\
\vdots & \vdots & \ddots & \vdots \\
\phi \left( \Vert \boldsymbol{\theta }\left( \mathbf{x}_{0,K}\right) -%
\boldsymbol{\theta }\left( \mathbf{x}_{0,1}\right) \Vert \right) & \phi
\left( \Vert \boldsymbol{\theta }\left( \mathbf{x}_{0,K}\right) -\boldsymbol{%
\theta }\left( \mathbf{x}_{0,2}\right) \Vert \right) & \cdots & \phi \left(
\Vert \boldsymbol{\theta }\left( \mathbf{x}_{0,K}\right) -\boldsymbol{\theta
}\left( \mathbf{x}_{0,K}\right) \Vert \right)%
\end{bmatrix}%
.  \label{eqn:Phi}
\end{equation}
For the well-posed linear system (\ref{eqn:mat_PRFD}), its solution can be
obtained using the Schur complement of $\boldsymbol{\Phi }$\ by
\begin{equation}
\begin{array}{l}
\mathbf{a}=\boldsymbol{\Phi }^{-1}\left( \mathbf{I}-\boldsymbol{P}\left(
\boldsymbol{P}^{\top }\boldsymbol{\Phi }^{-1}\boldsymbol{P}\right) ^{-1}%
\boldsymbol{P}^{\top }\boldsymbol{\Phi }^{-1}\right) \mathbf{f}_{\mathbf{x}%
_{0}} \\
\mathbf{b}=\left( \boldsymbol{P}^{\top }\boldsymbol{\Phi }^{-1}\boldsymbol{P}%
\right) ^{-1}\boldsymbol{P}^{\top }\boldsymbol{\Phi }^{-1}\mathbf{f}_{%
\mathbf{x}_{0}}.%
\end{array}
\label{eqn:grbf1}
\end{equation}

For the PHS $\phi (r)=r^{2\kappa +1}$, the chosen parameter $\kappa $ should
be less than or equal to the degree of the polynomial $l$  as defined in (\ref%
{eq:thet12}). Under this constraint on $\kappa $, it can be shown that $%
\boldsymbol{\Phi }$ is conditionally positive definite on the subspace
satisfying the $m$ moment conditions in (\ref{eqn:mcond}) \cite{iske2003approximation,Wendland2005Scat,jones2023generalized}.
If the stencil points satisfy $\mathrm{rank}(\boldsymbol{P}) = m$, then the system (\ref{eqn:mat_PRFD}) is non-singular
 and the PHS+Poly interpolant is well-posed.

\section{Two-step Generalized RBF-FD method}

\label{sec:gRBF-FD}

Motivated by the tangent-plane RBF-FD method in \cite{jones2023generalized},
we develop the generalized RBF-FD (gRBF-FD) method for solving PDEs on
manifolds, identified by randomly sampled data. For gRBF-FD method, the
local approximation space matches that of the tangent-plane RBF-FD but the
 distinction lies in how the expansion coefficients are calculated. That is,
we still consider using the PHS+Poly interpolant (\ref{eqn:If2}), in which
both the PHS function and the polynomials are defined over the projected
coordinate system. It is natural, as the geometry on the manifold $M$ can
be pulled back to define these functions on the projected tangent space,
which is spanned by the local $d$-dimensional coordinates $\boldsymbol{%
\theta }=(\theta _{1},\ldots ,\theta _{d})$. We next introduce how we
approximate the expansion coefficients $\left\{ a_{k}\right\} _{k=1}^{K}$\
and $\left\{ b_{j}\right\} _{j=1}^{m}$ in (\ref{eqn:If2}).

\subsection{Approximation of the expansion coefficients in gRBF-FD method}\label{sec:3.1}

The gRBF-FD method employs a two-step approximation for the target function $%
f$. For the first step, we apply the GMLS regression technique to
approximate the target function $f$. Then the coefficients $\mathbf{b}$ can
be solved from the normal equation in (\ref{eqn:Pij}),
\begin{equation}
\mathbf{b}=\left( \boldsymbol{P}^{\top }\boldsymbol{\Lambda P}\right) ^{-1}%
\boldsymbol{P}^{\top }\boldsymbol{\Lambda }\mathbf{f}_{\mathbf{x}_{0}}.
\label{eqn:b}
\end{equation}%
This provides a first-step approximation to the function $f$\ and then the
residual of the GMLS approximation becomes%
\begin{equation*}
\mathbf{s}_{{\mathbf{x}_{0}}}:=\mathbf{f}_{{\mathbf{x}_{0}}}-(\mathcal{I}_{p}%
\mathbf{f}_{{\mathbf{x}_{0}}})|_{S_{\mathbf{x}_{0}}}=\mathbf{f}_{{\mathbf{x}%
_{0}}}-\boldsymbol{P}\mathbf{b},
\end{equation*}%
where the operator $(\mathcal{I}_{p}\mathbf{f}_{{\mathbf{x}_{0}}})(\mathbf{x}%
)=\sum_{j=1}^{m}b_{j}p_{\boldsymbol{\alpha }(j)}\left( \boldsymbol{\theta }%
\left( \mathbf{x}\right) \right) $ is defined in (\ref{eqn:gmls}).

For the second step, we fit the above residual $\mathbf{s}_{{\mathbf{x}_{0}}}$ using the PHS interpolant, that is,
\begin{equation}
\left( \mathcal{I}_{\phi }\mathbf{s}_{\mathbf{x}_{0}}\right) \left( \mathbf{x%
}\right) =\sum_{k=1}^{K}a_{k}\phi \left( \Vert \boldsymbol{\theta }\left(
\mathbf{x}\right) -\boldsymbol{\theta }\left( \mathbf{x}_{0,k}\right) \Vert
\right) .  \label{eqn:Ipph}
\end{equation}%
The coefficients $\mathbf{a}=\left[ a_{1},a_{2},\cdots ,a_{K}\right] ^{\top
} $ are determined by the interpolation condition,
\begin{equation}
\boldsymbol{\Phi }\mathbf{a=s}_{{\mathbf{x}_{0}}}=\mathbf{f}_{{\mathbf{x}_{0}%
}}-\boldsymbol{P}\mathbf{b=f}_{{\mathbf{x}_{0}}}-\boldsymbol{P}\left(
\boldsymbol{P}^{\top }\boldsymbol{\Lambda P}\right) ^{-1}\boldsymbol{P}%
^{\top }\boldsymbol{\Lambda }\mathbf{f}_{\mathbf{x}_{0}},  \label{eqn:PHisa}
\end{equation}%
where $\boldsymbol{\Phi }$ is given by (\ref{eqn:Phi}). Then we can compute
the coefficients $\mathbf{a}$ as
\begin{equation}
\mathbf{a}=\boldsymbol{\Phi }^{-1}\left( \mathbf{f}_{{\mathbf{x}_{0}}}-%
\boldsymbol{P}\mathbf{b}\right) =\boldsymbol{\Phi }^{-1}\left( \mathbf{I}-%
\boldsymbol{P}\left( \boldsymbol{P}^{\top }\boldsymbol{\Lambda P}\right)
^{-1}\boldsymbol{P}^{\top }\boldsymbol{\Lambda }\right) \mathbf{f}_{\mathbf{x%
}_{0}}.  \label{eqn:a}
\end{equation}



The basic idea of gRBF-FD is to regress the function using GMLS followed
by fitting the residual $\mathbf{s}_{{\mathbf{x}_{0}}}$ using the PHS
interpolant. Intuitively,  GMLS captures the smooth, leading
component of the Taylor expansion of $f$, while PHS compensates for the
expansion's remainder contained within the residual. As a result, we arrive
at a PHS+Poly interpolant for the gRBF-FD method, 
\begin{equation}
(\mathcal{I}_{\phi p}\mathbf{f}_{{\mathbf{x}_{0}}})\left( \mathbf{x}\right)
=(\mathcal{I}_{\phi }\mathbf{s}_{{\mathbf{x}_{0}}})\left( \mathbf{x}\right)
+\left( \mathcal{I}_{p}\mathbf{f}_{\mathbf{x}_{0}}\right) \left( \mathbf{x}%
\right) :=\sum_{k=1}^{K}a_{k}\phi \left( \Vert \boldsymbol{\theta }\left(
\mathbf{x}\right) -\boldsymbol{\theta }\left( \mathbf{x}_{0,k}\right) \Vert
\right) +\sum_{j=1}^{m}b_{j}p_{\boldsymbol{\alpha }(j)}\left( \boldsymbol{%
\theta }\left( \mathbf{x}\right) \right) ,\text{ \ }\mathbf{x}\in M,
\label{eqn:GPRFD_inerpolant}
\end{equation}%
where the coefficients $\mathbf{a}$ and $\mathbf{b}$ are determined by (\ref%
{eqn:a}) and (\ref{eqn:b}), respectively. We note that this interpolant
formulation is exactly the same as the one in (\ref{eqn:If2}) for RBF-FD
while the only difference comes from the evaluation of the coefficients $%
\mathbf{a}$ and $\mathbf{b}$.




\begin{rem}
The relationship between RBF-FD and gRBF-FD can be revealed by a direct comparison of the coefficients in (\ref{eqn:grbf1}) with those in (\ref{eqn:a}) and (\ref{eqn:b}).
Specifically, the coefficients $\mathbf{a}$ and $\mathbf{b}$ are
identical for both methods if the weight matrix in gRBF-FD is chosen as $\boldsymbol{%
\Lambda }=\boldsymbol{\Phi }^{-1}$.
We therefore refer to this more general formulation as the generalized RBF-FD (gRBF-FD) method.
Moreover,
for RBF-FD, those moment conditions in (\ref{eqn:mcond}) admit a
natural interpretation. By applying the RBF interpolant to the
residual, we arrive at the expression, $\boldsymbol{\Phi }\mathbf{a=f}_{{%
\mathbf{x}_{0}}}-\boldsymbol{P}\left( \boldsymbol{P}^{\top }\boldsymbol{\Phi
}^{-1}\boldsymbol{P}\right) ^{-1}\boldsymbol{P}^{\top }\boldsymbol{\Phi }%
^{-1}\mathbf{f}_{\mathbf{x}_{0}}$. This follows immediately from equation (%
\ref{eqn:PHisa}) by selecting the weight $\boldsymbol{\Lambda }=\boldsymbol{%
\Phi }^{-1}$. Left multiplying both sides of this equation by $\boldsymbol{P}%
^{\top }\boldsymbol{\Phi }^{-1}$ causes the right-hand-side term to vanish,
yielding the moment conditions: $\boldsymbol{P}^{\top }\boldsymbol{\Phi }%
^{-1}\boldsymbol{\Phi }\mathbf{a}=\boldsymbol{P}^{\top }\mathbf{a}=\mathbf{0}
$.
\end{rem}

\begin{rem}
\label{rem:1oK} There are many choices of the weight $\boldsymbol{\Lambda }%
=\left\{ \lambda _{ij}\right\} _{i,j=1}^{K}$ in the moving least-squares (%
\ref{eqn:LS_optm}) as shown in {\cite%
{Wendland2005Scat,liang2013solving,gross2020meshfree}}. As reported in
previous works {\cite%
{liang2012geometric,liang2013solving,jiang2024generalized,li2024generalized}}, the diagonal
weight function,
\begin{equation}
\boldsymbol{\Lambda }_{K} = \{\lambda_{ij}\}_{i,j=1}^K,\qquad \textrm{with diagonal entries\ }
\lambda _{kk}=\left\{
\begin{array}{ll}
1, & \text{if }k=1 \\
1/K, & \text{if }k=2,\ldots ,K%
\end{array}%
\right.,  \label{eqn:k_weight}
\end{equation}%
can numerically provide a stable approximation to the Laplace-Beltrami
operator across various data sets. In the following, this weight is referred
to as the diagonal $1/K$ weight function. For the gRBF-FD method, we will focus
on this special weight function and compare its performance with other
weight functions (see Figs. \ref{fig:PRFD_1} and \ref{fig:PRFD_2} in the
following Section \ref{sec:wtrg} for the numerical performance among
different choices of $\boldsymbol{\Lambda }$).
\end{rem}



\subsection{Monge parametrization}\label{sec:3.2}

To facilitate the computation of local geometry, we employ a Monge
parametrization of the local manifold near the base point $\mathbf{x}_{0}$
\cite%
{monge1809application,pressley2010elementary,liang2013solving,jones2023generalized,gross2020meshfree}%
. A local coordinate chart for the manifold near the base point is
defined using the embedding map $\boldsymbol{q}$,
\begin{equation}
\mathbf{x}=\boldsymbol{q}\left( \theta _{1},\ldots ,\theta _{d};\mathbf{x}%
_{0}\right) =\mathbf{x}_{0}+\sum_{i=1}^{d}\theta _{i}\boldsymbol{t}%
_{i}\left( \mathbf{x}_{0}\right) +\sum_{s=1}^{n-d}q_{s}\left( \theta
_{1},\ldots ,\theta _{d}\right) \boldsymbol{n}_{s}\left( \mathbf{x}%
_{0}\right) ,  \label{eqn:q0}
\end{equation}%
where $\left( \theta _{1},\ldots ,\theta _{d}\right) $ are the local
coordinates as defined in (\ref{eqn:Thex}), $q_{s}\left( \theta _{1},\ldots
,\theta _{d}\right) =\boldsymbol{n}_{s}\left( \mathbf{x}_{0}\right) \cdot (%
\mathbf{x}-\mathbf{x}_{0})$ ($s=1,\ldots ,n-d$), and $\{\boldsymbol{n}%
_{s}\}_{s=1}^{n-d}$\ are the $n-d$ orthonormal vectors that are orthogonal
to the tangent space $T_{\mathbf{x}_{0}}M$. This is known as the Monge
parametrization of manifold.  This representation allows us to express differential operators in terms of the local coordinates $\left( \theta _{1},\ldots ,\theta _{d}\right) $, which is essential for computing the Laplace–Beltrami operator in our mesh-free framework.

Indeed, all above multivariate polynomials and
polyharmonic spline functions in RBF-FD (\ref{eqn:If2}) and gRBF-FD (\ref%
{eqn:GPRFD_inerpolant}) are defined over the tangent plane and have been
formulated as their local representations in local Monge coordinates $\left( \theta _{1},\ldots ,\theta _{d}\right) $.
That is, the local representation of $(\mathcal{I}_{\phi p}\mathbf{f}_{{\mathbf{x}_{0}}})$
is $(\mathcal{I}_{\phi p}\mathbf{f}_{{\mathbf{x}_{0}}}) \circ \boldsymbol{\theta }^{-1}:\boldsymbol{\theta }\left(
U_{0}\right) \rightarrow \mathbb{R}$ for some open set $U_0 \subset M$, and the interpolant $\sum_{k=1}^{K}a_{k}\phi \left( \cdot
\right) +\sum_{j=1}^{m}b_{j}p_{\boldsymbol{\alpha }(j)}\left( \cdot \right) $ is approximating the local representation of the target function $f$, which is defined over $\boldsymbol{\theta }\left(U_{0}\right)$.
This is natural in differential geometry because differentiating a function on a manifold requires pulling it back to a local coordinate chart and then computing the derivative with respect to those local coordinates.


Then the tangent vector bases for $\mathbf{x}\in S_{\mathbf{x}_{0}}$ in the
Monge coordinate system can be calculated as
\begin{equation}
\frac{\partial \boldsymbol{q}}{\partial \theta _{k}}=\boldsymbol{t}%
_{k}\left( \mathbf{x}_{0}\right) +\sum_{s=1}^{n-d}\frac{\partial q_{s}}{%
\partial \theta _{k}}\boldsymbol{n}_{s}\left( \mathbf{x}_{0}\right) ,\quad
\text{for }k=1,\ldots ,d.  \label{eqn:dq}
\end{equation}%
Each component $g_{ij}$\ of the local Riemannian metric $g=\left(
g_{ij}\right) _{i,j=1}^{d}$ for $\mathbf{x}\in S_{\mathbf{x}_{0}}$ becomes
\begin{equation}
g_{ij}=\left\langle \frac{\partial \boldsymbol{q}}{\partial \theta _{i}},%
\frac{\partial \boldsymbol{q}}{\partial \theta _{j}}\right\rangle =\delta
_{ij}+\sum_{s=1}^{n-d}\frac{\partial q_{s}}{\partial \theta _{i}}\frac{%
\partial q_{s}}{\partial \theta _{j}},  \label{eqn:gij}
\end{equation}%
where $\delta _{ij}$\ is the Kronecker delta function. Let $\left(
g^{ij}\right) _{i,j=1}^{d}$ be the inverse of the metric tensor. Then we
have the following properties in a Monge coordinate system centered at $\mathbf{x}_{0}$:

\begin{prop}
\label{prop:norm} Let $\{U_{0};\theta _{1},\ldots ,\theta _{d}\}$ be a Monge
coordinate system centered at $\mathbf{x}_{0}$\ where $U_{0}\subset M$ is an
open neighborhood of $\mathbf{x}_{0}$ and $M$ is a $d$-dimensional manifold
with Riemannian metric $g$. Then\newline
\noindent (1) For all $1\leq i,j\leq d$, $g_{ij}(\mathbf{x}_{0})=g^{ij}(%
\mathbf{x}_{0})=\delta _{ij}$.\newline
\noindent (2) For all $1\leq i,j,k\leq d$, $\Gamma _{ij}^{k}(\mathbf{x}%
_{0})=0$.
\newline
\noindent (3) For all $1\leq i,j,k\leq d$, $\partial _{k}g_{ij}(\mathbf{x}%
_{0})=0$. 
\end{prop}

\begin{proof}
We notice that the embedding map (\ref{eqn:q0})\ restricted on the base
point should be $\mathbf{x}_{0}$, $\boldsymbol{q}|_{\mathbf{x}_{0}}=\mathbf{x%
}_{0}$, which implies that all  coordinates vanish at the base point,
that is, $\theta _{1}|_{\mathbf{x}_{0}}=\cdots =\theta _{d}|_{\mathbf{x}%
_{0}}=0$ and $q_{1}|_{\mathbf{x}_{0}}=\cdots =q_{n-d}|_{\mathbf{x}_{0}}=0$.
Also notice that the tangent vector bases (\ref{eqn:dq})\ should live in the
tangent space at the base point, which implies that the first-order
derivative$\frac{\partial q_{s}}{\partial \theta _{k}}|_{\mathbf{x}_{0}}=0$
for all $s=1,\ldots ,n-d$ and $k=1,\ldots ,d$. Then we obtain the property
for the metric $g_{ij}$:%
\begin{equation*}
g_{ij}(\mathbf{x}_{0})=\Big(\delta _{ij}+\sum_{s=1}^{n-d}\frac{\partial q_{s}%
}{\partial \theta _{i}}\frac{\partial q_{s}}{\partial \theta _{j}}\Big)%
\bigg\vert_{\mathbf{x}_{0}}=\delta _{ij},
\end{equation*}%
as well as for the inverse $g^{ij}(\mathbf{x}_{0})=\delta _{ij}$. The
derivative can be calculated as%
\begin{equation*}
\partial _{k}g_{ij}(\mathbf{x}_{0})=\Big(\sum_{s=1}^{n-d}\frac{\partial q_{s}%
}{\partial \theta _{i}\partial \theta _{k}}\frac{\partial q_{s}}{\partial
\theta _{j}}+\frac{\partial q_{s}}{\partial \theta _{i}}\frac{\partial q_{s}%
}{\partial \theta _{j}\partial \theta _{k}}\Big)\bigg\vert_{\mathbf{x}%
_{0}}=0.
\end{equation*}%
Thus, by definition, $\Gamma _{ij}^{k}=\frac{1}{2}g^{km}\left( \frac{%
\partial g_{jm}}{\partial \theta _{i}}+\frac{\partial g_{im}}{\partial
\theta _{j}}-\frac{\partial g_{ij}}{\partial \theta _{m}}\right) $,\ all the
Christoffel symbols vanish at the base point.
\end{proof}

In the next section, we will use these properties to simplify the calculation
of the Laplace-Beltrami operator.

\subsection{Approximation of the Laplace-Beltrami operator}

\label{sec:LBapp}

Our goal is to approximate the Laplace-Beltrami operator $\Delta _{M}$
acting on a function $f$ at the base point $\mathbf{x}_{0}$ by a linear
combination of the function values in the stencil $S_{\mathbf{x}_{0}}$,
i.e., find the coefficients $\left\{ w_{k}\right\} _{k=1}^{K}$ such that
\begin{equation}
\Delta _{M}f\left( \mathbf{x}_{0}\right) \approx \sum_{k=1}^{K}w_{k}f\left(
\mathbf{x}_{0,k}\right) .  \label{eqn:lincomb}
\end{equation}%
Arranging the coefficients at each point $\mathbf{x}_{0}\in \mathbf{X}%
_{M}\subset M$\ into each row of a sparse $N\times N$ matrix $\boldsymbol{L}_{\mathbf{X}_{M}}$, we can approximate $\Delta _{M}f$ at all points by $\boldsymbol{L}_{%
\mathbf{X}_M}\mathbf{f}$. We now provide the calculation details for the
coefficients $\left\{ w_{k}\right\} _{k=1}^{K}$.

The Laplace-Beltrami operator can be formulated in the local Monge
coordinate system as
\begin{equation*}
\Delta _{M}f=\sum_{i,j=1}^{d}\frac{1}{\sqrt{\left\vert g\right\vert }}\frac{%
\partial }{\partial \theta _{i}}\left( \sqrt{\left\vert g\right\vert }g^{ij}%
\frac{\partial f}{\partial \theta _{j}}\right) =\sum_{i,j=1}^{d}g^{ij}\left(
\frac{\partial ^{2}f}{\partial \theta _{i}\partial \theta _{j}}-\Gamma
_{ij}^{k}\frac{\partial f}{\partial \theta _{k}}\right) ,
\end{equation*}%
where $\left\vert g\right\vert =\det (g)$ is the determinant of the
Riemannian metric tensor and $d$ is the intrinsic dimension of manifold. The
two formulas are equivalent for the Laplace-Beltrami operator. Using
Proposition \ref{prop:norm}, the Laplace-Beltrami operator at the base point can be
simplified as
\begin{equation}
\Delta _{M}f\left( \mathbf{x}_{0}\right) =\sum_{i=1}^{d}\frac{\partial ^{2}}{%
\partial \theta _{i}^{2}}f\left( \mathbf{x}_{0}\right) :=\Delta _{%
\boldsymbol{\theta }}f\left( \mathbf{x}_{0}\right) .  \label{eqn:ddthe}
\end{equation}%
Here, we use $\Delta _{\boldsymbol{\theta }}$\ to denote the Laplace
operator in the local Monge coordinate chart, $\{U_{0};\boldsymbol{\theta }%
=(\theta _{1},\ldots ,\theta _{d})\},$ where $\boldsymbol{\theta }%
:U_{0}\subset M\rightarrow V_{0}\subset \mathbb{R}^{d}$ is a diffeomorphism
onto an open subset of $\mathbb{R}^{d}$. The local representation of $f$\
becomes $f\circ \boldsymbol{\theta }^{-1}:\boldsymbol{\theta }\left(
U_{0}\right) \rightarrow \mathbb{R}$. Indeed, for gRBF-FD, the linear
combination of multivariate polynomials and PHS functions in the interpolant
(\ref{eqn:GPRFD_inerpolant}) approximates the local representation $f\circ
\boldsymbol{\theta }^{-1}$ in the local Monge coordinate chart (on the
tangent plane).

Then the Laplace-Beltrami operator at the base point can be approximated
using the gRBF-FD method,
\begin{equation}
\Delta _{M}f\left( \mathbf{x}_{0}\right) \approx \Delta _{\boldsymbol{\theta
}}(\mathcal{I}_{\phi p}\mathbf{f}_{{\mathbf{x}_{0}}})\left( \mathbf{x}%
_{0}\right) =\left. \left( \sum_{k=1}^{K}a_{k}\Delta _{\boldsymbol{\theta }%
}\phi \left( \Vert \boldsymbol{\theta }\left( \mathbf{x}\right) -\boldsymbol{%
\theta }\left( \mathbf{x}_{0,k}\right) \Vert \right)
+\sum_{s=1}^{m}b_{s}\Delta _{\boldsymbol{\theta }}p_{\boldsymbol{\alpha }%
(s)}\left( \boldsymbol{\theta }\left( \mathbf{x}\right) \right) \right)
\right\vert _{\mathbf{x}=\mathbf{x}_{0}},  \label{eqn:LB_intr}
\end{equation}%
where (\ref{eqn:GPRFD_inerpolant}) has been used. Here, the first term in (%
\ref{eqn:LB_intr}) can be calculated as
\begin{eqnarray}
\Delta _{\boldsymbol{\theta }}\phi \left( \Vert \boldsymbol{\theta }\left(
\mathbf{x}\right) -\boldsymbol{\theta }\left( \mathbf{x}_{0,k}\right) \Vert
\right) \big|_{\mathbf{x}=\mathbf{x}_{0}} &=&\sum_{i=1}^{d}\frac{\partial }{%
\partial \theta _{i}}\left[ \left( 2\kappa +1\right) \Vert \boldsymbol{%
\theta }\left( \mathbf{x}\right) -\boldsymbol{\theta }\left( \mathbf{x}%
_{0,k}\right) \Vert ^{2\kappa -1}\left( \theta _{i}\left( \mathbf{x}\right)
-\theta _{i}\left( \mathbf{x}_{0,k}\right) \right) \right]  \notag \\
&=&\left( 4\kappa ^{2}+2d\kappa +d-1\right) \Vert \boldsymbol{\theta }\left(
\mathbf{x}_{0}\right) -\boldsymbol{\theta }\left( \mathbf{x}_{0,k}\right)
\Vert ^{2\kappa -1},  \label{eqn:Lphi}
\end{eqnarray}%
where the PHS $\phi \left( \Vert \boldsymbol{\theta }\left( \mathbf{x}%
\right) -\boldsymbol{\theta }\left( \mathbf{x}_{0,k}\right) \Vert \right)
=\Vert \boldsymbol{\theta }\left( \mathbf{x}\right) -\boldsymbol{\theta }%
\left( \mathbf{x}_{0,k}\right) \Vert ^{2\kappa +1}$ has been used. The
second term in (\ref{eqn:LB_intr}) can be calculated as%
\begin{equation}
\Delta _{\boldsymbol{\theta }}p_{\boldsymbol{\alpha }(s)}\left( \boldsymbol{%
\theta }\left( \mathbf{x}\right) \right) \big|_{\mathbf{x}=\mathbf{x}%
_{0}}=\left\{
\begin{array}{ll}
2, & \boldsymbol{\alpha }(s)\in \mathbf{E,} \\
0, & \boldsymbol{\alpha }(s)\notin \mathbf{E,}%
\end{array}%
\right.  \label{eqn:2}
\end{equation}%
where the index set $\mathbf{E}=\left\{ \left( \alpha _{1},\cdots ,\alpha
_{d}\right) \in \mathbb{N}^{d}|\text{each }\alpha _{i}\in \left\{
0,2\right\} \text{ }(i=1,\ldots ,d)\text{ and }\left\vert \boldsymbol{\alpha
}\right\vert =\sum_{j=1}^{d}\alpha _{j}=2\right\} $. In fact, for the index
in $\mathbf{E}$, only one entry $\alpha _{i}=2$ for some $i\in \left\{
1,\ldots ,d\right\} $ and all the others are zero. Note that $%
\mathbf{E}$ contains $d$ number of indices. Substituting (\ref{eqn:Lphi})
and (\ref{eqn:2}) into (\ref{eqn:LB_intr}), we arrive at
\begin{eqnarray}
\Delta _{M}f\left( \mathbf{x}_{0}\right) &\approx &\sum_{k=1}^{K}a_{k}\left(
4\kappa ^{2}+2d\kappa +d-1\right) \Vert \boldsymbol{\theta }\left( \mathbf{x}%
_{0}\right) -\boldsymbol{\theta }\left( \mathbf{x}_{0,k}\right) \Vert
^{2\kappa -1}+\sum_{\left\{ s|\boldsymbol{\alpha }(s)\in \mathbf{E}\right\}
}2b_{s}  \notag \\
&:= &\left( \Delta _{\boldsymbol{\theta }}\boldsymbol{\phi }_{\mathbf{x}%
_{0}}\right) \mathbf{a}+\left( \Delta _{\boldsymbol{\theta }}\boldsymbol{p}_{%
\mathbf{x}_{0}}\right) \mathbf{b},  \label{eqn:GPRFD_LB}
\end{eqnarray}%
where $\Delta _{\boldsymbol{\theta }}\boldsymbol{\phi }_{\mathbf{x}_{0}}\in
\mathbb{R}^{1\times K}$ with the $k$th entry given by (\ref{eqn:Lphi}) and $%
\Delta _{\boldsymbol{\theta }}\boldsymbol{p}_{\mathbf{x}_{0}}\in \mathbb{R}%
^{1\times m}$ with $d$ number of 2 and $m-d$ number of 0. Substituting $%
\mathbf{a}$ (equation \eqref{eqn:a}) and $\mathbf{b}$\ (equation %
\eqref{eqn:b}) into above \eqref{eqn:GPRFD_LB}, we obtain the
coefficients in (\ref{eqn:lincomb}):
\begin{equation}
\Delta _{M}f\left( \mathbf{x}_{0}\right) \approx \sum_{k=1}^{K}w_{k}f\left(
\mathbf{x}_{0,k}\right) =\left( \left( \Delta _{\boldsymbol{\theta }}%
\boldsymbol{\phi }_{\mathbf{x}_{0}}\right) \boldsymbol{\Phi }^{-1}\big(%
\mathbf{I}-\boldsymbol{P}\left( \boldsymbol{P}^{\top }\boldsymbol{\Lambda P}%
\right) ^{-1}\boldsymbol{P}^{\top }\boldsymbol{\Lambda }\big)+\left( \Delta
_{\boldsymbol{\theta }}\boldsymbol{p}_{\mathbf{x}_{0}}\right) \left(
\boldsymbol{P}^{\top }\boldsymbol{\Lambda P}\right) ^{-1}\boldsymbol{P}%
^{\top }\boldsymbol{\Lambda }\right) \mathbf{f}_{\mathbf{x}_{0}}.
\label{eqn:LB_weight}
\end{equation}



\begin{rem}\label{rem:35}
Here we do not employ  PHS functions defined in Euclidean coordinates due to
additional computational cost required to evaluate their derivatives and the
Laplace-Beltrami operator. In particular, let $r_{\mathbf{x}%
,k}\left( \mathbf{x}\right) =\Vert \mathbf{x}-\mathbf{x}_{0,k}\Vert $ be the
radial distance function defined for $\mathbf{x}\in \mathbb{R}^{n}$ and then
we compute the first-order derivative of the PHS function as%
\begin{equation*}
\frac{\partial }{\partial \theta _{j}}\phi \left( \Vert \mathbf{x}-\mathbf{x}%
_{0,k}\Vert \right) =\frac{\partial }{\partial \theta _{j}}\left( r_{\mathbf{%
x},k}^{2}\left( \mathbf{x}\right) \right) ^{\frac{2\kappa +1}{2}}=\frac{%
2\kappa +1}{2}(r_{\mathbf{x},k}\left( \mathbf{x}\right) )^{2\kappa -1}\frac{%
\partial r_{\mathbf{x},k}^{2}\left( \mathbf{x}\right) }{\partial \theta _{j}}%
.
\end{equation*}%
Using the embedding map (\ref{eqn:q0}), we can calculate the derivative of
the square of the radial distance by%
\begin{eqnarray*}
\frac{\partial r_{\mathbf{x},k}^{2}\left( \mathbf{x}\right) }{\partial
\theta _{j}} &=&\frac{\partial }{\partial \theta _{j}}\left[ \sum_{i=1}^{d}%
\big(\theta _{i}-\theta _{i}(\mathbf{x}_{0,k})\big)^{2}+\sum_{s=1}^{n-d}\big(%
q_{s}-q_{s}(\mathbf{x}_{0,k})\big)^{2}\right] \\
&=&2\big(\theta _{j}-\theta _{j}(\mathbf{x}_{0,k})\big)+2\sum_{s=1}^{n-d}%
\big(q_{s}-q_{s}(\mathbf{x}_{0,k})\big)\frac{\partial q_{s}}{\partial \theta
_{j}},
\end{eqnarray*}%
where we see that the calculation involves many derivatives of normal
components $q_{s}$\ with respect to tangent components $\theta _{j}$ especially when
the ambient dimension $n$ is large. In contrast, as seen from the first line
of (\ref{eqn:Lphi}) for the PHS function defined over the tangent plane, the
first-order derivative does not contain any term of $\frac{\partial q_{s}}{%
\partial \theta _{j}}$ for $s=1,\ldots ,n-d$ and $j=1,\ldots ,d$, which  simplifies our computation.
\end{rem}

\section{Application to solving screened Poisson problems}

\label{sec:tunek}

\subsection{Screened Poisson problems on closed manifolds}

In this section, we will show an application of the proposed two-step
gRBF-FD discretization method to solve screened Poisson problems on closed
manifolds. For a closed manifold $M$, we consider the following screened Poisson
problem,
\begin{equation}
\left( a-\Delta _{M}\right) f=h,\quad \mathbf{x}\in M,  \label{eqn:poisson}
\end{equation}%
where $a>0$ and $h$ are defined such that the problem is well-posed. Here,
we set the parameter $a=1$ for simplicity. Numerically, we will approximate
the solution to the PDE problem (\ref{eqn:poisson}) pointwisely on the
scattered dataset $\mathbf{X}_{M}$, solving an $N\times N$ linear algebra
system%
\begin{equation}
\boldsymbol{L}_{\mathbf{X}_{M},\boldsymbol{I}}\mathbf{F}:=(\boldsymbol{I}-%
\boldsymbol{L}_{\mathbf{X}_{M}})\mathbf{F}=\mathbf{h},  \label{eqn:nums}
\end{equation}%
for $\mathbf{F}\in \mathbb{R}^{N\times 1}$ to approximate the true solution $%
\mathbf{f}:=(f(\mathbf{x}_{1}),\ldots ,f(\mathbf{x}_{N}))^{\top }\in \mathbb{%
R}^{N\times 1}$. Here, $\mathbf{h}:=(h(\mathbf{x}_{1}),\ldots ,h(\mathbf{x}%
_{N}))^{\top }\in \mathbb{R}^{N\times 1}$ and $\boldsymbol{L}_{\mathbf{X}%
_{M}}\in \mathbb{R}^{N\times N}$ is a sparse Laplacian matrix constructed
from a discretization scheme, such as GMLS and gRBF-FD. To verify
the convergence, we define the forward error (\textbf{FE}) for {the} operator
approximation and the inverse error (\textbf{IE}) for {the} solution approximation
as
\begin{equation}
\mathbf{FE}=\max_{1\leq i\leq N}|\Delta _{M}f\left( \mathbf{x}_{i}\right)
-\left( \boldsymbol{L}_{\mathbf{X}_{M}}\mathbf{f}\right)
_{i}|,\mathrm{\ \ \ \ }\mathbf{IE}=\max_{1\leq i\leq N}|\mathbf{F}_i-{f}(\mathbf{x}_i)%
|,  \notag
\end{equation}%
where $\boldsymbol{L}_{\mathbf{X}_{M}}\mathbf{f}\in \mathbb{R}%
^{N\times 1}$, $\left( \boldsymbol{L}_{\mathbf{X}_{M}}%
\mathbf{f}\right) _{i}$ is the $i$th component to approximate the value of $%
\Delta _{M}f$ at $\mathbf{x}_{i}$, and $\mathbf{F}_i$ is the $i$th component of the numerical solution $\mathbf{F}$.

In order to numerically verify the validity of operator approximation,
stability and convergence, we test on a  one-dimensional full ellipse
example in this section.

\begin{example}
Consider a 1D full ellipse in $\mathbb{R}^{2}$ parametrized by
\begin{equation}
\mathbf{x}:=\left( x^{1},x^{2}\right) =\left( \cos \theta ,2\sin \theta
\right) \in M\subset \mathbb{R}^{2},  \label{eqn:ellp}
\end{equation}%
defined with the Riemannian metric $g=\sin ^{2}\theta +4\cos ^{2}\theta$ for $0\leq \theta <2\pi $. The true solution is set to be $f=\sin
\theta \cos \theta $ and the right-hand-side $h:=\left( 1-\Delta _{M}\right)
f$ can be calculated. Then, we approximate the numerical solution $\mathbf{F}
$\ in (\ref{eqn:nums}) for the PDE problem in (\ref{eqn:poisson}), subjected
to the manufactured $h$. Numerically, the convergence is examined on the
ellipse using the number $N=\left[ 400,800,1600,3200,6400\right] $ of
dataset. Here, two types of dataset are considered for convergence
examination: \newline
\noindent \textbullet\ \textit{Well-sampled data}. The grid points are well-ordered and
all consecutive points have equal intrinsic distance in $\theta $ space (see
Fig.~\ref{fig:PRFD_1}(a)). \newline
\noindent \textbullet\ \textit{Randomly sampled data}. The data $\left\{ \theta
_{i}\right\} _{i=1}^{N}$ are  randomly generated from an independent and identical distribution
with density $p(\theta) = \theta/(4\pi^2) + 1/(4\pi)$ on $[0,2\pi )$ in the intrinsic coordinate and then mapped to the point
cloud $\left\{ \mathbf{x}_{i}\right\} _{i=1}^{N}$ via the parametrization (%
\ref{eqn:ellp}) (see Fig.~\ref{fig:PRFD_1}(d)).
\end{example}


\subsection{Normalization of Monge coordinates}

\label{sec:nmmg} In our numerical implementation, several technical details
need to be addressed to ensure the stability and convergence using the
gRBF-FD approach. First, we need to normalize the Monge coordinates since
the PHS function is  sensitive to changes in the magnitude of the input
radial distance.

\begin{definition}
\label{def:srad} For a stencil $S_{\mathbf{x}_{0}}$ centered at the base
point $\mathbf{x}_{0} $, its diameter and radius are defined as
\begin{equation*}
D_{K,\max }({\mathbf{x}_{0}}):=\max_{1\leq i,j\leq K}\Vert \boldsymbol{%
\theta }\left( \mathbf{x}_{0,i}\right) -\boldsymbol{\theta }\left( \mathbf{x}%
_{0,j}\right) \Vert ,\text{ \ \ }R_{K,\max }({\mathbf{x}_{0}}):=\max_{1\leq
k\leq K}\Vert \boldsymbol{\theta }\left( \mathbf{x}_{0,k}\right) -%
\boldsymbol{\theta }\left( \mathbf{x}_{0}\right) \Vert ,
\end{equation*}%
respectively, over the projected tangent space of $\mathbf{x}_{0}$.
\end{definition}

To compute multivariate polynomials and PHS functions, we normalize the
Monge coordinates $\boldsymbol{\theta }\left( \mathbf{x}\right) $ and the
radial distances $r_{\boldsymbol{\theta },k}\left( \mathbf{x}\right) :=\Vert
\boldsymbol{\theta }\left( \mathbf{x}\right) -\boldsymbol{\theta }\left(
\mathbf{x}_{0,k}\right) \Vert $ in the tangent plane by the stencil diameter
as follows,%
\begin{equation}
\boldsymbol{\tilde{\theta}}\left( \mathbf{x}\right) =(\tilde{\theta}_{1}(%
\mathbf{x}),\ldots ,\tilde{\theta}_{d}(\mathbf{x})),\text{ \ }\tilde{r}_{%
\boldsymbol{\theta },k}\left( \mathbf{x}\right) =\Vert \boldsymbol{\tilde{%
\theta}}\left( \mathbf{x}\right) -\boldsymbol{\tilde{\theta}}\left( \mathbf{x%
}_{0,k}\right) \Vert ,\text{ \ with }\tilde{\theta}_{i}(\mathbf{x})=\frac{%
\theta _{i}(\mathbf{x})}{D_{K,\max }}=\frac{\boldsymbol{t}_{i}({\mathbf{x}%
_{0}})\cdot (\mathbf{x}-{\mathbf{x}_{0}})}{D_{K,\max }},\label{eqn:thetd}
\end{equation}%
where the normalized quantities are scaled to lie in the interval $[0,1]$,
that is, $0\leq \Vert \boldsymbol{\tilde{\theta}}\left( \mathbf{x}\right)
\Vert =\Vert \boldsymbol{\theta }\left( \mathbf{x}\right) /D_{K,\max }\Vert
\leq 1$ and $0\leq \tilde{r}_{\boldsymbol{\theta },k}\left( \mathbf{x}%
\right) =r_{\boldsymbol{\theta },k}\left( \mathbf{x}\right) /D_{K,\max }\leq
1$. Here we use the shorthand $D_{K,\max }$ for $D_{K,\max }(\mathbf{x}_{0})$%
. Then the coefficients in (\ref{eqn:LB_weight}) for approximating the
Laplace-Beltrami operator become
\begin{equation}
\Delta _{M}f\left( \mathbf{x}_{0}\right) \approx \frac{1}{D_{K,\max }^{2}}%
\left( \left( \Delta _{\boldsymbol{\tilde{\theta}}}\boldsymbol{\tilde{\phi}}%
_{\mathbf{x}_{0}}\right) \boldsymbol{\tilde{\Phi}}^{-1}\big(\mathbf{I}-%
\boldsymbol{\tilde{P}}\left( \boldsymbol{\tilde{P}}^{\top }\boldsymbol{%
\Lambda \tilde{P}}\right) ^{-1}\boldsymbol{\tilde{P}}^{\top }\boldsymbol{%
\Lambda }\big)+\left( \Delta _{\boldsymbol{\tilde{\theta}}}\boldsymbol{%
\tilde{p}}_{\mathbf{x}_{0}}\right) \left( \boldsymbol{\tilde{P}}^{\top }%
\boldsymbol{\Lambda \tilde{P}}\right) ^{-1}\boldsymbol{\tilde{P}}^{\top }%
\boldsymbol{\Lambda }\right) \mathbf{f}_{\mathbf{x}_{0}},
\label{eqn:normMon}
\end{equation}%
where $\boldsymbol{\tilde{P}}\in \mathbb{R}^{K\times m}$, $\boldsymbol{%
\tilde{\Phi}}\in \mathbb{R}^{K\times K}$, $\Delta _{\boldsymbol{\tilde{\theta%
}}}\boldsymbol{\tilde{p}}_{\mathbf{x}_{0}}\in \mathbb{R}^{1\times m}$ and $%
\Delta _{\boldsymbol{\tilde{\theta}}}\boldsymbol{\tilde{\phi}}_{\mathbf{x}%
_{0}}\in \mathbb{R}^{1\times K}$ are all defined over the normalized Monge
coordinates with components%
\begin{equation}
\begin{array}{ll}
\boldsymbol{\tilde{P}}_{kj}=p_{\boldsymbol{\alpha }(j)}\left( \boldsymbol{%
\tilde{\theta}}\left( \mathbf{x}{_{0,k}}\right) \right) ,\text{ \ \ } &
[\Delta _{\boldsymbol{\tilde{\theta}}}\boldsymbol{\tilde{p}}_{\mathbf{x}%
_{0}}]_{j}=[\Delta _{\boldsymbol{\theta }}\boldsymbol{p}_{\mathbf{x}%
_{0}}]_{j}=\left\{
\begin{array}{ll}
2, & \boldsymbol{\alpha }(j)\in \mathbf{E} \\
0, & \boldsymbol{\alpha }(j)\notin \mathbf{E}%
\end{array}%
\right. ,\text{ \ } \\
\boldsymbol{\tilde{\Phi}}_{ks}=\phi \left( \Vert \boldsymbol{\tilde{\theta}}%
\left( \mathbf{x}_{0,k}\right) -\boldsymbol{\tilde{\theta}}\left( \mathbf{x}%
_{0,s}\right) \Vert \right) ,\text{ \ \ } & [\Delta _{\boldsymbol{\tilde{%
\theta}}}\boldsymbol{\tilde{\phi}}_{\mathbf{x}_{0}}]_{k}=\left( 4\kappa
^{2}+2d\kappa +d-1\right) \Vert \boldsymbol{\tilde{\theta}}\left( \mathbf{x}%
_{0}\right) -\boldsymbol{\tilde{\theta}}\left( \mathbf{x}_{0,k}\right) \Vert
^{2\kappa -1},%
\end{array}
\label{eqn:Phitd}
\end{equation}%
for $j=1,\ldots ,m$ and $k,s=1,\ldots ,K$.


\subsection{Weighted ridge regression for the inverse of the PHS matrix}

\label{sec:wtrg} Since the PHS matrix $\boldsymbol{\tilde{\Phi}}$
might be singular, we use a weighted ridge regression or Tikhonov
regularization for gRBF-FD to compute the inverse $\boldsymbol{\tilde{\Phi}}%
^{-1}$ in (\ref{eqn:normMon}),%
\begin{equation}
\underset{{\mathbf{\boldsymbol{\tilde{a}}\in }}\mathbb{R}^{K}}{\min }\frac{1%
}{2}\left( \mathbf{s}_{{\mathbf{x}_{0}}}-\boldsymbol{\tilde{\Phi}\mathbf{%
\tilde{a}}}\right) ^{\top }\boldsymbol{\Lambda }_{K}\left( \mathbf{s}_{{%
\mathbf{x}_{0}}}-\boldsymbol{\tilde{\Phi}\mathbf{\tilde{a}}}\right) +\frac{1%
}{2}\delta ^{2}\Vert \boldsymbol{\mathbf{\tilde{a}}}\Vert ^{2},
\label{eqn:opphs}
\end{equation}%
where $\boldsymbol{\Lambda }_{K}$ is the diagonal $1/K$ weight function
defined in (\ref{eqn:k_weight})\ and $\delta $ is the regularization
parameter. For the gRBF-FD method, the inverse $\boldsymbol{\tilde{\Phi}}%
^{-1}$\ in (\ref{eqn:normMon}) can be found by solving the normal equation
to be,
\begin{equation}
\boldsymbol{\tilde{\Phi}}^{-1}:=\boldsymbol{\tilde{\Phi}}_{\boldsymbol{%
\Lambda }_{K}}^{\dagger }=\left( \boldsymbol{\tilde{\Phi}}^{\top }%
\boldsymbol{\Lambda }_{K}\boldsymbol{\tilde{\Phi}}+\delta ^{2}\mathbf{I}%
\right) ^{-1}\boldsymbol{\tilde{\Phi}}^{\top }\boldsymbol{\Lambda }_{K},
\label{eqn:PhiIn}
\end{equation}%
Here, we set $\delta =10^{-6}$ for a stable and unique inverse of $%
\boldsymbol{\tilde{\Phi}}$. Notably, we emphasize that the choice of $%
\boldsymbol{\Lambda }_{K}$ is very important for the stability of the
discretized Laplacian matrix as can be seen from Figs.~\ref{fig:PRFD_1} and %
\ref{fig:PRFD_2}. The weight $\boldsymbol{\Lambda }_{K}$ in this
optimization problem (\ref{eqn:opphs}) stabilizes the Laplacian analogous to
its role in the GMLS problem in (\ref{eqn:LS_optm}) (see e.g., \cite%
{liang2013solving,jiang2024generalized} for the numerical performance).



\begin{figure*}[htbp]
\centering
\begin{tabular}{ccc}
{(a) well-sampled data} & {(b) coefficients, well-sampled } & {(c) $\Vert
\boldsymbol{L}^{-1}_{\mathbf{X}_M,\boldsymbol{I}} \Vert _\infty$,
well-sampled} \\
\includegraphics[width=1.8
				in, height=1.5 in]{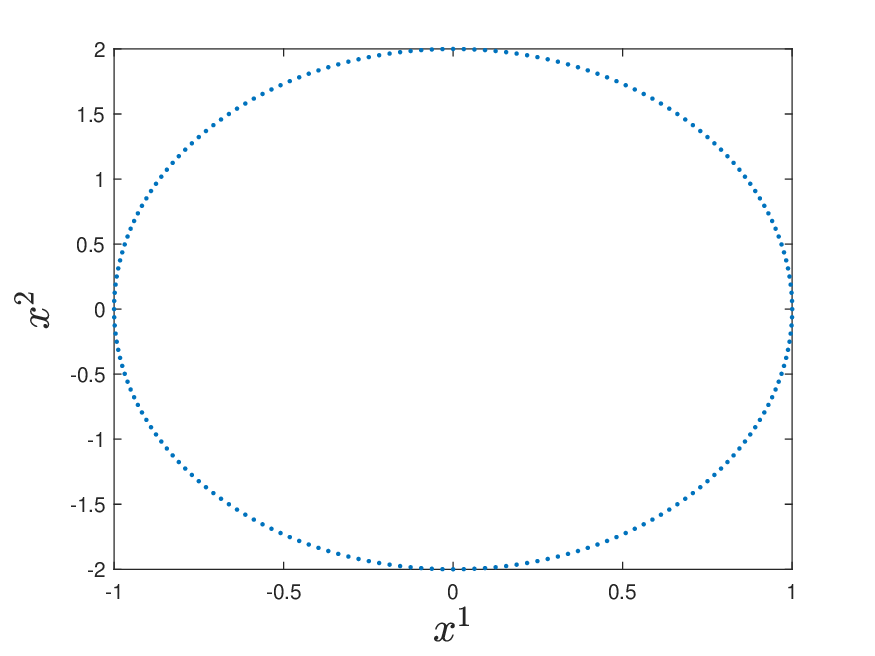} &
\includegraphics[width=2.1
				in, height=1.5 in]{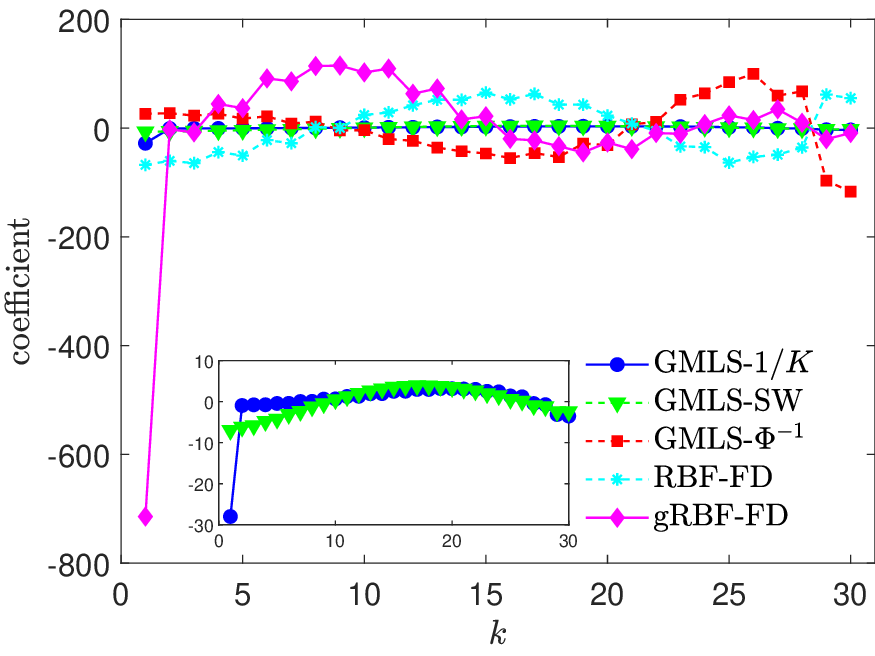} &
\includegraphics[width=2.1
				in, height=1.51 in]{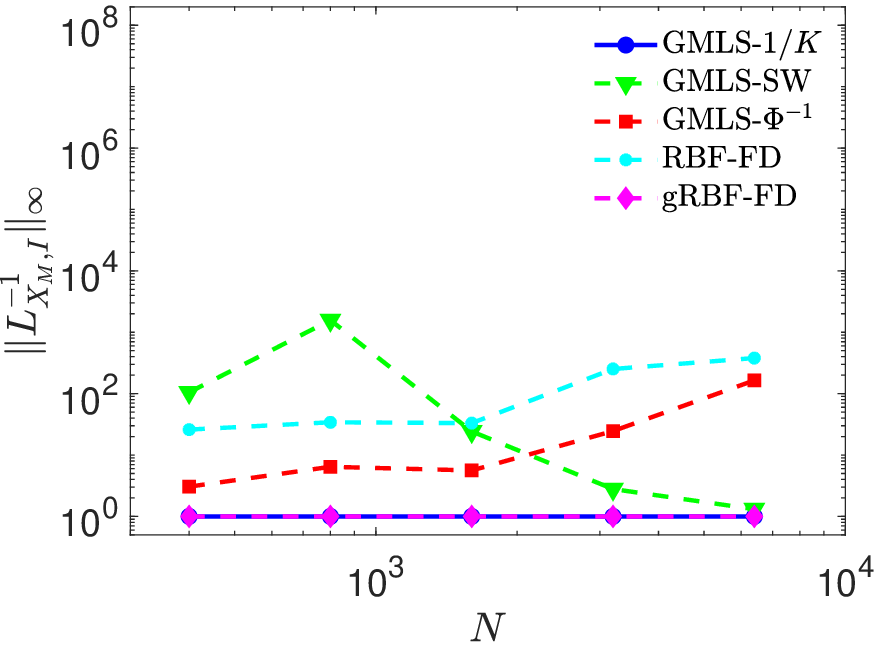} \\
{(d) random data} & {(e) coefficients, random } & {(f) $%
\Vert \boldsymbol{L}^{-1}_{\mathbf{X}_M,\boldsymbol{I}} \Vert _\infty$,
random } \\
\includegraphics[width=1.8
				in, height=1.5 in]{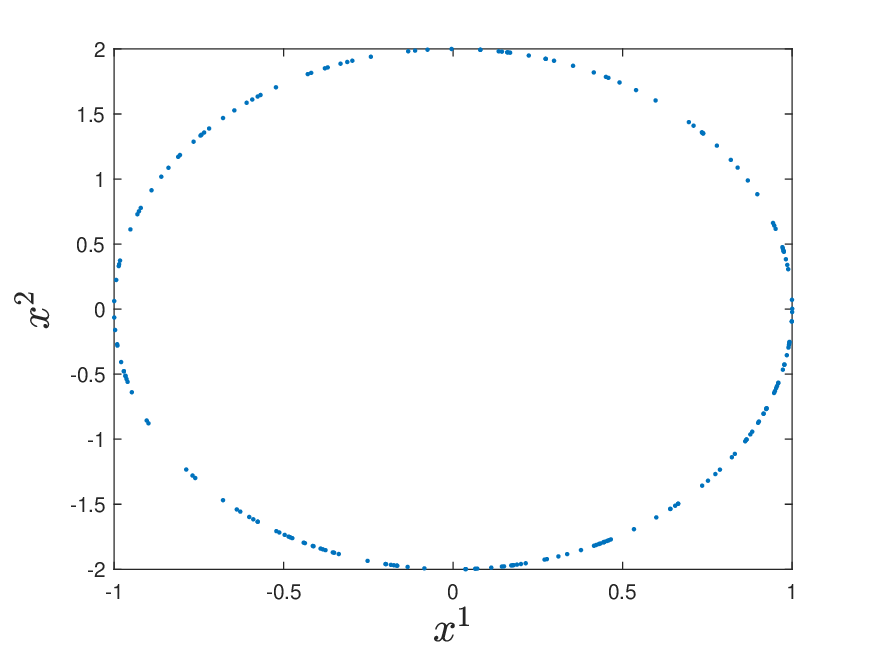} &
\includegraphics[width=2.1
				in, height=1.5 in]{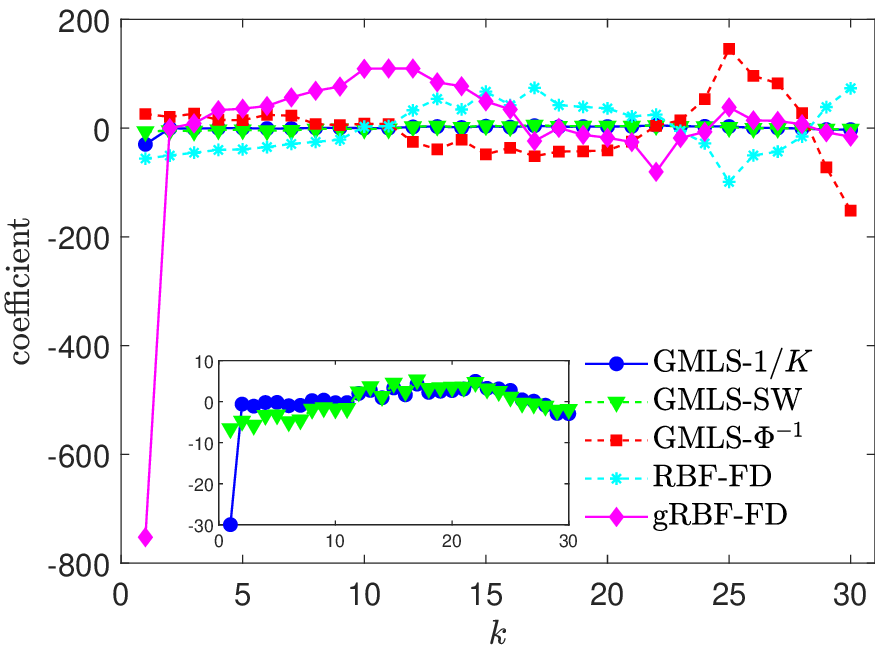} &
\includegraphics[width=2.1
				in, height=1.51 in]{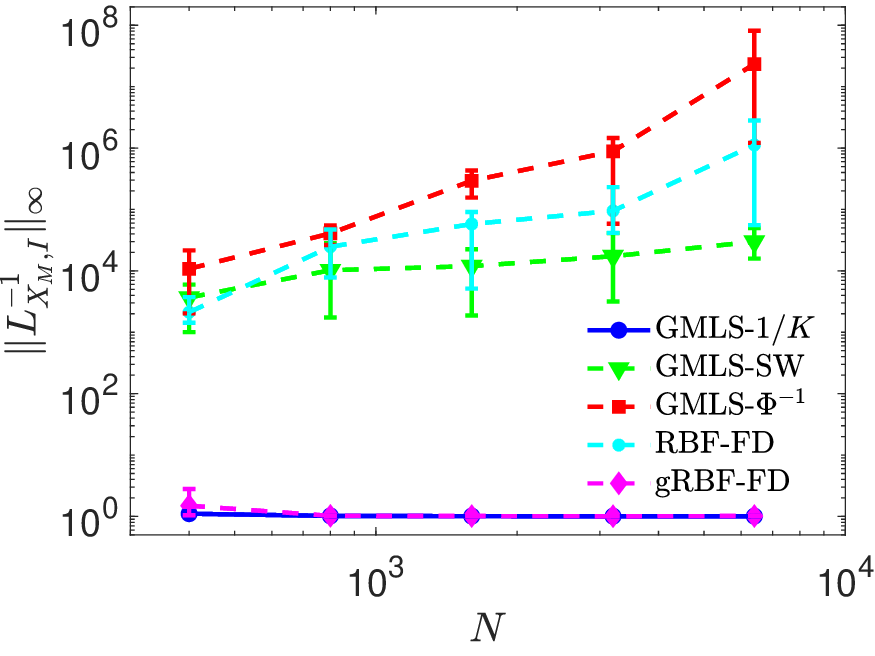}%
\end{tabular}%
\caption{\textbf{1D ellipse in} $\mathbb{R}^2$. Comparison among GMLS (with
three different weights), RBF-FD and gRBF-FD. The upper panels are
well-sampled data and the bottom panels are random data. The left(a)(d),
middle(b)(e), right(c)(f) columns correspond to the scattered dataset,
coefficient $w_k$ vs. $k$th nearest neighbor, $\Vert \boldsymbol{L}_{\mathbf{%
X}_{M},\boldsymbol{I}}^{-1}\Vert _{\infty }$ vs. $N$, respectively. We fix $%
K=30$, polynomial degree $4$ and PHS parameter $\protect\kappa=3$. }
\label{fig:PRFD_1}
\end{figure*}

For comparison, we choose three GMLS methods and two types of RBF-FD
methods. The Laplace-Beltrami operator using GMLS methods can be calculated
by
\begin{equation}
\Delta _{M}f\left( \mathbf{x}_{0}\right) \approx \frac{1}{D_{K,\max }^{2}}%
\left( \left( \Delta _{\boldsymbol{\tilde{\theta}}}\boldsymbol{\tilde{p}}_{%
\mathbf{x}_{0}}\right) \left( \boldsymbol{\tilde{P}}^{\top }\boldsymbol{%
\Lambda \tilde{P}}\right) ^{-1}\boldsymbol{\tilde{P}}^{\top }\boldsymbol{%
\Lambda }\right) \mathbf{f}_{\mathbf{x}_{0}},  \label{eqn:gmLB}
\end{equation}%
which corresponds to the GMLS part of equation (\ref{eqn:normMon}). One may
refer to \cite{liang2013solving,gross2020meshfree,jones2023generalized} for
more calculation details. GMLS methods can be applied with three different
weight functions:

\noindent \textbullet\ GMLS-$1/K$. The first weight function is the diagonal
$1/K$  as discussed in equation (\ref{eqn:k_weight}) in Remark \ref%
{rem:1oK}. \newline
\noindent \textbullet\ GMLS-SW. The second is a diagonal smooth weight
function with diagonal entries $\lambda _{kk}=\left( 1-\frac{r_{k}}{R}%
\right) ^{2}$, where $r_{k}$ is defined as $r_{k}=\Vert \boldsymbol{\theta }%
\left( \mathbf{x}_{0,k}\right) -\boldsymbol{\theta }\left( \mathbf{x}%
_{0}\right) \Vert $ and $R=1.5R_{K,\max }$, where $R_{K,\max }$ is the
stencil radius. \newline
\noindent \textbullet\ GMLS-$\boldsymbol{\Phi }^{-1}$. For the third one, we
take the weight $\boldsymbol{\Lambda }$ to be the inverse of the symmetric
PHS matrix, that is, $\boldsymbol{\tilde{\Phi}}^{-1}$ for invertible $%
\boldsymbol{\tilde{\Phi}}$\ and $\boldsymbol{\tilde{\Phi}}_{\mathbf{I}%
}^{\dagger }:=\left( \boldsymbol{\tilde{\Phi}}^{\top }\boldsymbol{\tilde{\Phi%
}}+\delta ^{2}\mathbf{I}\right) ^{-1}\boldsymbol{\tilde{\Phi}}^{\top }$ for
singular $\boldsymbol{\tilde{\Phi}}$. Here $\boldsymbol{\tilde{\Phi}}$ is
defined in (\ref{eqn:Phitd}) and the $\delta ^{2}\mathbf{I}$ term is used
for the uniqueness of the inverse of the singular $\boldsymbol{\tilde{\Phi}}$%
. Indeed, this set of Laplacian coefficients corresponds to the GMLS part of
the RBF-FD coefficients as can be derived from equation (\ref{eqn:grbf1}).

While there are many gRBF-FD methods depending on the choice of weight
functions and radial basis functions, we here only consider the following
two typical gRBF-FD methods.

\noindent \textbullet\ RBF-FD. The RBF-FD approach follows from the review
in Section \ref{sec:PRFD} and references in \cite{shaw2019radial,jones2023generalized}.
The Laplacian coefficients can be constructed from equation (\ref%
{eqn:normMon}) by substituting
$\boldsymbol{\Lambda }$  with the inverse of the PHS matrix, $\boldsymbol{\tilde{\Phi}}^{-1}$. When $%
\boldsymbol{\tilde{\Phi}}$ is invertible, its inverse can be
calculated directly by $\boldsymbol{\tilde{\Phi}}^{-1}$.
When $\boldsymbol{\tilde{\Phi%
}}$ is singular, its inverse can be approximated either from a
pseudo-inverse based on its singular values or from a
ridge regression $\boldsymbol{\tilde{\Phi}}_{\mathbf{I}}^{\dagger }:=\left(
\boldsymbol{\tilde{\Phi}}^{\top }\boldsymbol{\tilde{\Phi}}+\delta ^{2}%
\mathbf{I}\right) ^{-1}\boldsymbol{\tilde{\Phi}}^{\top }$. Here we simply
employ a common technique to approximate the inverse for singular $%
\boldsymbol{\tilde{\Phi}}$. \newline
\noindent \textbullet\ gRBF-FD using the $1/K$ weight. The Laplacian coefficients are
calculated exactly following equation (\ref{eqn:normMon}) with the $1/K$
weight $\boldsymbol{\Lambda } = \boldsymbol{\Lambda }_{K}$ in (\ref{eqn:k_weight}) and the specific
inverse $\boldsymbol{\tilde{\Phi}}^{-1}=\boldsymbol{\tilde{\Phi}}_{%
\boldsymbol{\Lambda }_{K}}^{\dagger }$\ in (\ref{eqn:PhiIn}). The
difference between RBF-FD and gRBF-FD lies in the usage of the weight function.
In particular, for gRBF-FD using the $1/K$ weight,
we allocate more weights to the base point than all other neighboring points
in the first GMLS regression step as well as in the second PHS interpolation step.


Figures \ref{fig:PRFD_1}(b)(e) display the Laplacian coefficient $w_{k}$\ as
a function of the $k$th nearest neighbor for a fixed stencil. It can be seen
that for GMLS-$1/K$ (blue) and gRBF-FD (magenta), the base point coefficient $|w_{1}|$
is significantly larger than all other neighboring coefficients $%
|w_{k}|$ ($k=2,\ldots ,K$). We refer to this characteristic as the \textit{central spike
pattern} of the coefficients $w_{k}$). However, for GMLS-SW (green), GMLS-$\boldsymbol{\Phi }%
^{-1}$ (red), RBF-FD (cyan), their coefficients are nearly uniform within a
certain range across the stencil, with no central spike pattern.

Figures \ref{fig:PRFD_1}(c) and \ref{fig:PRFD_1}(f) display the infinity
norm of the inverse of the Laplacian matrix, $\Vert \boldsymbol{L}_{\mathbf{X%
}_{M},\boldsymbol{I}}^{-1}\Vert _{\infty }$, for well-sampled data and
randomly-sampled data, respectively. Indeed, the quantity $\Vert \boldsymbol{%
L}_{\mathbf{X}_{M},\boldsymbol{I}}^{-1}\Vert _{\infty }$ reveals the
stability of the Laplacian matrix (see e.g., \cite%
{ahlberg1963,varah1975,gh2019,jiang2024generalized}). As shown in Fig. %
\ref{fig:PRFD_1}(c), for well-sampled data, all values of $\Vert
\boldsymbol{L}_{\mathbf{X}_{M},\boldsymbol{I}}^{-1}\Vert _{\infty }$ are
bounded above by a constant, which implies stability; in contrast,  the values of GMLS-SW (green), GMLS-$\boldsymbol{\Phi }%
^{-1} $ (red), and RBF-FD (cyan) are significantly larger than those of GMLS-$1/K$ (blue)
and gRBF-FD (magenta).
However, for randomly-sampled data (Fig. \ref{fig:PRFD_1}(f)), the values of $\Vert \boldsymbol{L}_{\mathbf{X}_{M},%
\boldsymbol{I}}^{-1}\Vert _{\infty }$ for GMLS-$1/K$ (blue) and gRBF-FD (magenta) remain a small constant; whereas those of the other three methods increase with $N$, potentially leading to their instability.


Moreover, a close connection can be observed between the
pattern of coefficients $w_{k}$ and the value of $\Vert \boldsymbol{L}_{%
\mathbf{X}_{M},\boldsymbol{I}}^{-1}\Vert _{\infty }$. In particular,  the coefficients for GMLS-$1/K$ and gRBF-FD exhibit the central spike pattern, resulting in
a small, constant $\Vert \boldsymbol{L}_{%
\mathbf{X}_{M},\boldsymbol{I}}^{-1}\Vert _{\infty }$ value and consequently demonstrating good
stability. In fact, according to the Lax Equivalence Theorem, these
numerical observations will be  connected to the convergence of the
solution, as demonstrated below in Fig. \ref{fig:PRFD_2}.


\begin{figure*}[tbph]
\centering
\begin{tabular}{cc}
{(a) \textbf{FE}, well-sampled} & {(b) \textbf{FE}, randomly-sampled } \\
\includegraphics[width=2.5
				in, height=2. in]{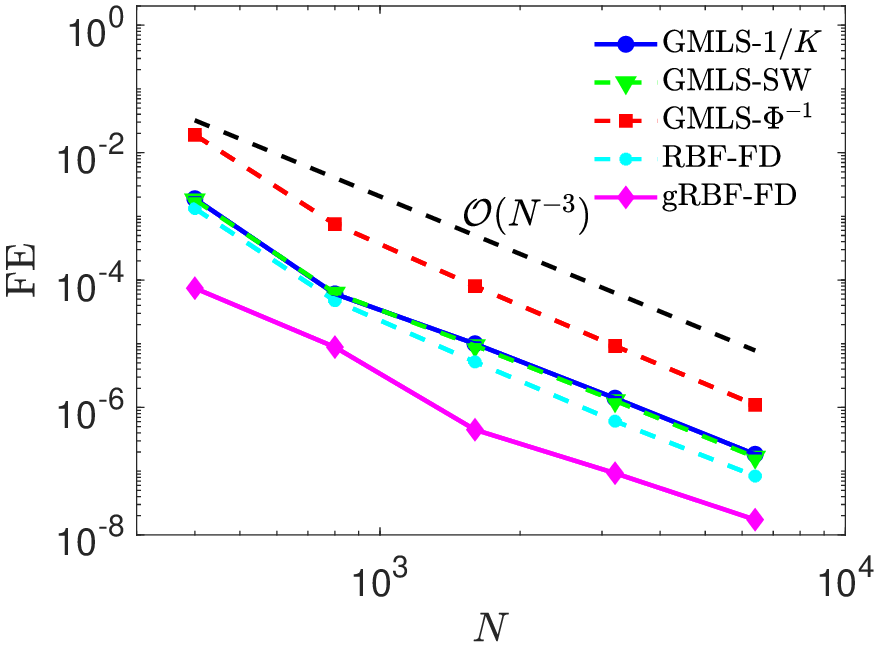} &
\includegraphics[width=2.5
				in, height=2. in]{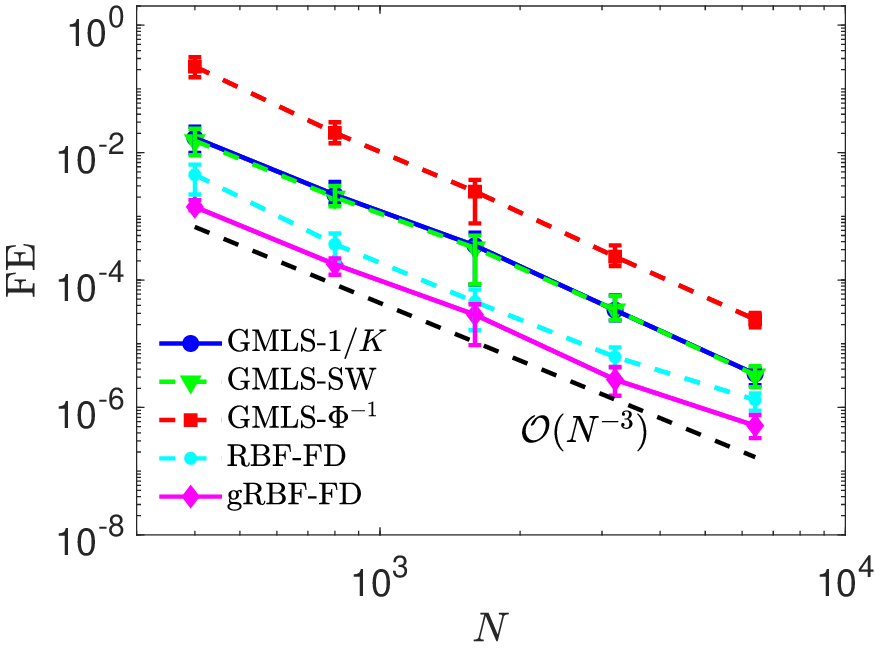} \\
{(c) \textbf{IE}, well-sampled} & {(d) \textbf{IE}, randomly-sampled } \\
\includegraphics[width=2.5
				in, height=2. in]{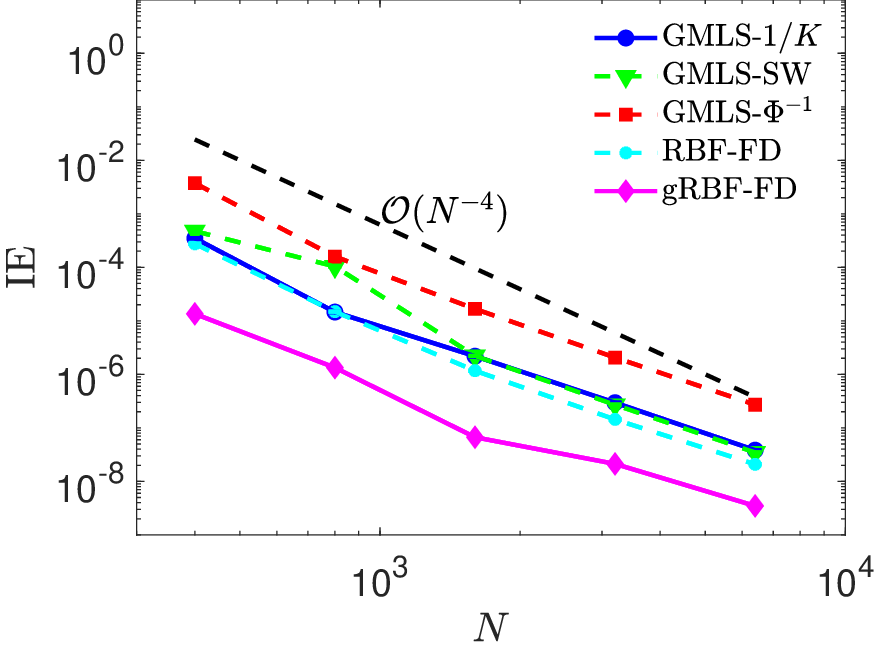} &
\includegraphics[width=2.5
				in, height=2. in]{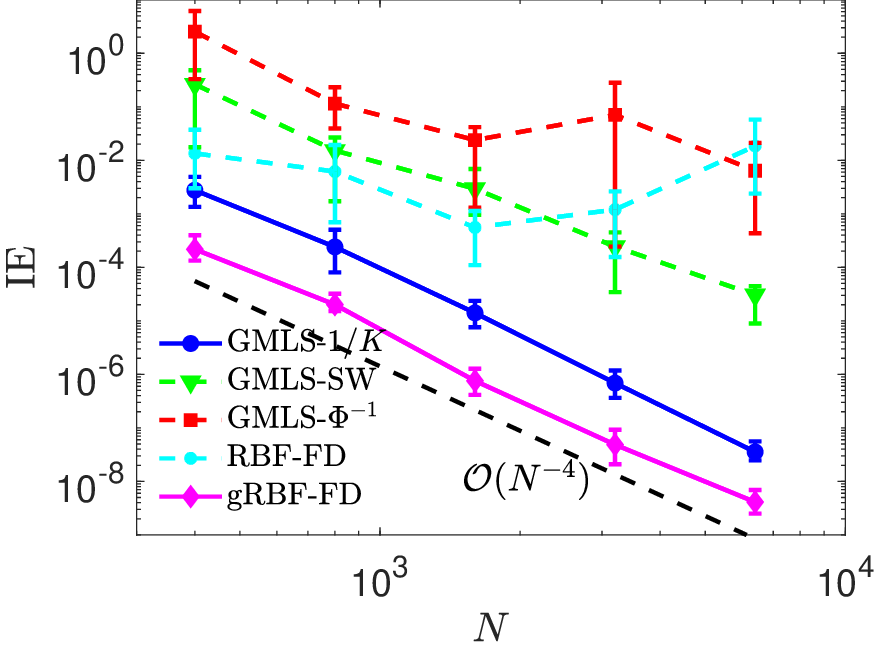}%
\end{tabular}%
\caption{\textbf{1D ellipse in} $\mathbb{R}^{2}$. Comparison among GMLS
(with three weight functions), RBF-FD and gRBF-FD. The left panels are
well-sampled data and the right panels are random data. The upper panels are
forward error (\textbf{FE}) vs. $N$ while the bottom panels are inverse
error (\textbf{IE}) vs. $N$. We fix $K=30$, polynomial degree $4$ and PHS
parameter $\protect\kappa =3$. }
\label{fig:PRFD_2}
\end{figure*}

\begin{figure*}[tbph]
\centering
\begin{tabular}{c}
\includegraphics[width=5.5
				in, height=1.7 in]{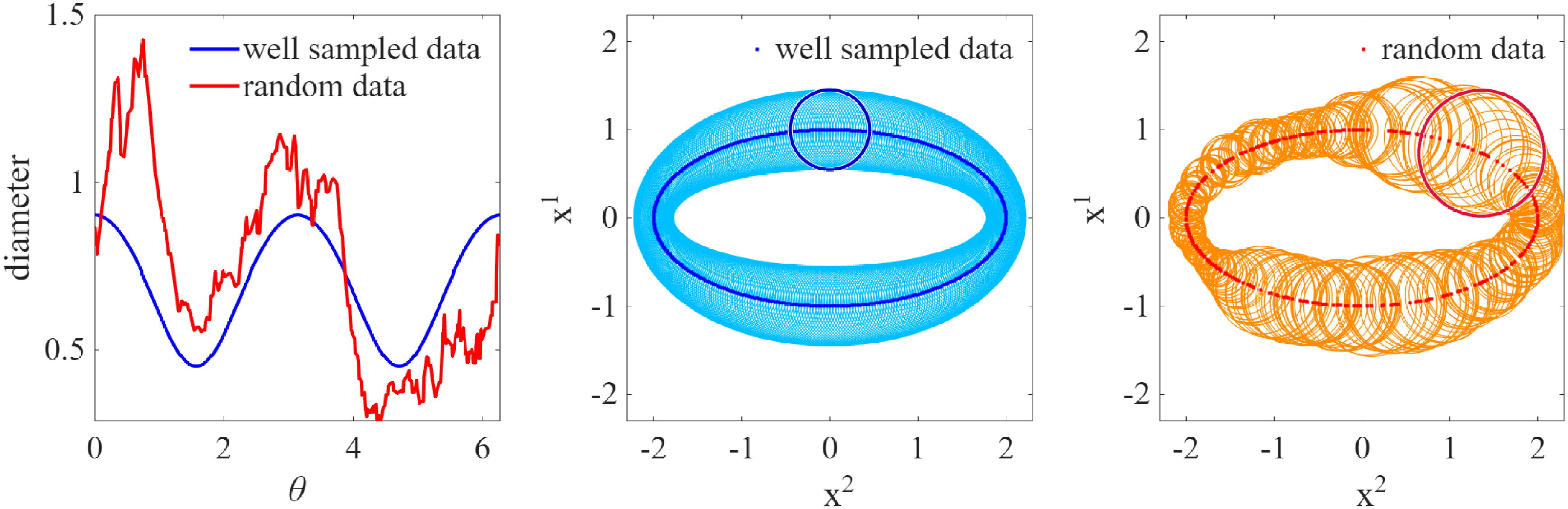}
\end{tabular}%
\caption{\textbf{1D ellipse in} $\mathbb{R}^{2}$. (Left) Stencil diameter as a function $\theta$ for $N=400$ points. Plotted are the sizes of all stencils for
well-sampled data (middle, light blue circles) and  random data (right, orange circles).
Each circle is plotted with a center $\mathbf{x}_{i}$ and a radius $D_{K,\max }({\mathbf{x}_{i}})/2$.
The circle plotted using dark blue (middle) and dark red (right) corresponds to the largest stencil diameter.
We fix $K=30$, polynomial degree $4$ and PHS
parameter $\protect\kappa =3$. }
\label{fig:diameter}
\end{figure*}

In Fig. \ref{fig:PRFD_2}(a) and Fig. \ref{fig:PRFD_2}(c), we illustrate the
consistency of the Laplace-Beltrami approximation and the convergence of the
solution approximation for well-sampled data. All the five
methods show the consistency rate of $\mathcal{O}\left( N^{-3}\right) $ as
well as the super-convergence rate of $\mathcal{O}\left( N^{-4}\right) $.
For well-sampled data or quasi-uniform data, the convergence of GMLS  has
been numerically examined  \cite{liang2013solving,gross2020meshfree,jones2023generalized,jiang2024generalized},
 as has that of
RBF-FD  \cite%
{shankar2015radial,flyer2016role,lehto2017radial,jones2023generalized}. The super-convergence phenomenon has also been discussed in \cite{liang2013solving}.

For randomly-sampled data, the results become different among GMLS, RBF-FD
and gRBF-FD. All the five methods still show the consistency rate of $%
\mathcal{O}\left( N^{-3}\right) $ as expected (Fig. \ref{fig:PRFD_2}(b)).
However, only GMLS-$1/K$ and gRBF-FD achieve convergent solutions  on random data, exhibiting both
small errors and small error fluctuations. This good performance stems from the stability of the approximate Laplacians as well as from the
central spike pattern of the coefficients $w_{k}$. Consequently, employing the $1/K$ weight function is very
important for ensuring the stability and convergence for both the GMLS and
gRBF-FD approaches.


In addition, Fig.~\ref{fig:PRFD_2} shows that the \textbf{FE}s for random data are nearly 10 times those for well-sampled data. As shown in Lemma \ref{lem:inte} and Theorem \ref{thm:Dk} below,
the \textbf{FE}s are of order $O(D_{K,\max }^{l-1})$, where $D_{K,\max }$ is the stencil diameter and $l$ is the polynomial degree. From Fig.~\ref{fig:diameter}, the largest stencil diameter for random data is nearly twice that for well-sampled data. Therefore, theoretically, the \textbf{FE}s for random data should be approximately $2^3=8$ times larger. It is a prediction that agrees well with our numerical observations.

\subsection{Automatic specification of $K$-nearest neighbors}

\label{sec:autoK} We now discuss an automated method of tuning the $K$%
-nearest neighbors for our gRBF-FD method. This automatic tuning method can
also be applied to the GMLS method.



\begin{figure*}[htbp]
\centering
\begin{tabular}{ccc}
{(a) coefficient} & {(b) stability } & {(c) eigenvalue} \\
\includegraphics[width=2.
		in, height=1.8 in]{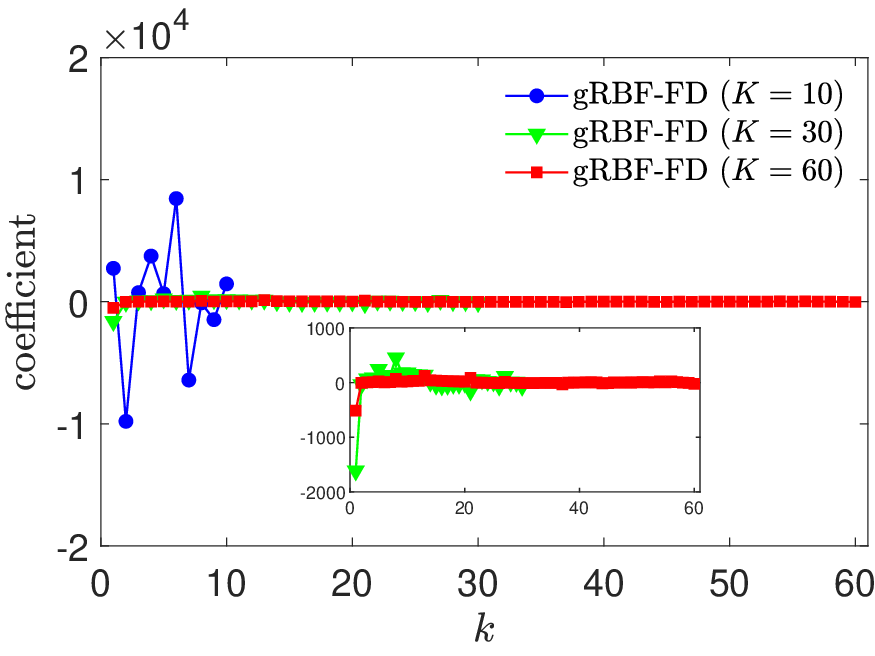} &
\includegraphics[width=2.
		in, height=1.8 in]{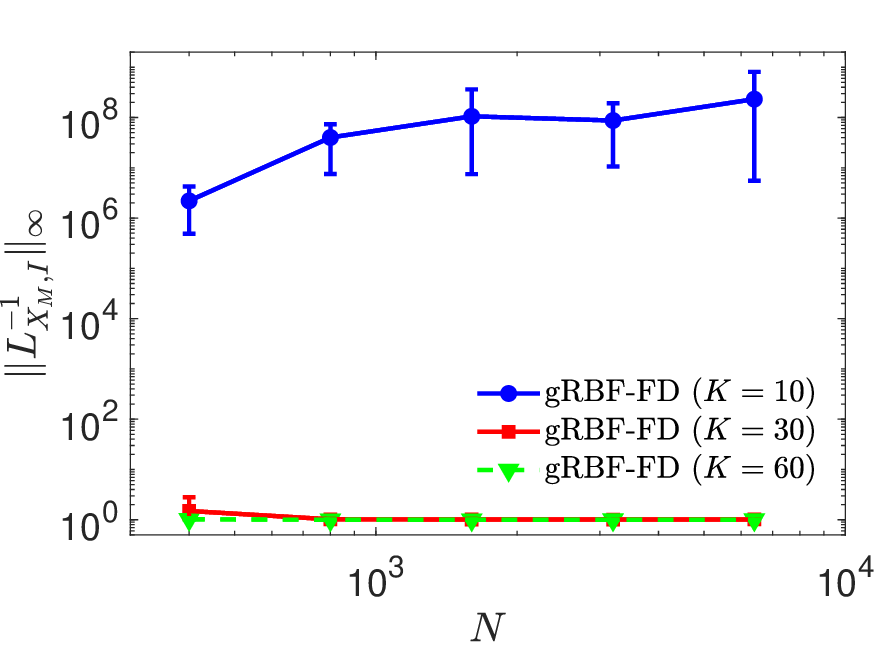} &
\includegraphics[width=2.
		in, height=1.8 in]{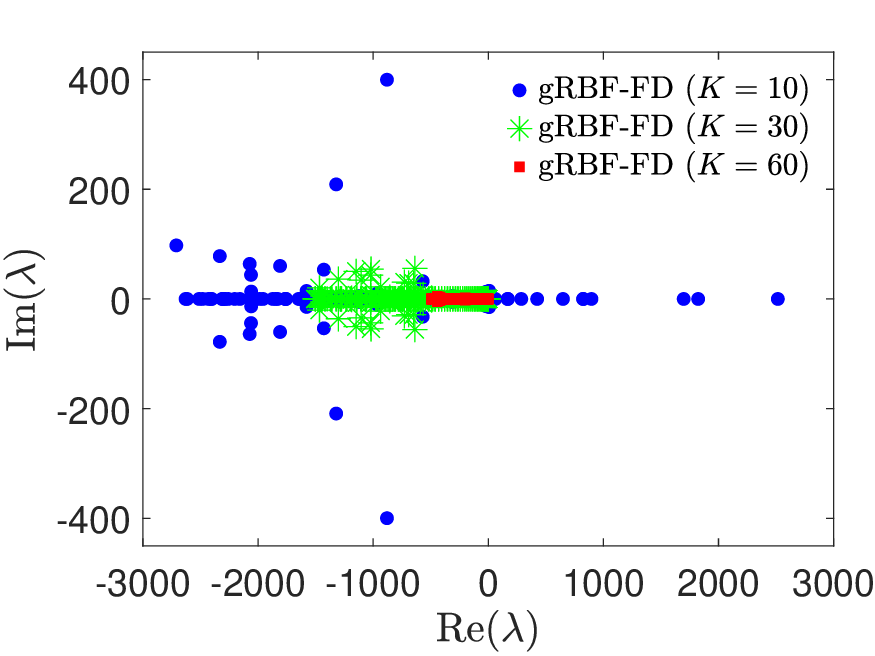}%
\end{tabular}%
\caption{\textbf{1D ellipse in} $\mathbb{R}^2$ using random data. Shown are
(a) the Laplacian coefficient $w_k$ vs. $k$th nearest neighbor, (b) $\Vert%
\boldsymbol{L}^{-1}_{\mathbf{X}_M,\boldsymbol{I}}\Vert_{\infty}$ vs. $N$,
and (c) the leading 200 eigenvalues of $\boldsymbol{L}_{\mathbf{X}_M}$ for
different  $K$ neighbors. We fix polynomial degree 4 and PHS parameter $%
\protect\kappa=3$ for all panels. In panels (a)(c), we use $N=1600$ random
data points. }
\label{fig:whyauto}
\end{figure*}

\textbf{Why do we need to tune }$K$\textbf{\ automatically? }


We first  note that relatively small $K$ may provide unstable approximation
to the Laplacian. For the 1D ellipse example, we know that $K$ should be
chosen such that $K\geq m=l+1=5$ if the polynomial degree is $l=4$. We now
investigate the numerical performance for three values of $K$ $(\geq 5)$ in
terms of the Laplacian coefficients $w_{k}$ (Fig. \ref{fig:whyauto}%
(a)), the stability $\Vert \boldsymbol{L}_{\mathbf{X}_{M},\boldsymbol{I}%
}^{-1}\Vert _{\infty }$ (Fig. \ref{fig:whyauto}(b)), and the leading
eigenvalues (Fig. \ref{fig:whyauto}(c)). For a relatively small value of $K$
(e.g., $K=10$), the numerical results show the following closely connected
characteristics: a positive coefficient at the base point ($w_{1}>0$) (Fig. %
\ref{fig:whyauto}(a)), an unstable approximation to the Laplacian (Fig. \ref%
{fig:whyauto}(b)), and the presence of spurious eigenvalues in the right
half of the complex plane (Fig. \ref{fig:whyauto}(c)). Here, the
Laplace-Beltrami operator is negative semi-definite. For relatively large
values of $K$ (e.g., $30,60$), the Laplacian approximation becomes stable and
all leading eigenvalues lie in the left half complex plane. Hence, avoiding
relatively small values of $K$ is crucial for the successful implementation
of the gRBF-FD method (as well as the GMLS method).


There are other minor reasons for auto-tuning $K$. For instance, a global
choice of $K$ for all points might not be a good strategy, especially for
randomly sampled data. Also for random data, it can be challenging to get $K$
right on the first try, especially in three-dimensional or higher
dimensional problems. Here, we only conjecture some potential drawbacks of
manually-tuned $K$, without showing evidence or numerical validation.

\begin{figure*}[htbp]
\centering
\begin{tabular}{cc}
{(a) $w_1$} & {(b) $\gamma$ } \\
\includegraphics[width=2.8
		in, height=2.3 in]{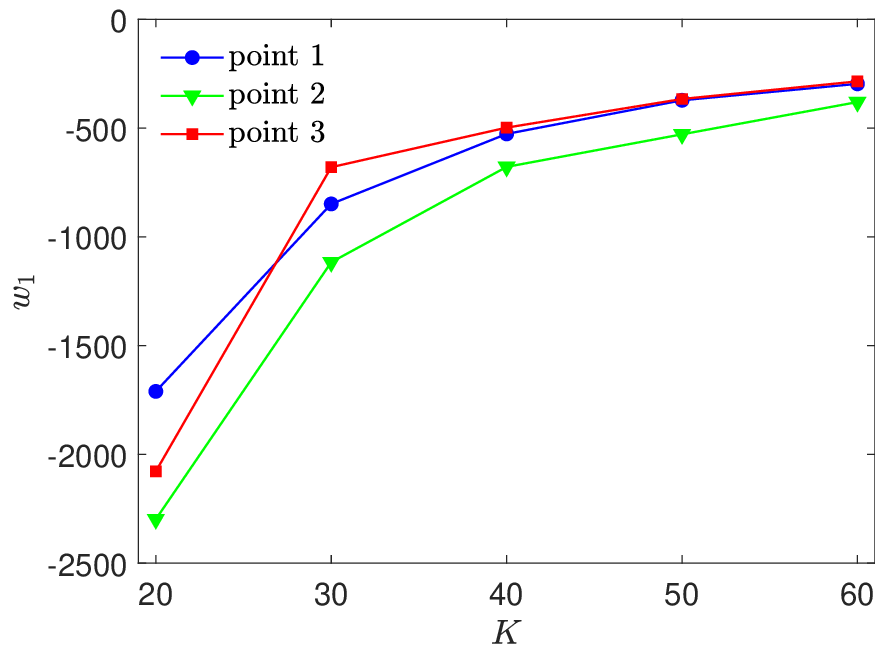} &
\includegraphics[width=2.8
		in, height=2.3 in]{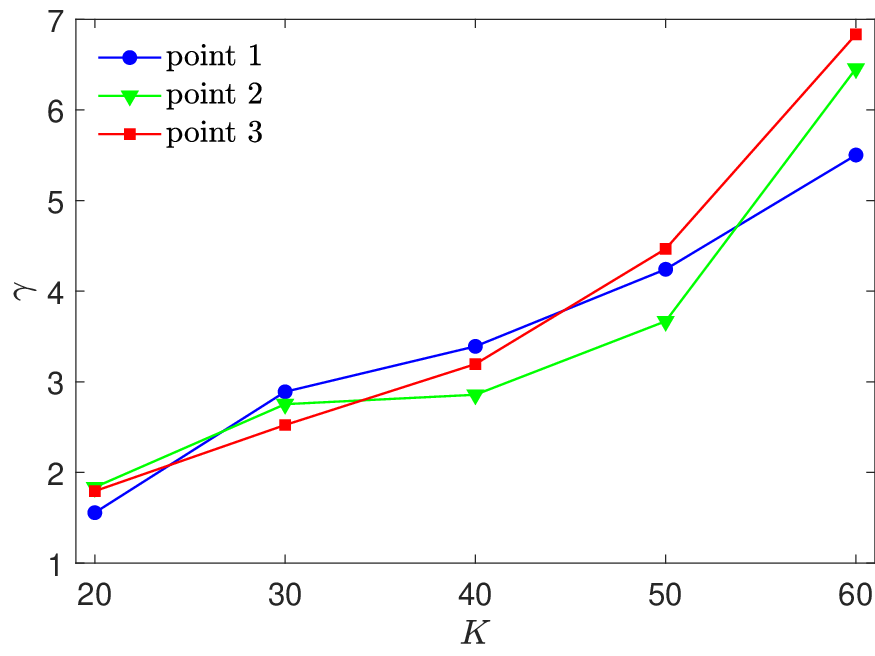}%
\end{tabular}%
\caption{\textbf{1D ellipse in} $\mathbb{R}^2$ using random data. Shown are
(a) the coefficient of the center point $w_1$ vs. the stencil size $K$ and
(b) the ratio $\protect\gamma$ vs. $K$. In both panels (a) and (b), we show
the results on three different stencils. We fix the number of points $N=1600$%
, polynomial degree 4 and PHS parameter $\protect\kappa=3$.}
\label{fig:howauto}
\end{figure*}

\textbf{How to automatically tune }$K$\textbf{?}

\comment{3. w1 vs. K 4. gamma vs. K}

As can be seen from Fig. \ref{fig:whyauto}, there is a close connection
among the Laplacian coefficients $w_{k}$, the stability $\Vert \boldsymbol{L}%
_{\mathbf{X}_{M},\boldsymbol{I}}^{-1}\Vert _{\infty }$,\ and the leading
eigenvalues. We hope to reduce the upper bound of $\Vert \boldsymbol{L}_{%
\mathbf{X}_{M},\boldsymbol{I}}^{-1}\Vert _{\infty }$ as well as to ensure
that all eigenvalues lie in the left half plane. However, it might not be
easy to predict these two properties of the Laplacian matrix before it is
constructed. Fortunately, we are able to control the central spike pattern of the
Laplacian coefficients by tuning the $K$ for each point so that the Laplacian approximation is
expected to be stable.

\begin{definition}\label{def:4.3}
To quantitatively characterize the central spike pattern, we define a ratio $\gamma(\mathbf{x}_i) $
as
\begin{equation*}
\gamma :=\frac{|w_{1}|}{\max_{2\leq k\leq K}\left\vert w_{k}\right\vert },
\end{equation*}%
for each $\mathbf{x}_i \in \mathbf{X}_M$.
The Laplacian matrix generated by the GMLS or gRBF-FD methods is said to be
nearly diagonally dominant if the ratio is such that $\gamma \geq \gamma _{\text{th}}=3$ for each point in $\mathbf{X}_M$.
\end{definition}

Our goal is to find the parameter $K$ such that the Laplacian coefficients $%
\left\{ w_{k}\right\} _{k=1}^{K}$ satisfy two conditions: a negative
coefficient at the base point ($w_{1}<0$) and a ratio $\gamma $ exceeding a
given threshold $\gamma _{\text{th}}$ ($\gamma \geq \gamma _{\text{th}}$).
It can be seen from Fig. \ref{fig:howauto}(a) that the first condition ($%
w_{1}<0 $) can be easily satisfied for the coefficient to be negative. This
can be expected since the Laplace-Beltrami is negative semi-definite. It
can be further seen from Fig. \ref{fig:howauto}(b) that the second condition, $\gamma \geq \gamma _{%
\text{th}}$, can
also be satisfied
when  the threshold  $\gamma _{\text{th}}=3$.
This can be expected because assigning a higher weight to the base point in
both optimization problems (\ref{eqn:LS_optm}) and (\ref{eqn:opphs}) is designed  to increase $\gamma$.
The threshold $%
\gamma _{\text{th}}=3$ was chosen empirically because it consistently leads
to a stable approximation in our experiments.
Moreover, in all
numerical examples in Section \ref{sec:numerical}, we successfully found a parameter $K$ for each base point  that
satisfies both conditions: $w_{1}<0$ and $\gamma \geq \gamma _{\text{th}}=3$.


Now we summarize the automatic tuning strategy of $K$ as follows. For each
point, the parameter $K$ is kept growing from an initial value $K_{0}$
($>m$) until an appropriate $K$ is identified such that a set of
Laplacian coefficients $\{w_k\}_{k=1}^K$ satisfies both $w_{1}<0$\ and $\gamma \geq 3$.
Hence the Laplacian matrix  becomes nearly diagonally
dominant for the GMLS-$1/K$ and gRBF-FD methods.

\comment{%
	if
really
cannot
be
satisfied,
then
find
the
K
with
the
most
negative w1%
}

\begin{figure*}[htbp]
\centering
\begin{tabular}{ccc}
{(a) gRBF-FD} & {(b) GMLS } & {(c) \textbf{IE} vs. $K$} \\
\includegraphics[width=2.
		in, height=1.8 in]{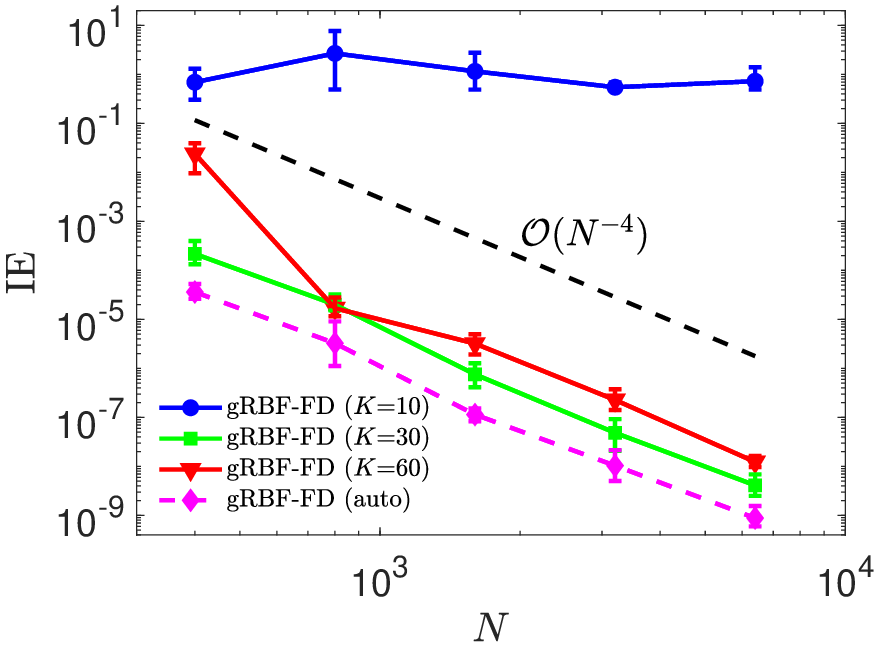} &
\includegraphics[width=2.
		in, height=1.8 in]{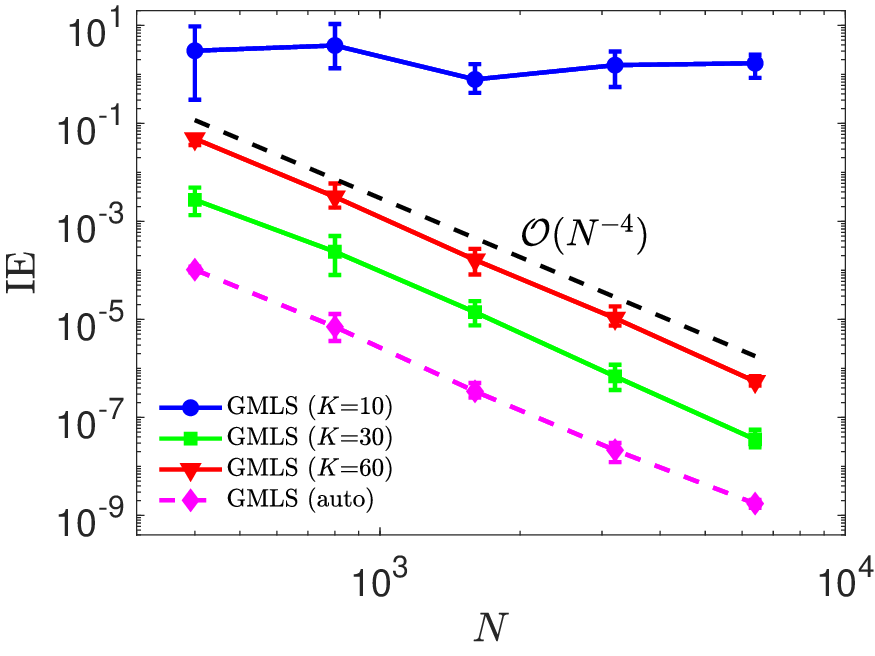} &
\includegraphics[width=2.
		in, height=1.8 in]{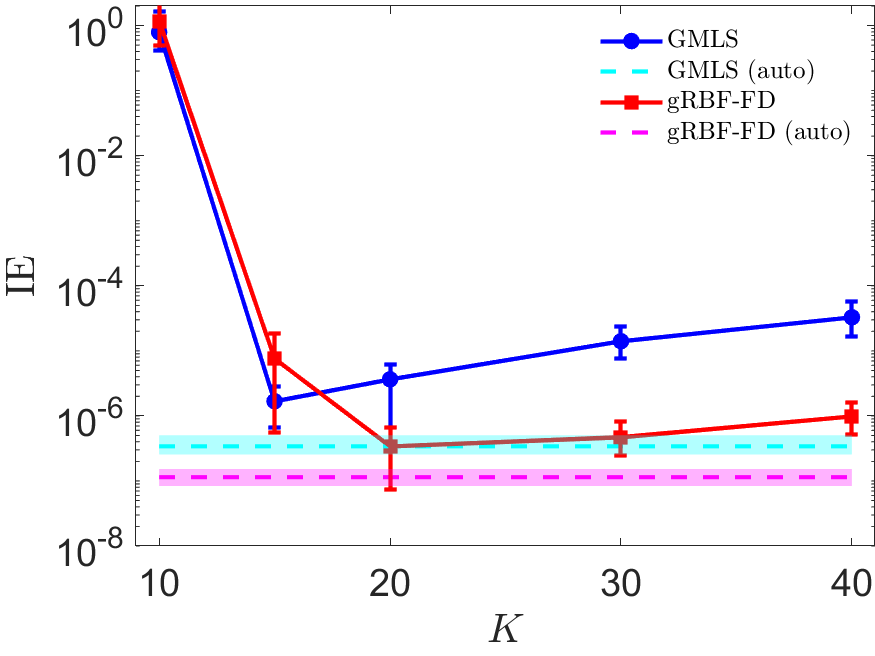}%
\end{tabular}%
\caption{\textbf{1D ellipse in} $\mathbb{R}^2$ using random data. Comparison
of \textbf{IE}s for different stencil size $K$. Panels (a) and (b) display
$\mathbf{IE}$ vs. $N$ for gRBF-FD and GMLS, respectively. Panel (c) displays
$\mathbf{IE}$ vs. $K$ using $N=1600$ points. The errors are plotted with
error bars obtained from the standard deviations. We fix polynomial degree 4
and PHS parameter $\protect\kappa=3$. For gRBF-FD (auto) and GMLS (auto), we
set the initial $K_0=10$.}
\label{fig:PRFD_adpt}
\end{figure*}

\textbf{What is the numerical performance of auto-tuned }$K$\textbf{?}


In Fig. \ref{fig:PRFD_adpt}(a) and Fig. \ref{fig:PRFD_adpt}(b), we compare
the performance of \textbf{IE}s between manually-tuned $K$ and auto-tuned $K$
for the GMLS-$1/K$ and the gRBF-FD, respectively. It can be seen that
a manually chosen large $K(=60)$ provides large \textbf{IE}s, whereas a
manually chosen small $K(=10)$ fails to maintain the stability. In Fig. \ref%
{fig:PRFD_adpt}(c), we show \textbf{IE}s as a function of $K$ for GMLS-$1/K$
and gRBF-FD. It can be observed that gRBF-FD (red solid)\ not only
exhibits smaller \textbf{IE}s than GMLS (blue solid)\ but is also less
sensitive to variations in $K$ for $K\geq 20$. This advantage of gRBF-FD
primarily results from the improved operator approximation and the enhanced
stability using the $1/K$ weight $\boldsymbol{\Lambda }_{K}$. It can be
further observed from Fig. \ref{fig:PRFD_adpt} that approaches using
auto-tuned $K$ always provide smaller \textbf{IE}s than those using
manually-tuned $K$ for both GMLS-$1/K$ and gRBF-FD. Indeed, the
auto-tuning strategy provides an appropriate value of $K$ that stabilizes
the approximation and avoids unnecessary stencil growth for each
point. Therefore, using auto-tuned $K$ is important for achieving good
numerical performance of GMLS-$1/K$ and gRBF-FD.

\subsection{GRBF-FD algorithm using auto-tuned $K$-nearest neighbors}

Combining the computational techniques of the normalization of Monge
coordinates (Sec.~\ref{sec:nmmg}), weighted ridge regression (Sec.~\ref%
{sec:wtrg}), and auto-tuned $K$-nearest neighbors (Sec.~\ref{sec:autoK}), we
now provide our gRBF-FD algorithm with auto-tuned $K$ in Algorithm \ref%
{algo:intrin-LB_ad}. The GMLS algorithm using auto-tuned $K$ can be applied
similarly following the similar procedure.



\begin{algorithm}[ht]
	\caption{gRBF-FD using auto-tuned $K$ for the Laplace-Beltrami operator}
	\begin{algorithmic}[1]
		\STATE {\bf Input:} A point cloud $\{\mathbf{x}_{i}\}_{i=1}^{N}\subset M$, bases of analytic tangent vectors at each node $\left\{ \boldsymbol{t}_{1}(\mathbf{x}_{i}),\ldots,\boldsymbol{t}_{d}(%
		\mathbf{x}_{i})\right\} _{i=1}^{N}$, the degree $l$ of  polynomials, the PHS smoothness parameter $\kappa$, and a parameter $K(=K_0>m)$ nearest neighbors where $%
		m=\left(
		\begin{array}{c}
			l+d \\
			d%
		\end{array}%
		\right) $ is the number of monomial basis functions $\left\{ \boldsymbol{%
			\theta }^{\boldsymbol{\alpha }}|0\leq \left\vert \boldsymbol{\alpha }%
		\right\vert \leq l\right\} $.
		\STATE Set $\boldsymbol{L}_{\mathbf{X}_M}$ to be a sparse $N\times N$ matrix with initial $NK$ nonzeros.
		\FOR{$i\in \{1,2,...,N\}$}
		\STATE Set $\gamma = 0$ and $w_1 = 0$.
		\WHILE{$\gamma < 3$ or $w_1>0$}
		\STATE Find the $K$ nearest neighbors of the point $\mathbf{x}_i$, denoted by the stencil $S_{\mathbf{x}_i}=\{\mathbf{x}_{i,k}\}_{k=1}^K$.
		\STATE Compute the stencil diameter $D_{K,\max}(\mathbf{x}_i)$ and the normalized Monge coordinates $\boldsymbol{\tilde{\theta}}\left(\mathbf{x}_{i,k}\right)$ using (\ref{eqn:thetd}).
		\STATE  Construct the $K$ by $m$ Vandermonde-type matrix $\boldsymbol{\tilde{P}}$ with entries $\boldsymbol{\tilde{P}}_{kj} = p_{\boldsymbol{\alpha }(j)}\left( \boldsymbol{%
\tilde{\theta}}\left( \mathbf{x}{_{i,k}}\right) \right) = \prod_{s=1}^{d} \tilde{\theta}_s^{\alpha_s(j)}\left(\mathbf{x}_{i,k}\right)$.
		\STATE Construct the $K$ by $K$ PHS matrix $\boldsymbol{\tilde{\Phi}}$ with entries
			$\boldsymbol{\tilde{\Phi}}_{ks} = \Vert \boldsymbol{\tilde{\theta}}\left(\mathbf{x}_{i,k}\right) - \boldsymbol{\tilde{\theta}}\left(\mathbf{x}_{i,s}\right)\Vert^{2\kappa+1}$.
		\STATE Calculate the Laplacian coefficients $\{w_k\}_{k=1}^{K}$ by the formula in \eqref{eqn:normMon},%
		\begin{equation*}
(w_1,\ldots,w_K) = \frac{1}{D_{K,\max }^{2}}%
\left( \left( \Delta _{\boldsymbol{\tilde{\theta}}}\boldsymbol{\tilde{\phi}}%
_{\mathbf{x}_{0}}\right) \boldsymbol{\tilde{\Phi}}^{\dag}_{\boldsymbol{\Lambda}_K}\big(\mathbf{I}-%
\boldsymbol{\tilde{P}}\left( \boldsymbol{\tilde{P}}^{\top }\boldsymbol{%
\Lambda}_K \boldsymbol{\tilde{P}}\right) ^{-1}\boldsymbol{\tilde{P}}^{\top }\boldsymbol{%
\Lambda }_K\big)+\left( \Delta _{\boldsymbol{\tilde{\theta}}}\boldsymbol{%
\tilde{p}}_{\mathbf{x}_{0}}\right) \left( \boldsymbol{\tilde{P}}^{\top }%
\boldsymbol{\Lambda}_K \boldsymbol{\tilde{P}}\right) ^{-1}\boldsymbol{\tilde{P}}^{\top }%
\boldsymbol{\Lambda }_K\right),
		\end{equation*}
where $\boldsymbol{\tilde{\Phi}}^{\dag}_{\boldsymbol{\Lambda}_K}$ is given (\ref{eqn:PhiIn}), $\boldsymbol{%
\Lambda}_K$ is given in (\ref{eqn:k_weight}), and $\Delta _{\boldsymbol{\tilde{\theta}}}\boldsymbol{\tilde{\phi}}%
_{\mathbf{x}_{0}}$ and $\Delta _{\boldsymbol{\tilde{\theta}}}\boldsymbol{%
\tilde{p}}_{\mathbf{x}_{0}}$ are given in (\ref{eqn:Phitd}).
		\STATE Calculate the ratio $\gamma(\mathbf{x}_i) = \frac{|w_1|}{ \max_{2\leq k\leq K}|w_k |}$ to characterize the central spike pattern at $\mathbf{x}_i$.
        \STATE Increase the $K$ by a small positive integer, e.g., 2.
		\ENDWHILE
		\STATE Arrange the Laplacian coefficients $\{w_k\}_{k=1}^{K}$ into the corresponding rows and columns of $\boldsymbol{L}_{\mathbf{X}_M}$.
		\ENDFOR
		\STATE {\bf Output:} The approximate operator matrix $\boldsymbol{L}_{\mathbf{X}_M}$.
	\end{algorithmic}
	\label{algo:intrin-LB_ad}
\end{algorithm}



\subsection{Error analysis for operator approximation}

In this section, we provide the error estimate for the Laplace-Beltrami
operator approximation. Note that for quasi-uniform data, the error
estimate is written in terms of the fill distance and separation distance
(see e.g., \cite%
{Wendland2005Scat,mirzaei2012generalized,fuselier2012scattered,fuselier2013high,gross2020meshfree}%
). Here, for random data, we first write the error in terms of the stencil
diameter $D_{K,\max }$, and then relate it to the number of data points, $N$.

\subsubsection{Error estimate in terms of the stencil diameter}

The following local polynomial reproduction property is presented in terms
of the stencil diameter, in contrast to the definition using the fill distance in \cite%
{Wendland2005Scat,mirzaei2012generalized}. For
simplicity, we only consider the derivative to be the Laplace-Beltrami
operator.

\begin{lemma}
\label{lem:pol} Consider a process that defines for each stencil $S_{\mathbf{%
x}_{0}}=\left\{ \mathbf{x}_{0,k}\right\} _{k=1}^{K}$ a family of function $%
u_{k,\Delta }:M\rightarrow \mathbb{R},1\leq k\leq K$ to approximate%
\begin{equation}
\Delta _{M}f\left( \mathbf{x}\right) \approx \sum_{k=1}^{K}u_{k,\Delta }(%
\mathbf{x})f(\mathbf{x}_{0,k}):=\Delta _{M}\mathcal{I}_{p}\mathbf{f}_{{%
\mathbf{x}_{0}}}\left( \mathbf{x}\right) ,  \label{eqn:hatL}
\end{equation}%
for functions $f\in C^{l}(M)$. Then we say that the process provides a local
polynomial reproduction of degree $l$ on $M$ if there exists a constant $%
C_{1}>0$ such that

(1) $\sum_{k=1}^{K}u_{k,\Delta }(\mathbf{x})p(\mathbf{x}_{0,k})=\Delta _{M}p(%
\mathbf{x})$, \ for all $p\in \mathbb{P}_{\mathbf{x}_{0}}^{l,d},\mathbf{x}%
\in M$,

(2) $\sum_{k=1}^{K}\left\vert u_{k,\Delta }(\mathbf{x})\right\vert \leq
C_{1}D_{K,\max }^{-2}$, $\forall \mathbf{x}\in M$,

(3) $u_{k,\Delta }(\mathbf{x})=0$ if $\Vert \boldsymbol{\theta }\left(
\mathbf{x}\right) -\boldsymbol{\theta }\left( \mathbf{x}_{0,k}\right) \Vert
>D_{K,\max }$.
\end{lemma}

For the first step of our gRBF-FD, the process\ satisfies the local
polynomial reproduction if we use the GMLS interpolant (\ref{eqn:gmls}) with
the regression coefficients in (\ref{eqn:b}) (see Appendix \ref{app:A}).
Then we have the  interpolation result following from Theorem 4.2 and Theorem
4.3 of \cite{mirzaei2012generalized}.

\begin{lemma}
\label{lem:inte} Define $\mathcal{I}_{p}\mathbf{f}_{{\mathbf{x}_{0}}}\left(
\mathbf{x}\right) $ (equation (\ref{eqn:gmls})) and $\Delta _{M}\mathcal{I}%
_{p}\mathbf{f}_{{\mathbf{x}_{0}}}\left( \mathbf{x}\right) \ $(equation (\ref%
{eqn:hatL})) to be the GMLS approximation to $f\left( \mathbf{x}\right) $
and $\Delta _{M}f\left( \mathbf{x}\right) $, respectively, using local
polynomials up to degree $l$. Let $\{u_{k,\Delta }(\mathbf{x})\}$ in (\ref%
{eqn:hatL}) be a local polynomial reproduction of order $l$. Then for any $%
f\in C^{l+1}(M)$ with $l\geq 2$, there is an error bound%
\begin{eqnarray*}
\big|f(\mathbf{x})-\mathcal{I}_{p}\mathbf{f}_{{\mathbf{x}_{0}}}\left(
\mathbf{x}\right) \big| &\leq &C_{2}D_{K,\max }^{l+1} {(\mathbf{x}_0)} \big|f\big|%
_{C^{l+1}(M)}, \\
\big|\Delta _{M}f(\mathbf{x})-\Delta _{M}\mathcal{I}_{p}\mathbf{f}_{{\mathbf{%
x}_{0}}}\left( \mathbf{x}\right) \big| &\leq &\tilde{C}_{2}D_{K,\max }^{l-1}%
{(\mathbf{x}_0)} \big|f\big|_{C^{l+1}(M)},
\end{eqnarray*}%
for all $\mathbf{x}\in M$ and some $C_{2},\tilde{C}_{2}>0$. Here the
semi-norm, $|f|_{C^{l+1}(M)}:=\max_{|\boldsymbol{\alpha }|=l+1}\Vert D^{%
\boldsymbol{\alpha }}f\Vert _{L^{\infty }(M)}$, is defined over $M$, where $%
D^{\boldsymbol{\alpha }}$ denotes a general multi-dimensional derivative
with multi-index $\boldsymbol{\alpha }$.
\end{lemma}


To study the second step of gRBF-FD in (\ref{eqn:Ipph}), we next provide a local reproduction property for a PHS function.

\begin{lemma}
\label{lem:rbfwk} Let $\phi $\ be a PHS function and its interpolant (\ref%
{eqn:Ipph}) satisfies $\mathcal{I}_{\phi }:\mathbb{R}^{K}\rightarrow
C^{l_{2}}$, where $l_{2}$ denotes the smoothness of $\phi $. Consider a
process that defines for each stencil $S_{\mathbf{x}_{0}}=\left\{ \mathbf{x}%
_{0,k}\right\} _{k=1}^{K}$ a family of functions $v_{k,\Delta }:M\rightarrow
\mathbb{R},1\leq k\leq K$ to approximate%
\begin{equation}
\Delta _{M}s\left( \mathbf{x}\right) \approx \sum_{k=1}^{K}v_{k,\Delta }(%
\mathbf{x})s(\mathbf{x}_{0,k}):=\Delta _{M}\mathcal{I}_{\phi }\mathbf{s}_{{%
\mathbf{x}_{0}}}\left( \mathbf{x}\right) ,  \label{eqn:hatL2}
\end{equation}%
for functions $s\in C^{l_{2}}(M)$. Then we say that the process provides a
local PHS weak reproduction on $M$ if there exist  $C_{3},C_{4}>0$
 such that

(1) $\left\vert \Delta _{M}\mathcal{I}_{\phi }\boldsymbol{\varphi }_{{%
\mathbf{x}_{0}}}\left( \mathbf{x}\right) -\Delta _{M}\varphi (\mathbf{x}%
)\right\vert =\left\vert \sum_{k=1}^{K}v_{k,\Delta }(\mathbf{x})\varphi (%
\mathbf{x}_{0,k})-\Delta _{M}\varphi (\mathbf{x})\right\vert \leq
C_{3}\delta ^{2}D_{K,\max }^{-2}\max_{1\leq k\leq K}\left\vert \tilde{c}_{k}\right\vert$, for all $\mathbf{x}\in M$, $\varphi =\sum_{k=1}^{K}\tilde{c}%
_{k}\tilde{\phi}_{k}\in \mathrm{Span}\{\tilde{\phi}_{1},\ldots ,\tilde{\phi}%
_{K}\}$, where $\delta $ is the regularization parameter in (\ref{eqn:opphs}%
) and $\tilde{\phi}_{k}(\mathbf{x})=\phi (\Vert \boldsymbol{\tilde{\theta}}%
\left( \mathbf{x}\right) -\boldsymbol{\tilde{\theta}}\left( \mathbf{x}%
_{0,k}\right) \Vert )=\phi (\Vert \boldsymbol{\theta }\left( \mathbf{x}%
\right) -\boldsymbol{\theta }\left( \mathbf{x}_{0,k}\right) \Vert /D_{K,\max
})$ for $k=1,\ldots ,K$,

(2) $\sum_{k=1}^{K}\left\vert v_{k,\Delta }(\mathbf{x})\right\vert \leq
C_{4}D_{K,\max }^{-2}$, $\forall \mathbf{x}\in M$,

(3) $v_{k,\Delta }(\mathbf{x})=0$ if $\Vert \boldsymbol{\theta }\left(
\mathbf{x}\right) -\boldsymbol{\theta }\left( \mathbf{x}_{0,k}\right) \Vert
>D_{K,\max }$.

Here, $C_{3},C_{4}$  are independent of the stencil diameter $D_{K,\max }$, but depend on the regularization parameter $\delta$.
\end{lemma}

We partially prove the above Lemma \ref{lem:rbfwk} and examine some of the results by numerical verification. See Appendix~\ref{app:A} in detail.

\begin{thm}
\label{thm:Dk}Let $\mathcal{I}_{\phi p}\mathbf{f}_{{\mathbf{x}_{0}}}\left(
\mathbf{x}\right)$ in equation (\ref{eqn:GPRFD_inerpolant}) be the gRBF-FD
approximation to $f\left( \mathbf{x}\right) $ using local polynomials up to
degree $l$ and radial basis functions with smoothness $l_{2}\geq l$. Let $%
\{u_{k,\Delta }(\mathbf{x})\}$ in (\ref{eqn:hatL}) be a local polynomial
reproduction of order $l$ and let $\{v_{k,\Delta }(\mathbf{x})\}$ in (\ref%
{eqn:hatL2}) be a local PHS weak reproduction.\ Then for any $f\in
C^{l+1}(M) $ with $l\geq 2$, there is an error bound%
\begin{equation}
\big|\Delta _{M}f(\mathbf{x})-\Delta _{M}\mathcal{I}_{\phi p}\mathbf{f}_{{%
\mathbf{x}_{0}}}\left( \mathbf{x}\right) \big|\leq C_{5}D_{K,\max }^{l-1}{(\mathbf{x}_0)}%
\big|f\big|_{C^{l+1}(M)},  \label{eqn:Ife}
\end{equation}%
for all $\mathbf{x}\in M$ and some $C_{5}>0$.
\end{thm}

\begin{proof}

Define the residual function $s(\mathbf{x})=f(\mathbf{x})-\mathcal{I}_{p}%
\mathbf{f}_{{\mathbf{x}_{0}}}(\mathbf{x})$ for the first step using the GMLS
regression. Using the gRBF-FD interpolant (\ref{eqn:GPRFD_inerpolant}), we
have the error bound
\begin{eqnarray*}
\big|\Delta _{M}f(\mathbf{x})-\Delta _{M}\mathcal{I}_{\phi p}\mathbf{f}_{{%
\mathbf{x}_{0}}}\left( \mathbf{x}\right) \big| &=&\big|\Delta _{M}f(\mathbf{x%
})-\Delta _{M}\mathcal{I}_{p}\mathbf{f}_{{\mathbf{x}_{0}}}\left( \mathbf{x}%
\right) -\Delta _{M}\mathcal{I}_{\phi }\mathbf{s}_{{\mathbf{x}_{0}}}\left(
\mathbf{x}\right) \big| \\
&=&\big|\Delta _{M}s(\mathbf{x})-\Delta _{M}(\mathcal{I}_{\phi }\mathbf{s}_{{%
\mathbf{x}_{0}}})\left( \mathbf{x}\right) \big|,
\end{eqnarray*}%
where $\mathbf{s}_{{\mathbf{x}_{0}}}=s(\mathbf{x})|_{\mathbf{x}\in S_{%
\mathbf{x}_{0}}}\in \mathbb{R}^{K}$. Let$\ \varphi \in \mathrm{Span}\{\tilde{%
\phi}_{1},\ldots ,\tilde{\phi}_{K}\}$ be an arbitrary function spanned by
the $K$-dimensional subspace. Then the error bound becomes%
\begin{eqnarray}
&&\big|\Delta _{M}f(\mathbf{x})-\Delta _{M}\mathcal{I}_{\phi p}\mathbf{f}_{{%
\mathbf{x}_{0}}}\left( \mathbf{x}\right) \big|  \notag \\
&\leq &\big|\Delta _{M}s(\mathbf{x})-\Delta _{M}\varphi \left( \mathbf{x}%
\right) \big|+\big|\Delta _{M}\varphi (\mathbf{x})-\Delta _{M}(\mathcal{I}%
_{\phi }\boldsymbol{\varphi }_{{\mathbf{x}_{0}}})\left( \mathbf{x}\right) %
\big|+\big|\Delta _{M}(\mathcal{I}_{\phi }\boldsymbol{\varphi }_{{\mathbf{x}%
_{0}}})(\mathbf{x})-\Delta _{M}(\mathcal{I}_{\phi }\mathbf{s}_{{\mathbf{x}%
_{0}}})\left( \mathbf{x}\right) \big|  \notag \\
&\leq &\big|\Delta _{M}s(\mathbf{x})-\Delta _{M}\varphi \left( \mathbf{x}%
\right) \big|+\big|\Delta _{M}\varphi (\mathbf{x})-\Delta _{M}(\mathcal{I}%
_{\phi }\boldsymbol{\varphi }_{{\mathbf{x}_{0}}})\left( \mathbf{x}\right) %
\big|+\sum_{k=1}^{K}\left\vert v_{k,\Delta }(\mathbf{x})\right\vert |\varphi
\left( \mathbf{x}_{0,k}\right) -s(\mathbf{x}_{0,k})|  \notag \\
&\leq &\Vert \Delta _{M}s-\Delta _{M}\varphi \Vert _{L^{\infty }(\mathcal{D}%
)}+C_{3}\delta ^{2}D_{K,\max }^{-2}\max_{1\leq k\leq K}\left\vert
\tilde{c}_{k}\right\vert +C_{4}D_{K,\max }^{-2}\Vert \varphi -s\Vert _{L^{\infty }(%
\mathcal{D})},  \label{eqn:las}
\end{eqnarray}%
where $\mathcal{D}=B({\mathbf{x}_{0},}R_{K,\max }({\mathbf{x}_{0}}))$. In
the last inequality, we have used the properties in Lemma \ref{lem:rbfwk}.
Now we choose $\varphi $ to be%
\begin{equation*}
\varphi =\mathop{\textrm{arg min}}\limits_{\hat{\varphi}\in \mathrm{Span}\{%
\tilde{\phi}_{1},\ldots ,\tilde{\phi}_{K}\}}\Vert \Delta _{M}s-\Delta _{M}%
\hat{\varphi}\Vert _{L^{\infty }(\mathcal{D})}+C_{4}D_{K,\max }^{-2}\Vert
\hat{\varphi}-s\Vert _{L^{\infty }(\mathcal{D})}.
\end{equation*}%
Notice that the optimal value must be less than the objective function with $%
\hat{\varphi}=0$,%
\begin{equation}
\Vert \Delta _{M}s-\Delta _{M}\varphi \Vert _{L^{\infty }(\mathcal{D}%
)}+C_{4}D_{K,\max }^{-2}\Vert \varphi -s\Vert _{L^{\infty }(\mathcal{D}%
)}\leq \Vert \Delta _{M}s\Vert _{L^{\infty }(\mathcal{D})}+C_{4}D_{K,\max
}^{-2}\Vert s\Vert _{L^{\infty }(\mathcal{D})}\lesssim D_{K,\max }^{l-1}\big|%
f\big|_{C^{l+1}(M)}  \label{eqn:ins}
\end{equation}%
where the results in Lemma \ref{lem:inte} have been used.

Next we estimate the second error term of (\ref{eqn:las}). From (\ref%
{eqn:ins}), we see that
\begin{equation*}
\Vert \varphi -s\Vert _{L^{\infty }(\mathcal{D})}\lesssim D_{K,\max }^{l+1}%
\big|f\big|_{C^{l+1}(M)}.
\end{equation*}%
Since $\Vert s\Vert _{L^{\infty }(\mathcal{D})}\lesssim D_{K,\max }^{l+1}%
\big|f\big|_{C^{l+1}(M)}$\ using Lemma \ref{lem:inte}, we have that $\Vert
\varphi \Vert _{L^{\infty }(\mathcal{D})}\lesssim D_{K,\max }^{l+1}\big|f%
\big|_{C^{l+1}(M)}$ and each $\left\vert \tilde{c}_{k}\right\vert \lesssim
D_{K,\max }^{l+1}\big|f\big|_{C^{l+1}(M)}$. Then the second term in (\ref%
{eqn:las}) can be bounded by%
\begin{equation}
C_{3}\delta ^{2}D_{K,\max }^{-2}\max_{1\leq k\leq K}\left\vert \tilde{c}%
_{k}\right\vert \lesssim \delta ^{2}D_{K,\max }^{l-1}\big|f\big|%
_{C^{l+1}(M)}.  \label{eqn:sec2}
\end{equation}

Substituting (\ref{eqn:ins}) and (\ref{eqn:sec2}) into (\ref{eqn:las}), we
arrive at the desired error bound in (\ref{eqn:Ife}). In particular, if the
second error term (\ref{eqn:sec2}) is negligible, then the error bound
becomes
\begin{equation*}
\big|\Delta _{M}f(\mathbf{x})-\Delta _{M}\mathcal{I}_{\phi p}\mathbf{f}_{{%
\mathbf{x}_{0}}}\left( \mathbf{x}\right) \big|\leq \inf_{\hat{\varphi}\in
\mathrm{Span}\{\tilde{\phi}_{1},\ldots ,\tilde{\phi}_{K}\}}\left\{ \Vert
\Delta _{M}s-\Delta _{M}\hat{\varphi}\Vert _{L^{\infty }(\mathcal{D}%
)}+C_{4}D_{K,\max }^{-2}\Vert \hat{\varphi}-s\Vert _{L^{\infty }(\mathcal{D}%
)}\right\} \leq C_{5}D_{K,\max }^{l-1}\big|f\big|_{C^{l+1}(M)}.
\end{equation*}

\end{proof}

\comment{%

does
not
allow
us
to
use
the
concept
of
native
space,
or
the
interpolation
error
in
Chap
of
Wendland book.%
}


\subsubsection{Error estimate in terms of the number of data}

We now estimate the error bound in terms of the number of data.

\begin{lemma}
\label{lem:hK} For some integer $K$ and define the stencil diameter $%
D_{K,\max }\left( \mathbf{x}_{0}\right) $\ and stencil radius $R_{K,\max
}\left( \mathbf{x}_{0}\right) $ as in Definition \ref{def:srad} for each
base point $\mathbf{x}_{0}$. Assume that there exists a constant $C_{K}>0$\
such that $R_{K,\max }({\mathbf{x}_{i}})\leq D_{K,\max }({\mathbf{x}_{i}}%
)\leq C_{K}R_{K,\max }({\mathbf{x}_{i}})$ for all stencils $\{S_{\mathbf{x}%
_{i}}\}_{i=1}^{N}$. Let ${\mathbf{x}_{1},\ldots ,\mathbf{x}_{N}}$\ be i.i.d.
random samples from the distribution $Q$ with the density $q\in L^{1}(M)$
such that $\int_{M}q(\mathbf{x})dV(\mathbf{x})=1$\ and $q\geq q_{\min }>0$
for some positive $q_{\min }$. Then, we have for each ${\mathbf{x}_{i}}$ and all $\delta>0$,%
\begin{equation*}
\mathbb{P}_{\mathbf{X}_{M}\sim Q}(R_{K,\max }{(\mathbf{x}_i)}>\delta )\leq \exp (-q_{\min
}C_{d}(N-K)\delta ^{d}),
\end{equation*}%
where $C_{d}$ is a constant depending on the dimension $d$. Moreover, with
probability higher than $1-\frac{1}{N}$, we have
\begin{equation*}
R_{K,\max }({\mathbf{x}_{i}})=O\left( \left( \frac{\log N}{N}\right) ^{\frac{1}{d}}\right) ,%
\text{ \ \ }D_{K,\max }({\mathbf{x}_{i}})=O\left( \left( \frac{\log N}{N}\right) ^{\frac{1}{d}%
}\right) ,
\end{equation*}%
where the constant in the big-oh is independent of $N$. Subsequently, we
also have
\begin{equation*}
ED_{K,\max }({\mathbf{x}_{i}})=O\left( \left( \frac{\log N}{N}\right) ^{%
\frac{1}{d}}\right) ,\text{ \ \ }\sigma \left( D_{K,\max }({\mathbf{x}_{i}}%
)\right) =O\left( \left( \frac{\log N}{N}\right) ^{\frac{1}{d}}\right) ,
\end{equation*}%
where the expectation $E$ and the standard deviation $\sigma $ are taken
with respect to the distribution of $D_{K,\max }({\mathbf{x}_{i}})$ at a
fixed ${\mathbf{x}_{i}}$ across different samplings of the point cloud data $%
\left\{ \mathbf{x}_{i}\right\} _{i=1}^{N}\subset M$.
\end{lemma}

\begin{proof}
See Appendix \ref{app:B}.
\end{proof}

Using the results in Theorem \ref{thm:Dk} and Lemma \ref{lem:hK}, we arrive
at the following the error bound in terms of the number of data.

\begin{thm}\label{thm:fee}
Suppose that the conditions and assumptions in Theorem \ref{thm:Dk} and
Lemma \ref{lem:hK} hold. Then for any $f\in C^{l+1}(M)$ with $l\geq 2$,
there is an error bound%
\begin{equation}
\big|\Delta _{M}f(\mathbf{x})-\Delta _{M}\mathcal{I}_{\phi p}\mathbf{f}_{{%
\mathbf{x}_{0}}}\left( \mathbf{x}\right) \big|=O\left( \left( \frac{\log N}{N%
}\right) ^{\frac{l-1}{d}}\right) ,  \label{eqn:bodO}
\end{equation}%
for all $\mathbf{x}\in M$, where the constant in the big-oh error bound is
independent of $N$.
\end{thm}



\section{Numerical experiments}

\label{sec:numerical}

To support the performance of gRBF-FD, we present numerical
results for solving the screened Poisson equations in \eqref{eqn:poisson} across various manifolds, identified by randomly sampled point clouds.
We also compare with the results using the GMLS method.
For all examples below, the manifolds are assumed to be known with analytic tangent planes. We fix the PHS parameter to be $\kappa=3$. For both GMLS and gRBF-FD, we apply the auto-tuned $K$-nearest neighbors  starting from an initial value $K_0$,  as detailed in
Section \ref{sec:autoK} and Algorithm \ref{algo:intrin-LB_ad}.

This section is organized as follows. In Sections \ref{sec:rbc} and \ref{sec:bsp}, we examine the numerical performance on two 2D smooth surfaces: a red blood cell (RBC) and a bumpy sphere (BSP), respectively. In Section \ref{sec:flattori},
we report the numerical results on two higher-dimensional manifolds, a 3D flat torus
embedded in $\mathbb{R}^{12}$ and a 4D flat torus embedded in $\mathbb{R}^{16}$.
The following numerical results demonstrate that our gRBF-FD approach can persistently provide stable and convergent
solutions with smaller errors than GMLS across multiple trials of randomly sampled data points.


\subsection{Red blood cell}

\label{sec:rbc} Consider a red blood cell (RBC) with the parametrization:
\begin{equation}
\mathbf{x} = \left(r\cos\theta\cos\phi,r\cos\theta\sin\phi, \frac{1}{2}%
\sin\theta \left( c_0+c_2\cos^2\theta+c_4\cos^4\theta\right) \right),
\label{eqn:rbc}
\end{equation}
where $-\pi/2 \leq \theta \leq \pi/2$, $-\pi \leq \phi < \pi$, $r=3.91/3.39$%
, $c_0=0.81/3.39$, $c_2=7.83/3.39$ and $c_4=-4.39/3.39$, which is the same
as the RBC in \cite{fuselier2013high}. In our numerical experiment, we set
the true solution to be $f = \cos^2 (\theta)$ and then we can calculate the RHS $h:=\left( 1-\Delta _{M}\right)
f$. Next, we approximate the numerical solution  for the PDE problem  subjected
to the manufactured $h$.
Numerically, the points $%
\{\mathbf{x}_i\}_{i=1}^N$ are generated from the parametrization %
\eqref{eqn:rbc} using randomly sampled $\{\theta_i,\phi_i\}_{i=1}^N$.


In Fig. \ref{fig:rbc}, we plot the \textbf{FE}s and \textbf{IE}s
as a function of $N$ over 4 independent trials. It can be seen that the \textbf{%
FE}s  decrease on the order of $N^{-(l-1)/2}$ for both
GMLS and gRBF-FD. This error rate is in good agreement with the theoretical one in \eqref{eqn:bodO}. It can be further observed that the \textbf{IE}s for
$l=2$ and $l=4$ decay with respective rates $N^{-1}$ and $N^{-2}$, which are
half an order faster than the corresponding \textbf{FE}s. This reveals that the
\textbf{IE}s possess the super-convergence phenomenon, as also observed in
\cite{liang2013solving,jiang2024generalized,li2024generalized}. Moreover,
it can be seen that the \textbf{IE}s of gRBF-FD are smaller than those of GMLS by  factors of approximately $%
0.3$ and $0.5$ for $l=2$ and $l=4$, respectively.
In particular, for $l=4$,
gRBF-FD achieves smaller \textbf{IE}s
than GMLS even though their \textbf{FE}s are comparable.

\begin{figure*}[htbp]
\centering
\begin{tabular}{cc}
{(a) \textbf{FE}} & {(b) \textbf{IE} } \\
\includegraphics[width=2.8
		in, height=2.3 in]{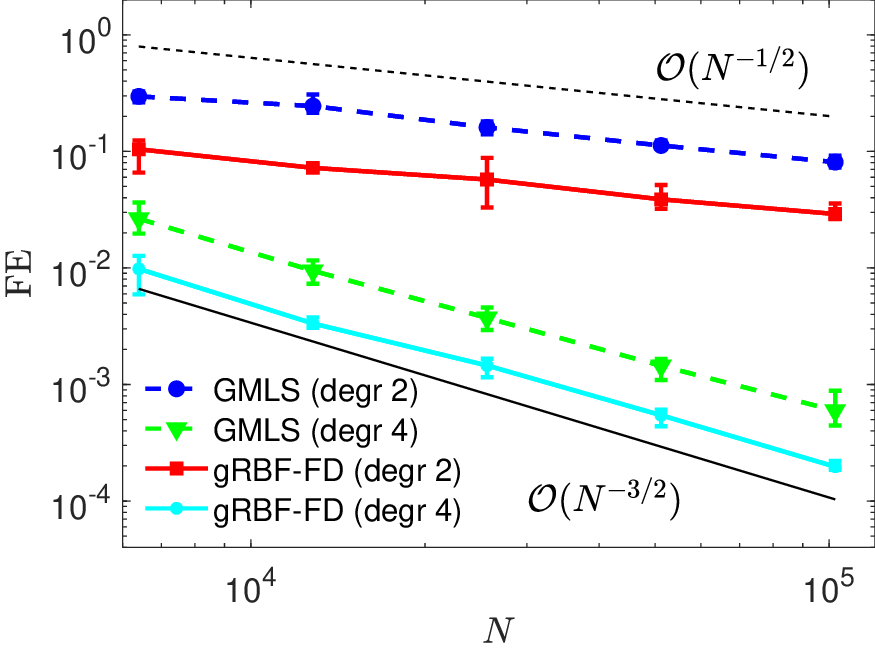} &
\includegraphics[width=2.8
		in, height=2.3 in]{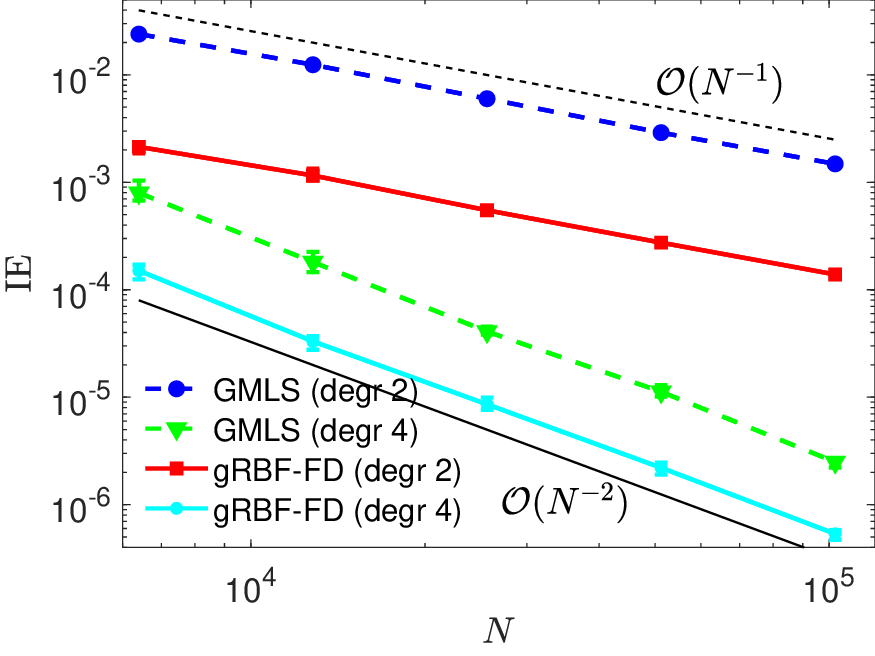}%
\end{tabular}%
\caption{\textbf{2D RBC in} $\mathbb{R}^3$ using random data. Shown are (a)
\textbf{FE}s vs. the number of points $N$ and (b) \textbf{IE}s vs. $N$. We show the results using different polynomial
degrees $l=2,4$. We fix the PHS parameter $\protect\kappa=3$ and the initial $K_0=40$.}
\label{fig:rbc}
\end{figure*}

\subsection{Bumpy sphere}
\label{sec:bsp}
We consider a smooth but complex surface, bumpy sphere (BSP), parameterized by
	\begin{equation}
		\textbf{x} = (x^1,x^2,x^3) = \left(r(\theta,\phi)\sin\theta\cos\phi,r(\theta,\phi)\sin\theta\sin\phi, r(\theta,\phi)\cos\theta \right),
		\label{eqn:psp}
	\end{equation}
	where $0 \leq \theta \leq \pi$, $0 \leq \phi < 2\pi$ and $r(\theta,\phi) = 1 + 0.1(\sin^7 4\theta)(\sin4\phi)$. This surface exhibits $C^6$ smoothness at the two poles and $C^\infty$ smoothness elsewhere. Its geometric complexity arises from the rapidly varying curvature in some regions (see Fig. \ref{fig:bsp}(a)). In our numerical experiment, we set the true solution to be the third coordinate of the surface $f=x^3$ and then construct the RHS $h := \left( 1-\Delta _{M}\right)
	f$. The points are generated randomly from the parameterization \eqref{eqn:psp} in intrinsic coordinates, where $(\theta,\phi)$ are i.i.d. sampled with the density $p(\theta,\phi) = \frac{\sin \theta}{4\pi}$ on $\left[0,\pi\right]\times \left[0,2\pi\right)$.

	
	In Fig. \ref{fig:bsp} (a), we show the surface of bumpy sphere as well as the true solution with its value color coded. In panels (b)(c), we plot \textbf{FE}s and \textbf{IE}s  as functions of $N$ over 4 independent trials. It can be seen that  \textbf{FE}s for both GMLS and gRBF-FD decay on the order of $N^{-(l-1)/2}$, which agrees with the theoretical prediction. Again, the \textbf{IE}s still decay half an order faster than the \textbf{FE}s. Moreover, gRBF-FD achieves smaller errors than GMLS in both operator and solution approximation.
	
	\begin{figure*}[htbp]
		\centering
		\begin{tabular}{ccc}
			{(a) Bumpy sphere} & {(b)\textbf{FE}} & {(c) \textbf{IE}} \\
			\includegraphics[width=2.35
			in, height=1.7 in]{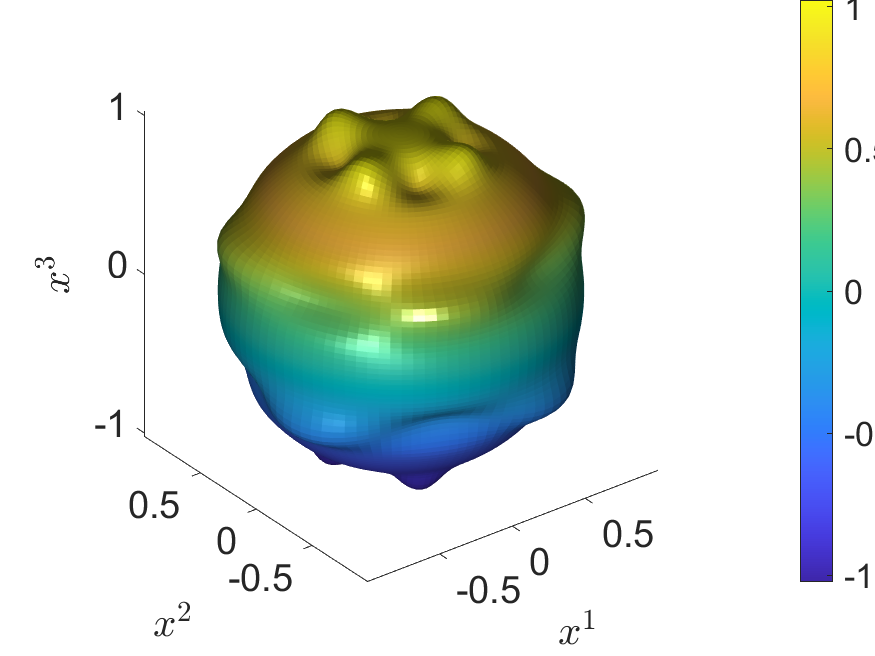} &
			\includegraphics[width=1.8
			in, height=1.7 in]{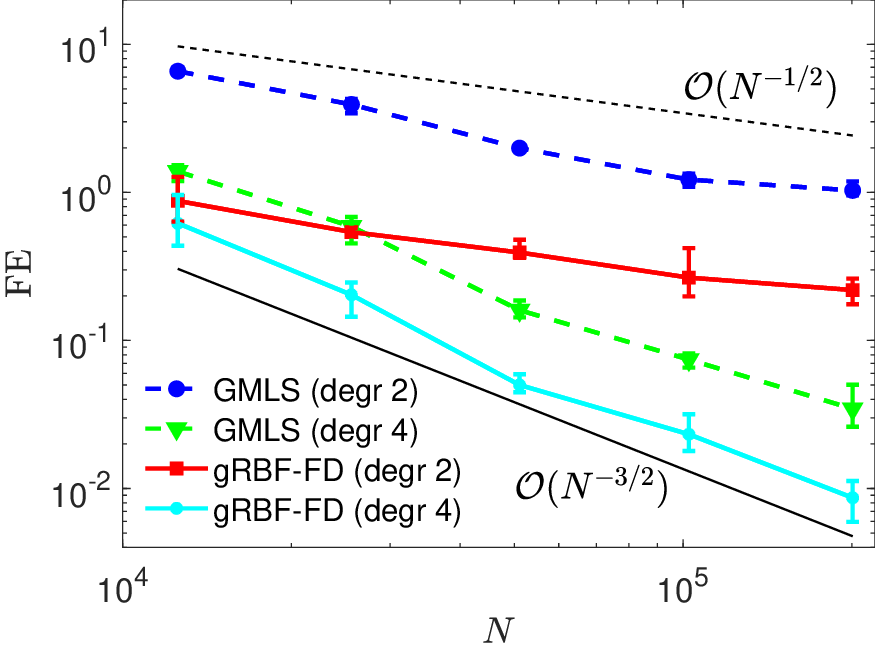} &
			\includegraphics[width=1.8
			in, height=1.7 in]{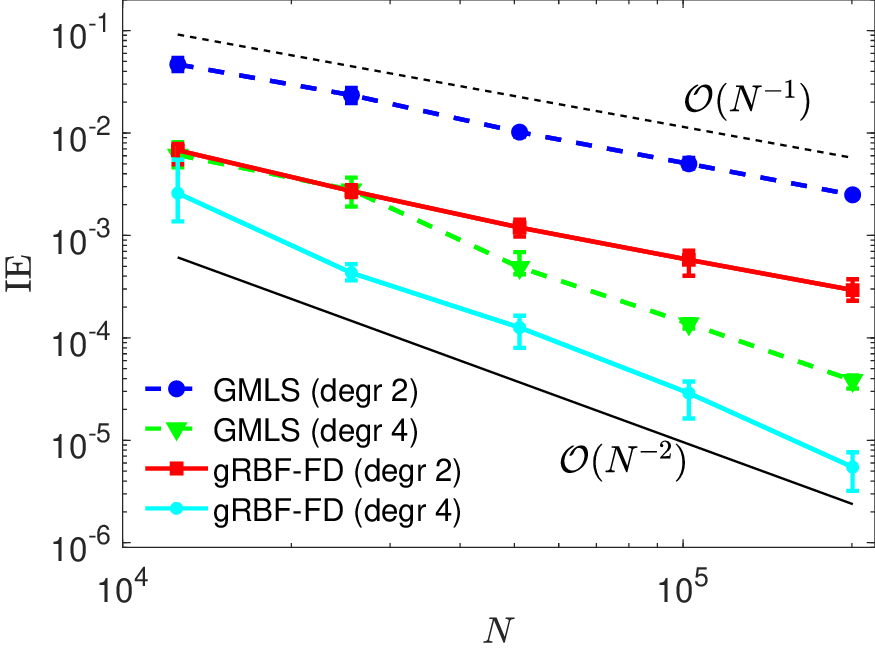}%
		\end{tabular}%
		\caption{\textbf{2D BSP in} $\mathbb{R}^3$ using random data. (a) The true solution $f=x^3$ with its value color coded. Shown are
			(b) \textbf{FE}s vs. $N$ and
			(c) \textbf{IE}s vs. $N$. In both panels (b) and (c), we show
			the results with polynomial degree $l=2,4$. We fix the PHS parameter $\protect\kappa=3$ and the initial $K_0=40$.}
		\label{fig:bsp}
	\end{figure*}

\subsection{Flat tori}

\label{sec:flattori}

We consider a $3$-dimensional flat torus embedded in $\mathbb{R}^{12}$ with
the parameterization,
\begin{equation}  \label{eqn:flattorus}
\mathbf{x}=\frac{1}{\sqrt{1^2 +2^{2}}}\left(
\begin{array}{cccc}
\cos (\phi_{1}), & \sin (\phi_{1}), & \cos (2\phi_1), & \sin (2\phi_{1}), \\
\cos (\phi_{2}), & \sin (\phi_{2}), & \cos (2\phi_2), & \sin (2\phi_{2}), \\
\cos (\phi_{3}), & \sin (\phi_{3}), & \cos (2\phi_3), & \sin (2\phi_{3})%
\end{array}
\right) ,
\end{equation}%
with $0\leq \phi_{1},\phi_2,\phi_3< 2\pi$. The Riemannian metric is given by
a $3\times 3$ identity matrix $\mathbf{I}_{3}$. The true solution $f$ is set
to be $f =\sin (\phi_1) \sin (\phi_2) \sin (\phi_3) $. Numerically, the
points $\{\mathbf{x}_i\}_{i=1}^N$ are generated from the parametrization %
\eqref{eqn:flattorus} using randomly sampled $\{\phi_{1},\phi_{2},\phi_{3}\}_{i=1}^N$ with the
uniform distribution on $\left[0,2\pi\right)^3$.
We use the degree $l=2,4$ and the initial $K_0=60$ nearest neighbors for illustrating the
convergence of solutions.



As shown in Fig. \ref{fig:3dflat}(a),  the \textbf{FE}s  decay at the rate  $N^{-(l-1)/3}$ for both GMLS and gRBF-FD, which is consistent
with the theory in \eqref{eqn:bodO} for dimension $d=3$. As shown in Fig. \ref{fig:3dflat}(b), the solutions again
show super-convergence, that is, \textbf{IE}s decay with  respective rates $N^{-2/3}$ and $N^{-4/3}$ for $l=2$ and $%
l=4$. Moreover, the \textbf{IE}s of gRBF-FD are approximately $30\%$ and $%
25\%$ of the GMLS errors for $l=2$ and $l=4$, respectively.

\begin{figure*}[htbp]
\centering
\begin{tabular}{cc}
{(a) \textbf{FE}} & {(b) \textbf{IE} } \\
\includegraphics[width=2.8
		in, height=2.3 in]{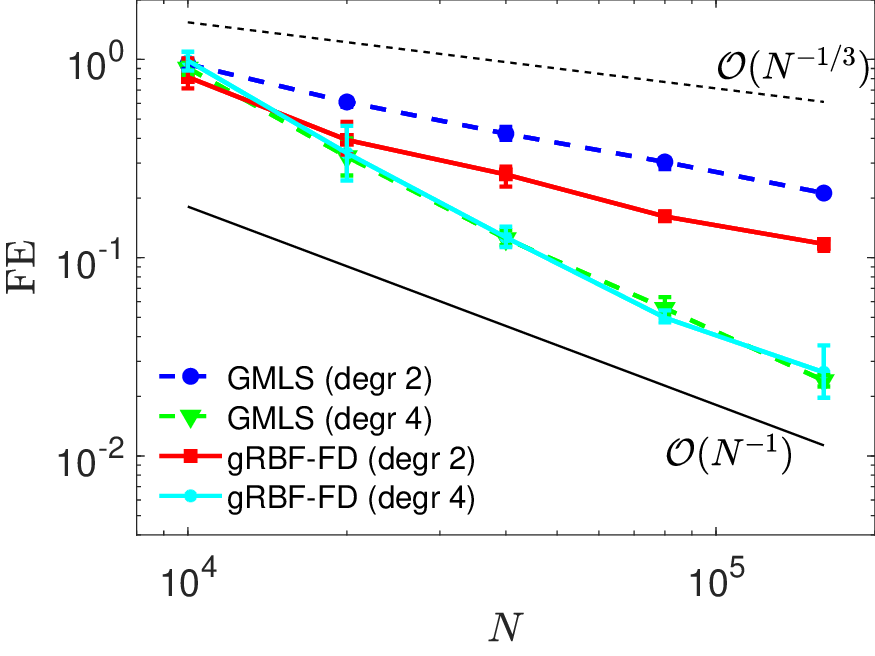} &
\includegraphics[width=2.8
		in, height=2.3 in]{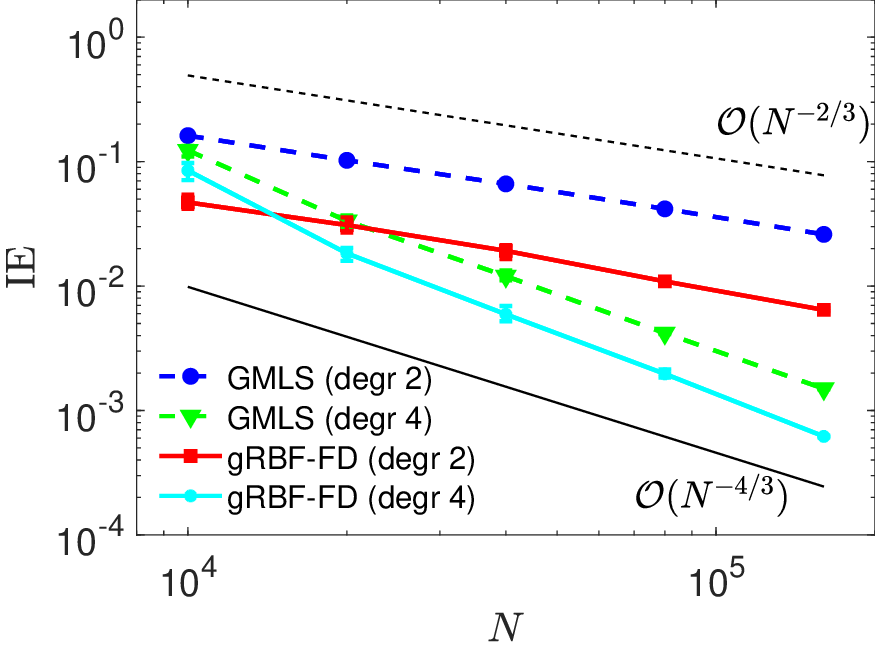}%
\end{tabular}%
\caption{\textbf{3D flat torus in} $\mathbb{R}^{12}$ using random data.
Shown are (a) \textbf{FE}s vs. $N$ and (b) \textbf{IE}s
vs. $N$. In both panels (a) and (b), we show the results with
polynomial degree $l=2,4$. We fix the PHS parameter $\protect\kappa=3$ and the initial $%
K_0=60$.}
\label{fig:3dflat}
\end{figure*}

We further consider a $4$-dimensional flat torus embedded in $\mathbb{R}^{16}$
with the parameterization,
\begin{equation}  \label{eqn:flattorus2}
\mathbf{x}=\frac{1}{\sqrt{1^2 +2^{2}}}\left(
\begin{array}{cccc}
\cos (\phi_{1}), & \sin (\phi_{1}), & \cos (2\phi_1), & \sin (2\phi_{1}), \\
\cos (\phi_{2}), & \sin (\phi_{2}), & \cos (2\phi_2), & \sin (2\phi_{2}), \\
\cos (\phi_{3}), & \sin (\phi_{3}), & \cos (2\phi_3), & \sin (2\phi_{3}), \\
\cos (\phi_{4}), & \sin (\phi_{4}), & \cos (2\phi_4), & \sin (2\phi_{4}),%
\end{array}
\right) ,
\end{equation}%
where $0\leq \phi_{1},\phi_2,\phi_3, \phi_{4}< 2\pi$. The Riemannian metric
is given by a $4\times 4$ identity matrix $\mathbf{I}_{4}$. The true
solution $f$ is set to be $f =\sin (\phi_1) \sin (\phi_2) \sin (\phi_3) \sin
(\phi_{4}) $.
We did not test with polynomial degree $l=4$ due to the large number of monomial basis functions  required compared to the 2D examples above.
We use the polynomial degree $l=3$ and the initial stencil size $K_0=75$ nearest neighbors for illustrating the
convergence of solutions. It can be seen from Fig. \ref{fig:4dflat} that \textbf{FE}s and \textbf{IE}s both decay with the rate $N^{-1/2}$, which agrees with the theory in \eqref{eqn:bodO}.
The \textbf{IE} of gRBF-FD is still smaller than that of GMLS, despite its larger \textbf{FE}.

Moreover, we numerically examine the stability of both approaches  when $N=40000$. In our implementation, we use the command \texttt{eigs($\boldsymbol{L}_{\mathbf{X}_M}$,6,10)} in MATLAB to compute the leading six eigenvalues close to 10. Then the leading eigenvalues of GMLS are 0, -1.30, -1.30, -1.30, -1.30, -1.30 while the leading ones of gRBF-FD are 0, -1.21, -1.21, -1.21, -1.21, -1.21. No spurious eigenvalues are observed in the right half of the complex plane. These numerical eigenvalues are approximations of the analytic true eigenvalues of the 4D flat torus, 0, -1, -1, -1, -1, -1 (see e.g., \cite{harlim2023radial}).





\begin{figure*}[htbp]
\centering
\begin{tabular}{cc}
{(a) \textbf{FE}} & {(b) \textbf{IE} } \\
\includegraphics[width=2.8
		in, height=2.3 in]{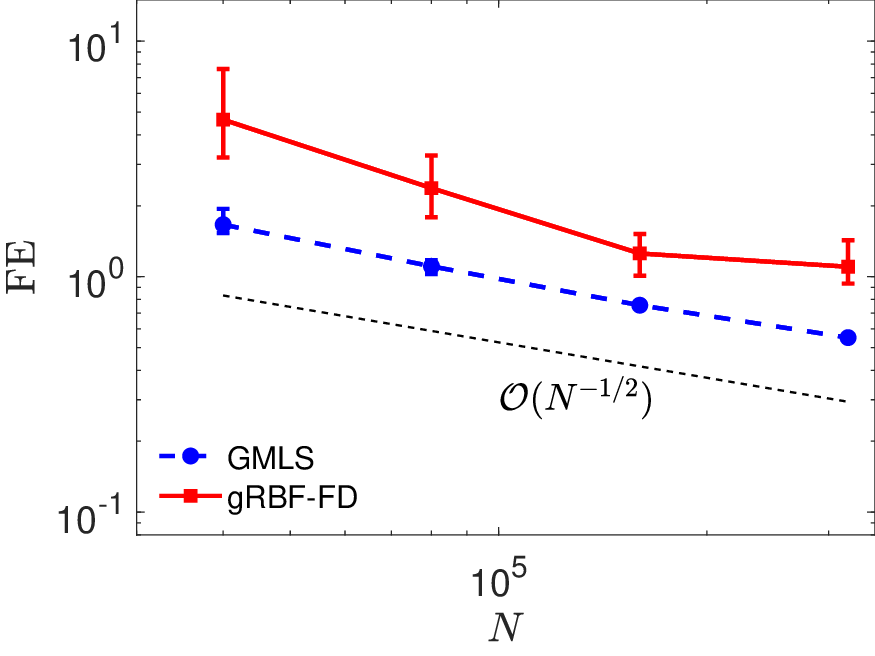} &
\includegraphics[width=2.8
		in, height=2.3 in]{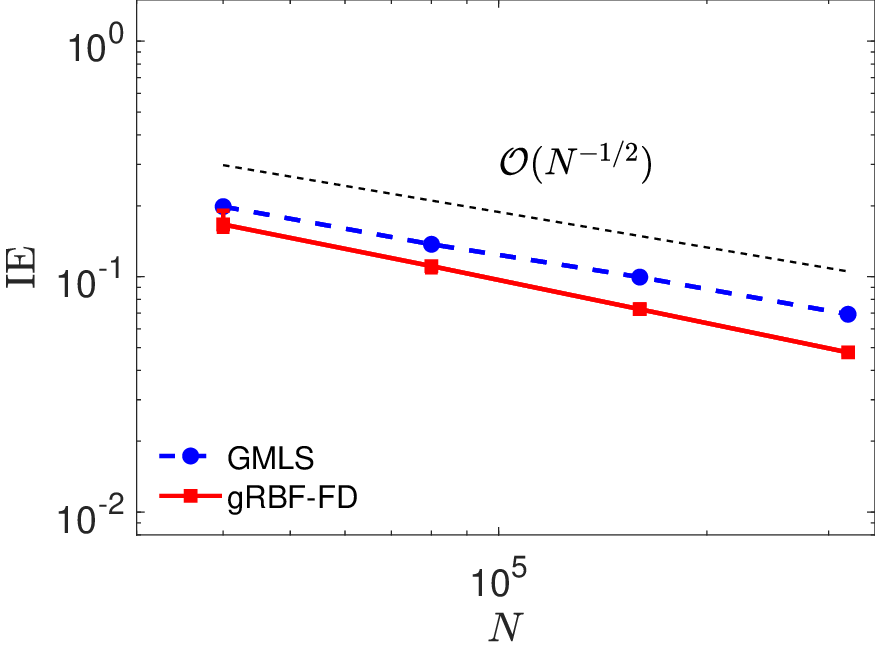}%
\end{tabular}%
\caption{\textbf{4D flat torus in} $\mathbb{R}^{16}$ using random data.
Shown are (a) \textbf{FE}s vs.  $N$ and (b) \textbf{IE}s
vs. $N$. In both panels (a) and (b), we fix the polynomial degree $l=3$, the PHS
parameter $\protect\kappa=3$ and the initial stencil size $K_0=75$.}
\label{fig:4dflat}
\end{figure*}

\subsection{Bunny and Armadillo models}


We now consider solving the elliptic problems on the Bunny and Armadillo models which are both two-dimensional surfaces
embedded in $\mathbb{R}^{3}$. The data of these two models are downloaded from the Stanford 3D Scanning Repository \cite{Stanford3d}.
The original dataset of the Bunny comprises a triangle mesh with 34,817 vertices and the Armadillo comprises a triangle mesh with 172,974 vertices.
To resolve the singular regions of the original dataset, we first generate new meshes of the surfaces using the Marching Cubes algorithm \cite{Marchingcubes} that is implemented  in MeshLab \cite{MeshLab}. Notice that the Marching Cubes algorithm does not smooth the surface.
We will compare the solutions among surface finite element method (FEM), GMLS, and gRBF-FD on smoothed surfaces for both Bunny and Armadillo models. Then, we  use the Screened Poisson surface reconstruction algorithm to smooth the surfaces that fit the point clouds of models. Next, we apply the Poisson-disk sampling algorithm
via the Meshlab to generate point clouds over the surfaces. Subsequently, new meshes could be obtained by using
the Marching Cubes algorithm again. As a result, we generate a point cloud of $N = 34,596$ points over the Stanford Bunny model and generate $N = 172,974$ points over the Stanford Armadillo model. For the Armadillo model, we  normalize its dataset into a unit box $[0,1]^3$ for  computational convenience. 
Finally, after all above processes, we use MeshLab to clean the vertices and meshes of the two models for further application in FEM.

\begin{figure*}[htbp]
\centering
\begin{tabular}{ccc}
{(a) FEM solution} & {(b) GMLS } & {(c) gRBF-FD} \\
\includegraphics[width=2.1
				in, height=1.5 in]{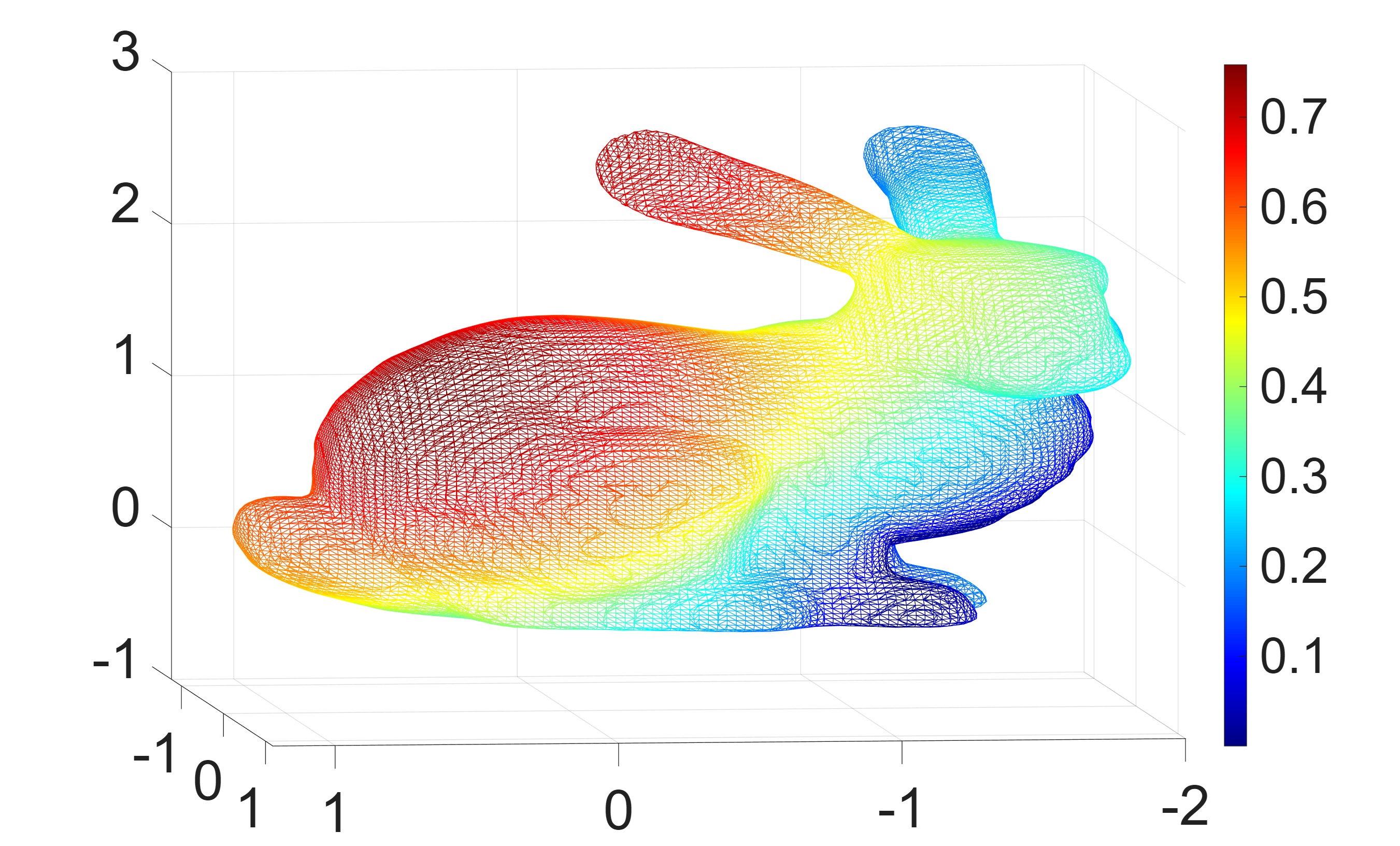} &
\includegraphics[width=2.1
				in, height=1.5 in]{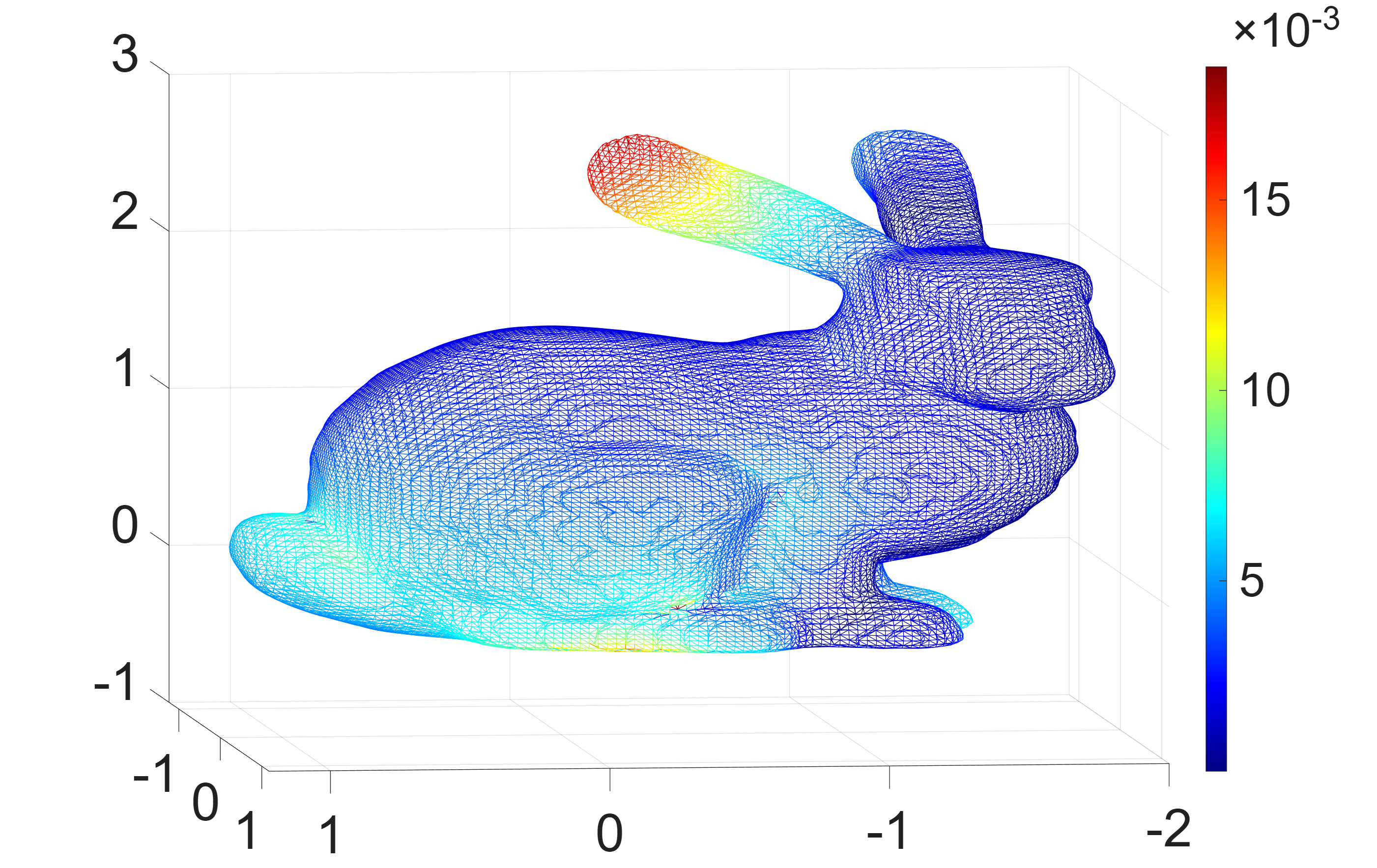} &
\includegraphics[width=2.1
				in, height=1.5 in]{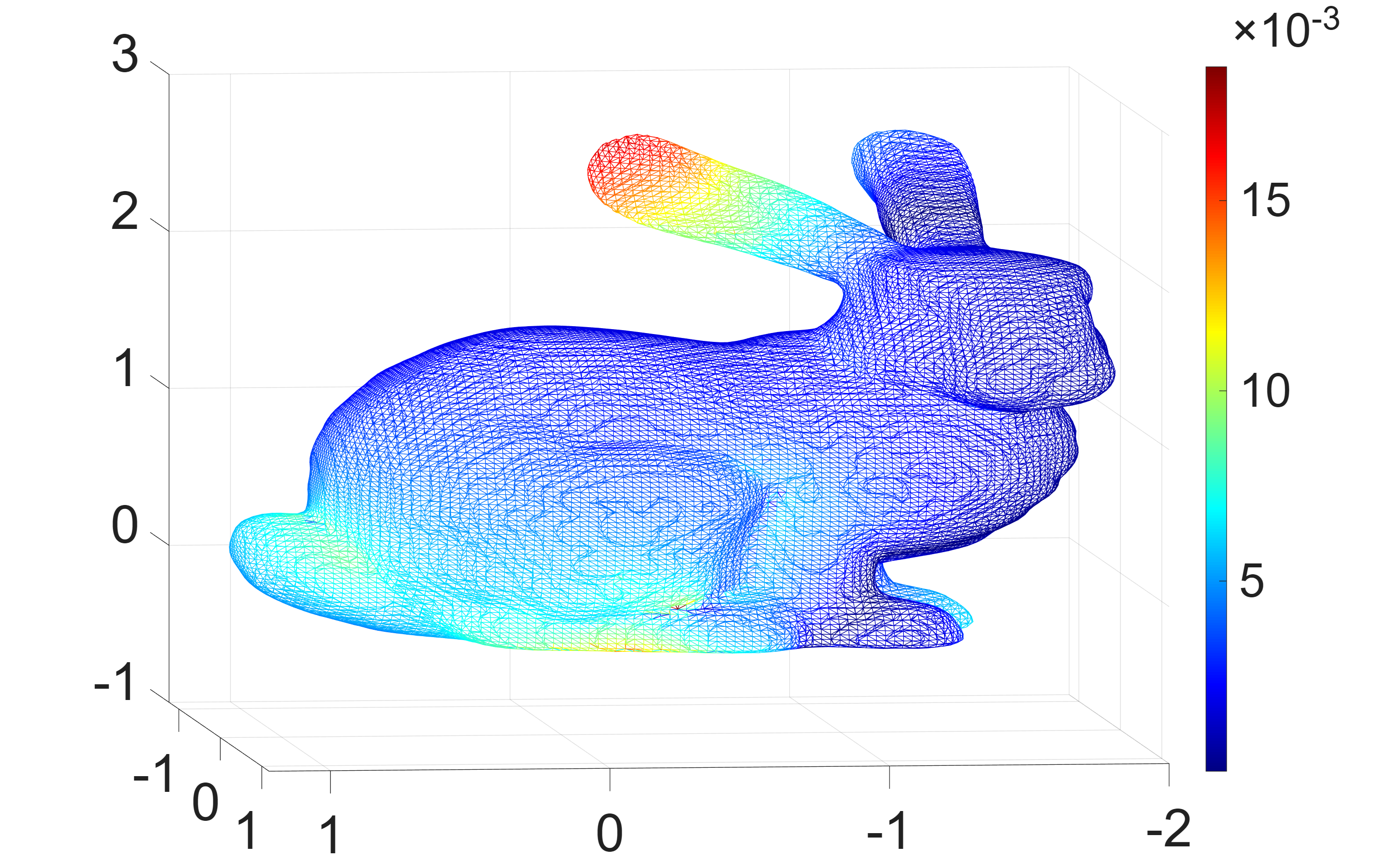}
\end{tabular}%
\caption{\textbf{Bunny example.} $N = 34,596$. (a) FEM solution as a reference. (b) Pointwise absolute difference  between FEM and GMLS with degree $l = 4$ and $K_0 = 41$-nearest neighbors.
(c) Pointwise absolute difference between FEM and gRBF-FD with $l = 4$ and $K_0 = 41$. GMLS and gRBF-FD both have maximum norm differences of 0.018 and relative differences of 2.4\%. }
\label{fig:bunny}
\end{figure*}

We solve the PDE problem in \eqref{eqn:poisson} with $a = 1$ and $f = 0.6(x_{1} + x_{2} + x_{3})$ for both models. Here we have no access to the analytic
true solutions due to the unknown embedding functions. For comparison, we take the FEM solution obtained
from the FELICITY FEM Matlab toolbox \cite{Felicity} as the reference (see Fig. \ref{fig:bunny}(a) for Bunny and Fig. \ref{fig:armadillo}(a) for Armadillo). For GMLS and gRBF-FD, the used parameters can be found in the captions of Figs. \ref{fig:bunny} and \ref{fig:armadillo}. For Bunny, it can be seen from Figs.~\ref{fig:bunny}(b) and (c) that GMLS and gRBF-FD show comparable results, both with maximum norm differences of 0.018 and relative differences of 2.4\%.  For Armadillo, it can be seen from Figs.~\ref{fig:armadillo}(b) and (c) that the absolute difference between the FEM  and the gRBF-FD is slightly smaller that of FEM and GMLS. For GMLS, the maximum norm difference is $7.4\times 10^{-4}$ and the relative difference is 0.41\%.  For gRBF-FD, the maximum norm difference is $5.9\times 10^{-4}$ and the relative difference is 0.33\%.

\begin{figure*}[htbp]
\centering
\begin{tabular}{ccc}
{(a) FEM solution} & {(b) GMLS } & {(c) gRBF-FD } \\
\includegraphics[width=2.1
				in, height=1.5 in]{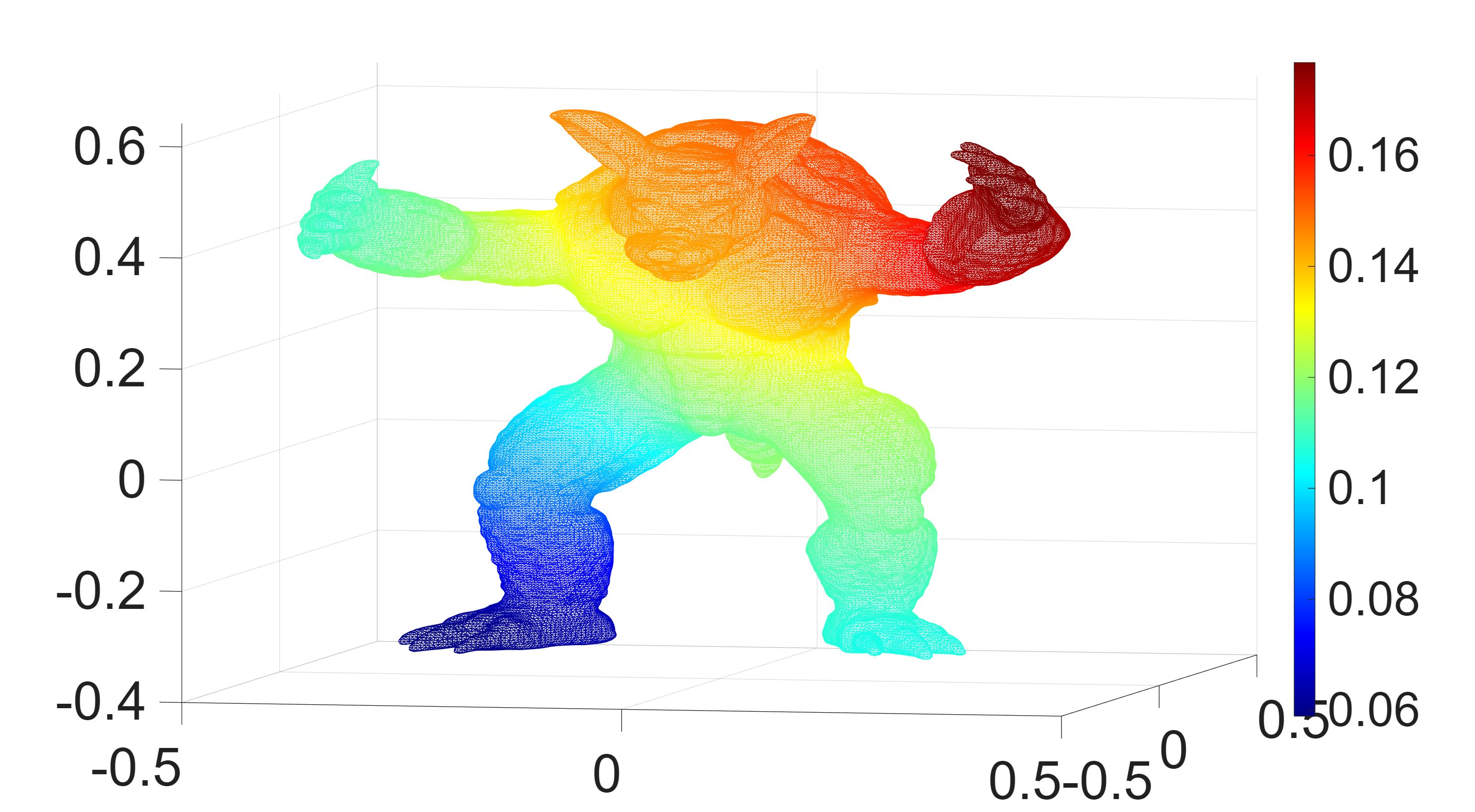} &
\includegraphics[width=2.1
				in, height=1.5 in]{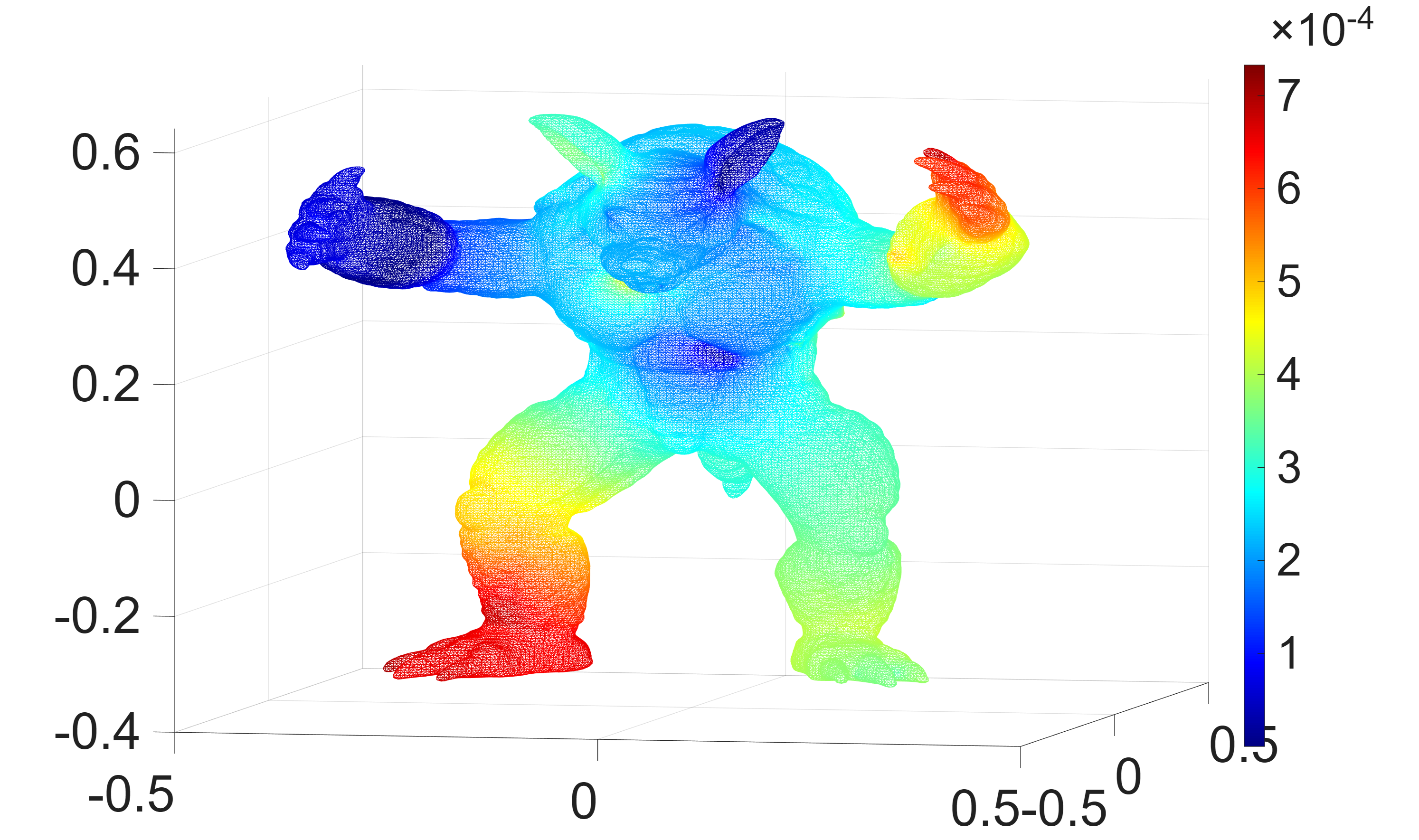} &
\includegraphics[width=2.1
				in, height=1.5 in]{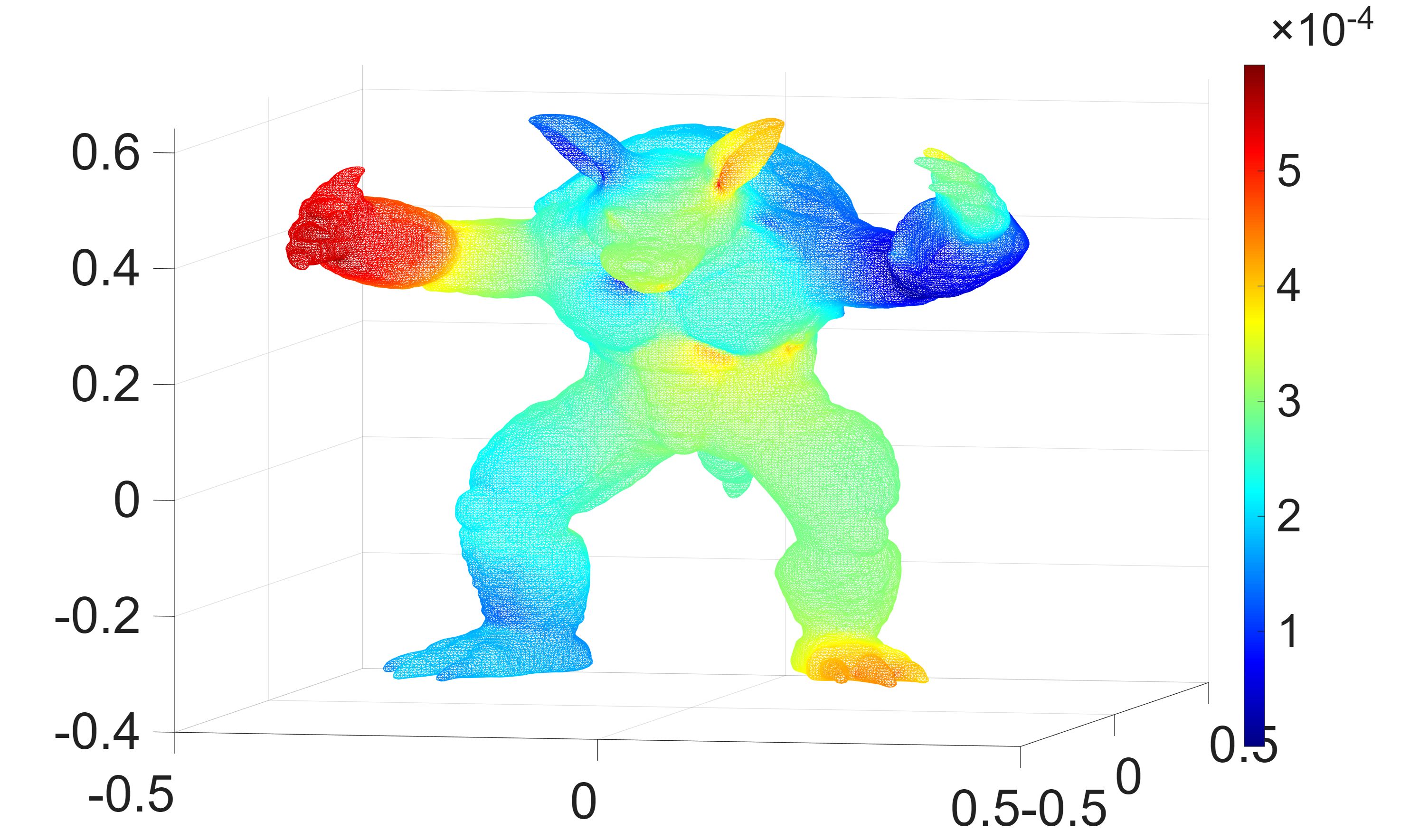}%
\end{tabular}%
\caption{\textbf{Armadillo example}. $N = 172,974$. (a) FEM solution as a reference. (b) Pointwise absolute difference between FEM and GMLS with degree $l = 4$ and $K_0 = 41$-nearest neighbors. GMLS has the maximum norm difference of $7.4\times 10^{-4}$ and the relative difference of 0.41\%.
(c) Pointwise absolute difference between FEM and gRBF-FD with $l = 4$ and $K_0 = 41$.
GRBF-FD has the maximum norm difference of $5.9\times 10^{-4}$ and the relative difference of 0.33\%.}
\label{fig:armadillo}
\end{figure*}

\section{Conclusions}

\label{sec:conclusion}


In this paper, we considered a two-step generalized RBF-FD approach for solving screened Poisson problems
on smooth manifolds, identified by randomly sampled point clouds. In the first step, we applied a GMLS regression
to capture the smooth, leading component in the Taylor's expansion of the target function $f$. In the second step, we employed a PHS interpolation to handle the residual from the GMLS regression, which contains the high-order remainder of the Taylor's expansion.
We established the error bound of gRBF-FD in Theorem~\ref{thm:fee} for the consistency analysis of the Laplace-Beltrami operator, demonstrating that the method can achieve   arbitrary-order algebraic accuracy on smooth manifolds.
For the stability of the Laplacian matrix, we employed both GMLS regression and PHS interpolation  within a weighted norm which assigns a larger weight  to the base point.
Moreover, we proposed an automated method of tuning the $K$-nearest neighbors, thereby
preventing the instability caused by a  small $K$.
Once $K$ is auto-tuned, the resulting  Laplacian matrix becomes
nearly diagonally dominant,  leading to  a numerically stable approximation.
Combining the weighted norm and auto-tuned $K$ technique,
we provided supporting numerical examples to verify the convergence of solutions.
Numerically, our gRBF-FD approach performs well for  i.i.d. randomly sampled data on manifolds---outperforming the GMLS approach.


Several open questions remain for future investigation. First, our study used the PHS function in the second step of the gRBF-FD approach to enable a fair comparison with the standard RBF-FD method, which also typically uses this parameter-free function (see e.g., \cite{flyer2016role,bayona2017role,bayona2019role,jones2023generalized,wright2023mgm}).
A natural extension is  to investigate the use of bell-shaped positive-definite RBFs, such as the Gaussian and Mat\'{e}rn class functions, in the second step of the gRBF-FD framework.
Second, a theoretical justification is lacking for the stability of gRBF-FD, although we observed numerical stability
for the nearly diagonally dominant Laplacian matrices.
Third, we realized that the specification of the $1/K$ weight function is important for the stability of Laplacian-type operators for randomly sampled point cloud data. A promising direction is to investigate weight function selection for gRBF-FD to solve other types of PDEs on manifolds with boundaries, extending the Euclidean-space results of \cite{bayona2017role}. This includes tackling equations such as the advection equation, along with handling derivative boundary conditions like Neumann and Robin types.










\section*{Acknowledgment}

This work was partially supported by the ShanghaiTech University Grant No. 2024X0303-902-01. S. J. was supported by the NSFC Grant No. 12471412 and the HPC Platform of
ShanghaiTech University.

\appendix

\section{Proof of reproduction properties}

\label{app:A}

\textbf{Proof of Lemma \ref{lem:pol}.} For the first property (1) of Lemma %
\ref{lem:pol}, we use the normalized polynomial basis functions for the GMLS
regression of an arbitrary $p=\sum_{j=1}^{m}\tilde{b}_{j}p_{\boldsymbol{%
\alpha }(j)}\left( \boldsymbol{\tilde{\theta}}\left( \mathbf{x}\right)
\right) \in \mathbb{P}_{\mathbf{x}_{0}}^{l,d}$,
\begin{equation*}
\sum_{k=1}^{K}u_{k,\Delta }(\mathbf{x})p(\mathbf{x}_{0,k})=\Delta _{M}%
\mathcal{I}_{p}\mathbf{p}_{{\mathbf{x}_{0}}}\left( \mathbf{x}\right) =\Delta
_{M}\boldsymbol{\tilde{p}}\left( \mathbf{x}\right) \left( \boldsymbol{\tilde{%
P}}^{\top }\boldsymbol{\Lambda \tilde{P}}\right) ^{-1}\boldsymbol{\tilde{P}}%
^{\top }\boldsymbol{\Lambda }\mathbf{p}_{{\mathbf{x}_{0}}},
\end{equation*}%
where $\boldsymbol{\tilde{p}}\left( \mathbf{x}\right) =\left( p_{\boldsymbol{%
\alpha }(1)}(\boldsymbol{\tilde{\theta}}\left( \mathbf{x}\right) ),\ldots
,p_{\boldsymbol{\alpha }(m)}(\boldsymbol{\tilde{\theta}}\left( \mathbf{x}%
\right) )\right) \in \mathbb{R}^{1\times m}$, $\mathbf{p}_{{\mathbf{x}_{0}}%
}=\left( p\left( \mathbf{x}{_{0,1}}\right) ,\ldots ,p(\mathbf{x}%
_{0,K})\right) ^{\top }\in \mathbb{R}^{K\times 1}$ and others are defined in
(\ref{eqn:Phitd}). Here we notice that $\mathbf{p}_{{\mathbf{x}_{0}}}=%
\boldsymbol{\tilde{P}}\mathbf{\tilde{b}}$ with $\mathbf{\tilde{b}}=(\tilde{b}%
_{1},\ldots ,\tilde{b}_{m})^{\top }\in \mathbb{R}^{m\times 1}$. Then we
arrive at
\begin{equation*}
\sum_{k=1}^{K}u_{k,\Delta }(\mathbf{x})p(\mathbf{x}_{0,k})=\Delta _{M}%
\boldsymbol{\tilde{p}}\left( \mathbf{x}\right) \mathbf{\tilde{b}}=\Delta
_{M}p(\mathbf{x}).
\end{equation*}

For the second property (2) of Lemma \ref{lem:pol}, we observe that
\begin{equation*}
(u_{1,\Delta }\left( \mathbf{x}\right) ,\ldots ,u_{K,\Delta }\left( \mathbf{x%
}\right) )=\Delta _{M}\boldsymbol{\tilde{p}}\left( \mathbf{x}\right) \left(
\boldsymbol{\tilde{P}}^{\top }\boldsymbol{\Lambda \tilde{P}}\right) ^{-1}%
\boldsymbol{\tilde{P}}^{\top }\boldsymbol{\Lambda }=\frac{1}{D_{K,\max }^{2}}%
\Delta _{\boldsymbol{\tilde{\theta}}}\boldsymbol{\tilde{p}}\left( \mathbf{x}%
\right) \left( \boldsymbol{\tilde{P}}^{\top }\boldsymbol{\Lambda \tilde{P}}%
\right) ^{-1}\boldsymbol{\tilde{P}}^{\top }\boldsymbol{\Lambda },
\end{equation*}%
where $\Delta _{\boldsymbol{\tilde{\theta}}}$ is the Laplace-Beltrami operator with
respect to the normalized Monge coordinates. We notice that all above
quantities $\Delta _{\boldsymbol{\tilde{\theta}}}\boldsymbol{\tilde{p}}%
\left( \mathbf{x}\right) $ and $\boldsymbol{\tilde{P}}$\ are of $O(1)$,
where the constant in the big-oh is independent of the stencil diameter $%
D_{K,\max }$. Hence, we arrive at $\sum_{k=1}^{K}\left\vert u_{k,\Delta }(%
\mathbf{x})\right\vert \leq C_{1}D_{K,\max }^{-2}$ for some constant $C_{1}$%
\ as desired. The proof is complete.

\begin{figure*}[htbp]
	\centering
	\begin{tabular}{cc}
		{(a) $C_3$ for different $\delta$} & {(b) $C_4$ for different $\delta$} \\
		\includegraphics[width=2.8
		in, height=2.3 in]{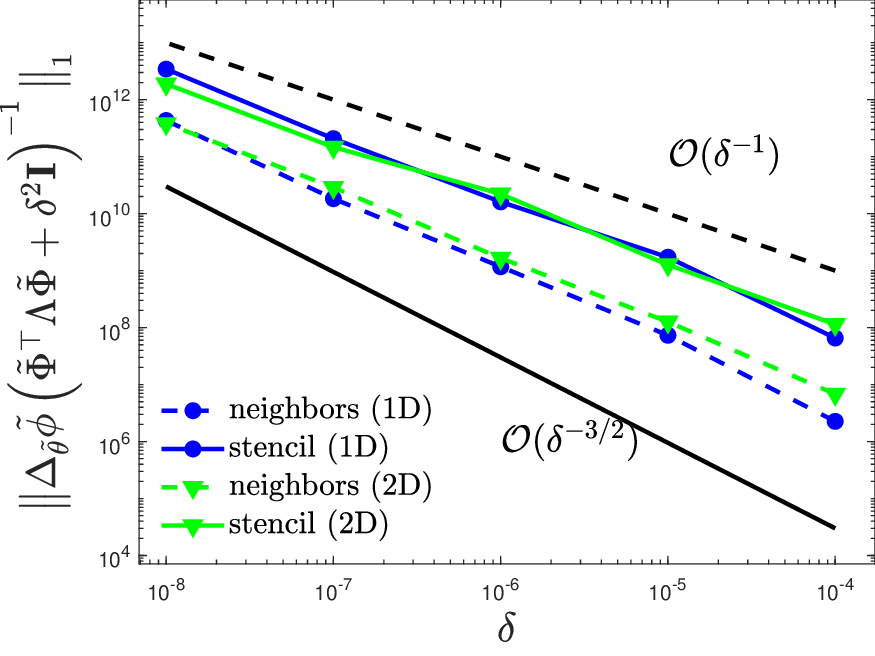} &
		\includegraphics[width=2.8
		in, height=2.3 in]{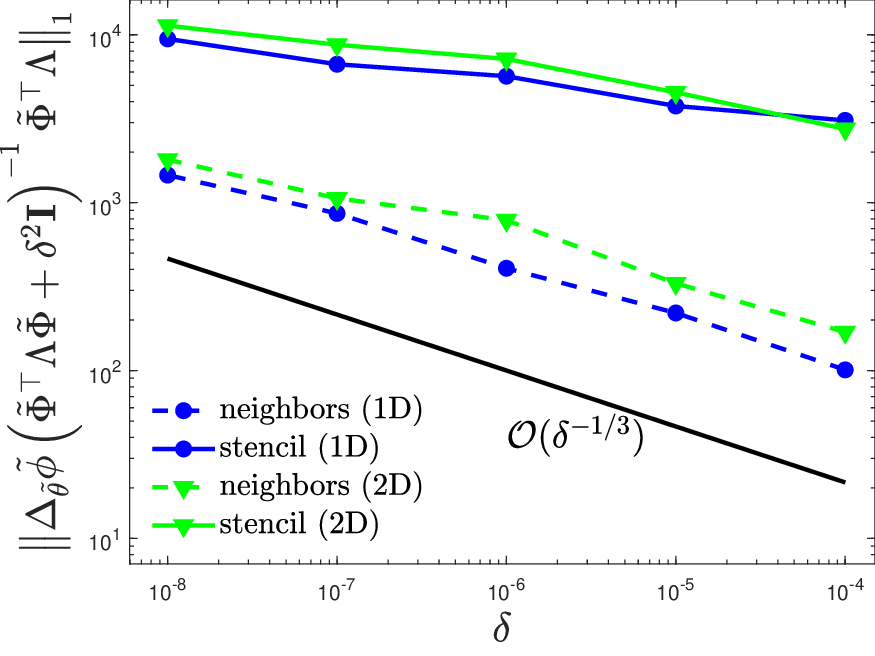}%
	\end{tabular}%
	\caption{ \textbf{1D ellipse in} $\mathbb{R}^{2}$ and \textbf{2D red blood cell (RBC) in} $\mathbb{R}^{3}$  using random data.
		Shown are (a)  the maximum of $\big\Vert\Delta _{\boldsymbol{\tilde{\theta}}}%
		\boldsymbol{\tilde{\phi}}\left( \mathbf{x}\right) \left( \boldsymbol{\tilde{%
				\Phi}}^{\top }\boldsymbol{\Lambda \tilde{\Phi}}+\delta ^{2}\mathbf{I}\right)
		^{-1}\big\Vert_1$ vs. $\delta$ and (b) the maximum of $\big\Vert\Delta _{\boldsymbol{\tilde{\theta}}}%
		\boldsymbol{\tilde{\phi}}\left( \mathbf{x}\right) \left( \boldsymbol{\tilde{%
				\Phi}}^{\top }\boldsymbol{\Lambda \tilde{\Phi}}+\delta ^{2}\mathbf{I}\right)
		^{-1}\boldsymbol{\tilde{\Phi}}^{\top }\boldsymbol{\Lambda}\big\Vert_1$
		vs. $\delta$ for ellipse (blue lines) and RBC (green lines). Here, given a fixed base point $\mathbf{x}_0$,  the maximum is taken on its $K$-nearest neighbors  (dashed lines)  and  the maximum is computed from the resamplings of  200 new points in the local stencil (solid lines). We fix $K_0=30$, $N=1600$ for ellipse and  $K_0=40$,$N=6400$ for RBC. The polynomial degree $l=4$ and the PHS parameter $\kappa=3$. }
	\label{fig:c3_c4}
\end{figure*}


\textbf{Partial proof and numerical verification of Lemma \ref{lem:rbfwk}.}
The proof for the lemma is closely related to the norm estimates of inverses of interpolation matrices for RBFs (see e.g.,  \cite{narcowich1991norms,ball1992sensitivity,narcowich1992norm,schaback1994lower,schaback1995error}).
Here, we use the weighted ridge regression (\ref{eqn:PhiIn}) to compute the inverse for random data. We only  prove the partial results and leave the remaining by numerical verification.
For the first property (1), we use the normalized PHS functions for the
interpolation of an arbitrary $\varphi =\sum_{k=1}^{K}\tilde{c}_{k}\tilde{%
\phi}_{k}\in \mathrm{Span}\{\tilde{\phi}_{1},\ldots ,\tilde{\phi}_{K}\}$,
\begin{eqnarray*}
&&\left\vert \Delta _{M}\varphi (\mathbf{x})-\Delta _{M}\mathcal{I}_{\phi }%
\boldsymbol{\varphi }_{{\mathbf{x}_{0}}}\left( \mathbf{x}\right) \right\vert
=\big\vert\Delta _{M}\boldsymbol{\tilde{\phi}}\left( \mathbf{x}\right)
\mathbf{\tilde{c}}-\Delta _{M}\boldsymbol{\tilde{\phi}}\left( \mathbf{x}%
\right) \left( \boldsymbol{\tilde{\Phi}}^{\top }\boldsymbol{\Lambda \tilde{%
\Phi}}+\delta ^{2}\mathbf{I}\right) ^{-1}\boldsymbol{\tilde{\Phi}}^{\top }%
\boldsymbol{\Lambda \tilde{\Phi}}\mathbf{\tilde{c}}\big\vert \\
&=&\big\vert\Delta _{M}\boldsymbol{\tilde{\phi}}\left( \mathbf{x}\right)
\left( \boldsymbol{\tilde{\Phi}}^{\top }\boldsymbol{\Lambda \tilde{\Phi}}%
+\delta ^{2}\mathbf{I}\right) ^{-1}\left( \boldsymbol{\tilde{\Phi}}^{\top }%
\boldsymbol{\Lambda \tilde{\Phi}}+\delta ^{2}\mathbf{I}\right) \mathbf{%
\tilde{c}}-\Delta _{M}\boldsymbol{\tilde{\phi}}\left( \mathbf{x}\right)
\left( \boldsymbol{\tilde{\Phi}}^{\top }\boldsymbol{\Lambda \tilde{\Phi}}%
+\delta ^{2}\mathbf{I}\right) ^{-1}\boldsymbol{\tilde{\Phi}}^{\top }%
\boldsymbol{\Lambda \tilde{\Phi}}\mathbf{\tilde{c}}\big\vert \\
&=&\delta ^{2}\big\vert\Delta _{M}\boldsymbol{\tilde{\phi}}\left( \mathbf{x}%
\right) \left( \boldsymbol{\tilde{\Phi}}^{\top }\boldsymbol{\Lambda \tilde{%
\Phi}}+\delta ^{2}\mathbf{I}\right) ^{-1}\mathbf{\tilde{c}}\big\vert=\delta
^{2}D_{K,\max }^{-2}\big\vert\Delta _{\boldsymbol{\tilde{\theta}}}%
\boldsymbol{\tilde{\phi}}\left( \mathbf{x}\right) \left( \boldsymbol{\tilde{%
\Phi}}^{\top }\boldsymbol{\Lambda \tilde{\Phi}}+\delta ^{2}\mathbf{I}\right)
^{-1}\mathbf{\tilde{c}}\big\vert,
\end{eqnarray*}%
where $\boldsymbol{\tilde{\phi}}\left( \mathbf{x}\right) =\left( \tilde{\phi}%
_{1}\left( \mathbf{x}\right) ,\ldots ,\tilde{\phi}_{K}\left( \mathbf{x}%
\right) \right) \in \mathbb{R}^{1\times K}$, $\mathbf{\tilde{c}}=(\tilde{c}%
_{1},\ldots ,\tilde{c}_{K})^{\top }\in \mathbb{R}^{K\times 1}$, and other
quantites can be found in (\ref{eqn:Phitd}). Here, in the last equality, the
Laplacian derivative is changed from the unnormalized $\boldsymbol{\theta }$
to the normalized $\boldsymbol{\tilde{\theta}}$. Hence we arrive at the
result, $\left\vert \Delta _{M}\varphi (\mathbf{x})-\Delta _{M}\mathcal{I}%
_{\phi }\boldsymbol{\varphi }_{{\mathbf{x}_{0}}}\left( \mathbf{x}\right)
\right\vert \leq C_{3}\delta ^{2}D_{K,\max }^{-2}\max_{1\leq k\leq
K}\left\vert \tilde{c}_{k}\right\vert $. In Fig. \ref{fig:c3_c4}(a), it can be seen that $C_3 =
\max _{\mathbf{x}\in S_{\mathbf{x}_0}}
\big\Vert\Delta _{\boldsymbol{\tilde{\theta}}}
\boldsymbol{\tilde{\phi}}\left( \mathbf{x}\right) \left( \boldsymbol{\tilde{\Phi}}^{\top }\boldsymbol{\Lambda \tilde{\Phi}}+\delta ^{2}\mathbf{I}\right)^{-1} \big\Vert_1$
is between ${O}(\delta^{-1})$ and ${O}(\delta^{-3/2})$. Consequently, we have $\left\vert \Delta _{M}\varphi (\mathbf{x})-\Delta _{M}\mathcal{I}%
_{\phi }\boldsymbol{\varphi }_{{\mathbf{x}_{0}}}\left( \mathbf{x}\right)
\right\vert = O(\delta ^{1/2}D_{K,\max }^{-2} ) $ .

%

For the second property (2), we have that
\begin{equation*}
(v_{1,\Delta }\left( \mathbf{x}\right) ,\ldots ,v_{K,\Delta }\left( \mathbf{x%
}\right) )=\Delta _{M}\boldsymbol{\tilde{\phi}}\left( \mathbf{x}\right)
\left( \boldsymbol{\tilde{\Phi}}^{\top }\boldsymbol{\Lambda \tilde{\Phi}}%
+\delta ^{2}\mathbf{I}\right) ^{-1}\boldsymbol{\tilde{\Phi}}^{\top }%
\boldsymbol{\Lambda }=D_{K,\max }^{-2}\Delta _{\boldsymbol{\tilde{\theta}}}%
\boldsymbol{\tilde{\phi}}\left( \mathbf{x}\right) \left( \boldsymbol{\tilde{%
\Phi}}^{\top }\boldsymbol{\Lambda \tilde{\Phi}}+\delta ^{2}\mathbf{I}\right)
^{-1}\boldsymbol{\tilde{\Phi}}^{\top }\boldsymbol{\Lambda }.
\end{equation*}%
Let $\boldsymbol{\Lambda }^{1/2}\boldsymbol{\tilde{\Phi}\Lambda }^{1/2}=%
\mathbf{U}\boldsymbol{\Sigma }\mathbf{U}^{\top }$\ be diagonalization of the
symmetric matrix $\boldsymbol{\Lambda }^{1/2}\boldsymbol{\tilde{\Phi}\Lambda
}^{1/2},$ where $\mathbf{U}$\ is orthonormal and $\boldsymbol{\Sigma }=%
\mathrm{diag}\left( \sigma _{1},\ldots ,\sigma _{K}\right) $ is diagonal
possibly with negative eigenvalues. Then
\begin{eqnarray*}
&&\left( \boldsymbol{\tilde{\Phi}}^{\top }\boldsymbol{\Lambda \tilde{\Phi}}%
+\delta ^{2}\mathbf{I}\right) ^{-1}\boldsymbol{\tilde{\Phi}}^{\top }%
\boldsymbol{\Lambda }=\boldsymbol{\Lambda }^{1/2}\left( \boldsymbol{\Lambda }%
^{1/2}\boldsymbol{\tilde{\Phi}}^{\top }\boldsymbol{\Lambda }^{1/2}%
\boldsymbol{\Lambda }^{1/2}\boldsymbol{\tilde{\Phi}\Lambda }^{1/2}+\delta
^{2}\boldsymbol{\Lambda }\right) ^{-1}\boldsymbol{\Lambda }^{1/2}\boldsymbol{%
\tilde{\Phi}}^{\top }\boldsymbol{\Lambda }^{1/2}\boldsymbol{\Lambda }^{1/2}
\\
&=&\boldsymbol{\Lambda }^{1/2}\left( \mathbf{U}\boldsymbol{\Sigma }^{2}%
\mathbf{U}^{\top }+\delta ^{2}\boldsymbol{\Lambda }\right) ^{-1}\mathbf{U}%
\boldsymbol{\Sigma }\mathbf{U}^{\top }\boldsymbol{\Lambda }^{1/2}=%
\boldsymbol{\Lambda }^{1/2}\mathbf{U}\left( \boldsymbol{\Sigma }^{2}+\delta
^{2}\mathbf{U}^{\top }\boldsymbol{\Lambda }\mathbf{U}\right) ^{-1}%
\boldsymbol{\Sigma }\mathbf{U}^{\top }\boldsymbol{\Lambda }^{1/2}.
\end{eqnarray*}%
We now bound the 2-norm of the above matrix,%
\begin{eqnarray*}
&&\Vert \left( \boldsymbol{\tilde{\Phi}}^{\top }\boldsymbol{\Lambda \tilde{%
\Phi}}+\delta ^{2}\mathbf{I}\right) ^{-1}\boldsymbol{\tilde{\Phi}}^{\top }%
\boldsymbol{\Lambda }\Vert _{2}\leq \Vert \boldsymbol{\Lambda }^{1/2}\Vert
_{2}\Vert \mathbf{U}\Vert _{2}\Vert \left( \boldsymbol{\Sigma }^{2}+\delta
^{2}\mathbf{U}^{\top }\boldsymbol{\Lambda }\mathbf{U}\right) ^{-1}%
\boldsymbol{\Sigma }\Vert _{2}\Vert \mathbf{U}^{\top }\Vert _{2}\Vert
\boldsymbol{\Lambda }^{1/2}\Vert _{2} \\
&=&\sqrt{\lambda _{\max }\left( \left( \boldsymbol{\Sigma }^{2}+\delta ^{2}%
\mathbf{U}^{\top }\boldsymbol{\Lambda }\mathbf{U}\right) ^{-1}\boldsymbol{%
\Sigma }^{2}\left( \boldsymbol{\Sigma }^{2}+\delta ^{2}\mathbf{U}^{\top }%
\boldsymbol{\Lambda }\mathbf{U}\right) ^{-1}\right) }=\sqrt{\lambda _{\max
}\left( \left\vert \boldsymbol{\Sigma }\right\vert \left( \boldsymbol{\Sigma
}^{2}+\delta ^{2}\mathbf{U}^{\top }\boldsymbol{\Lambda }\mathbf{U}\right)
^{-2}\left\vert \boldsymbol{\Sigma }\right\vert \right) } \\
&\leq &\sqrt{\lambda _{\max }\left( \left\vert \boldsymbol{\Sigma }%
\right\vert \left( \boldsymbol{\Sigma }^{2}+\delta ^{2}/K\mathbf{I}\right)
^{-2}\left\vert \boldsymbol{\Sigma }\right\vert \right) }=\max_{1\leq i\leq
K}\frac{\left\vert \sigma _{i}\right\vert }{\sigma _{i}^{2}+\delta ^{2}/K}%
\leq \frac{\sqrt{K}}{2\delta },
\end{eqnarray*}%
where we have used the fact that $\boldsymbol{\Sigma }^{2}+\delta ^{2}%
\mathbf{U}^{\top }\boldsymbol{\Lambda }\mathbf{U}\succeq \boldsymbol{\Sigma }%
^{2}+\delta ^{2}/K\mathbf{I}$. Then we obtain that
\begin{eqnarray*}
\sum_{k=1}^{K}\left\vert v_{k,\Delta }(\mathbf{x})\right\vert  &=&D_{K,\max
}^{-2}\Vert \Delta _{\boldsymbol{\tilde{\theta}}}\boldsymbol{\tilde{\phi}}%
\left( \mathbf{x}\right) \left( \boldsymbol{\tilde{\Phi}}^{\top }\boldsymbol{%
\Lambda \tilde{\Phi}}+\delta ^{2}\mathbf{I}\right) ^{-1}\boldsymbol{\tilde{%
\Phi}}^{\top }\boldsymbol{\Lambda }\Vert _{1} \\
&\leq &D_{K,\max }^{-2}\sqrt{K}\Vert \Delta _{\boldsymbol{\tilde{\theta}}}%
\boldsymbol{\tilde{\phi}}\left( \mathbf{x}\right) \left( \boldsymbol{\tilde{%
\Phi}}^{\top }\boldsymbol{\Lambda \tilde{\Phi}}+\delta ^{2}\mathbf{I}\right)
^{-1}\boldsymbol{\tilde{\Phi}}^{\top }\boldsymbol{\Lambda }\Vert _{2}\leq
C_{4}D_{K,\max }^{-2},
\end{eqnarray*}%
for some $C_{4}=O(1/\delta )$, where Cauchy-Schwarz inequality has
been used for the first inequality.
However, the above bound for $C_4$ is not sharp with respect to $\delta$. As can be seen from Fig. \ref{fig:c3_c4}(b), the maximum, $C_4 =
\max _{\mathbf{x}\in S_{\mathbf{x}_0}}
\big\Vert\Delta _{\boldsymbol{\tilde{\theta}}}
\boldsymbol{\tilde{\phi}}\left( \mathbf{x}\right) \left( \boldsymbol{\tilde{\Phi}}^{\top }\boldsymbol{\Lambda \tilde{\Phi}}+\delta ^{2}\mathbf{I}\right)^{-1} \boldsymbol{\tilde{\Phi}}^{\top }\boldsymbol{\Lambda} \big\Vert_1$,
is only around ${O}(\delta^{-1/3})$. Numerically, we have
$\sum_{k=1}^{K}\left\vert v_{k,\Delta }(\mathbf{x})\right\vert = O(\delta ^{-1/3}D_{K,\max }^{-2} ) $ .





\section{Probabilistic stencil size result}

\label{app:B}

To obtain the probabilistic stencil size result, we need the following
result from \cite{croke1980some}:

\begin{lemma}
\label{lem:vol}(Proposition 14 in \cite{croke1980some}) Let $B_{\delta }(x)$
denote a geodesic ball of radius $\delta $ around a point $\mathbf{x}\in M$.
For a sufficiently small $\delta <\iota (M)/2$ with $\iota (M)$\ the
injectivity radius of $M$, we have
\begin{equation*}
\text{Vol}(B_{\delta }(\mathbf{x}))\geq C_{d}\delta ^{d},
\end{equation*}%
where $C_{d}$ is a constant depending only on the dimension $d$ of the
manifold.
\end{lemma}

%

\textbf{Proof of Lemma \ref{lem:hK}. }Let $B_{\delta }(\mathbf{x}_{0})$\
denote a geodesic ball of radius $\delta $\ around the base point $\mathbf{x}%
_{0}\in \mathbf{X}_{M}\subset M$. Let $B_{\delta }^{c}(\mathbf{x}%
_{0})=M\backslash B_{\delta }(\mathbf{x}_{0})$ be the complement of $%
B_{\delta }(\mathbf{x}_{0})$. For a point $\mathbf{x}_{0}\in \mathbf{X}_{M}$
 sampled from $M$ with a density $q(\mathbf{x})$, it holds that
\begin{equation*}
\mathbb{P}_{\mathbf{X}_{M}\sim Q}\left( \mathbf{x}_{0}\in \mathbf{X}_{M}\cap
B_{\delta }^{c}(\mathbf{x}_{0})\right) =1-\int_{B_{\delta }(\mathbf{x}%
_{0})}q(\mathbf{x})dV(\mathbf{x}).
\end{equation*}%
Let $d_{g}(\mathbf{x}_{0,k},\mathbf{x}_{0})$ be the geodesic distance from $%
\mathbf{x}_{0,k}$ to the base $\mathbf{x}_{0}$ for $k=1,\ldots ,K$. For
small $\delta ,$ we can arrive at
\begin{eqnarray*}
&&\mathbb{P}_{\mathbf{X}_{M}\sim Q}(\max_{k\in \{1,\ldots ,K\}}d_{g}(\mathbf{%
x}_{0,k},\mathbf{x}_{0})>\delta )=\mathbb{P}\left( \text{there are at least }%
N-K\text{ points of }\mathbf{X}_{M}\text{ inside of }B_{\delta }^{c}(\mathbf{%
x}_{0})\right) \\
&\leq &\mathbb{P(}\mathbf{x}_{0,K+1}\in B_{\delta }^{c},\ldots ,\mathbf{x}%
_{0,N}\in B_{\delta }^{c})=\left( 1-\int_{B_{\delta }(\mathbf{x}_{0})}q(%
\mathbf{x})dV(\mathbf{x})\right) ^{N-K}.
\end{eqnarray*}%
where $\mathbf{x}_{0,K+1},\ldots ,\mathbf{x}_{0,N}$ are the $N-K$ points out
of the stencil $S_{\mathbf{x}_{0}}$. Using the assumption $q\geq q_{\min }$\
and Lemma \ref{lem:vol} from \cite{croke1980some}, we obtain that%
\begin{eqnarray*}
\mathbb{P}_{\mathbf{X}_{M}\sim Q}(\max_{k\in \{1,\ldots ,K\}}d_{g}(\mathbf{x}%
_{0,k},\mathbf{x}_{0}) >\delta )&\leq& \left( 1-q_{\min }\text{Vol}%
(B_{\delta }(\mathbf{x}_{0}))\right) ^{N-K}\leq \left( 1-q_{\min
}C_{d}\delta ^{d}\right) ^{N-K} \\
&\leq &\exp (-q_{\min }C_{d}(N-K)\delta ^{d}),
\end{eqnarray*}%
where $C_{d}$ is a constant depending on the dimension $d$. Using the fact
that $\left\Vert \boldsymbol{\theta }\left( \mathbf{x}_{0,k}\right) -%
\boldsymbol{\theta }\left( \mathbf{x}_{0}\right) \right\Vert \leq d_{g}(%
\mathbf{x}_{0,k},\mathbf{x}_{0}),$ we arrive at
\begin{equation*}
\mathbb{P}_{\mathbf{X}_{M}\sim Q}(R_{K,\max }\left( \mathbf{x}_{0}\right)
>\delta )\leq \exp (-q_{\min }C_{d}(N-K)\delta ^{d}).
\end{equation*}%
Moreover, taking $\exp (-q_{\min }C_{d}(N-K)\delta ^{d})=\frac{1}{N},$ we
obtain that%
\begin{equation*}
\delta =\left( \frac{\log N}{q_{\min }C_{d}(N-K)}\right) ^{\frac{1}{d}%
}=O\left( \left( \frac{\log N}{N}\right) ^{\frac{1}{d}}\right) .
\end{equation*}%
Since $R_{K,\max }({\mathbf{x}_{0}})\leq D_{K,\max }({\mathbf{x}_{0}})\leq
C_{K}R_{K,\max }({\mathbf{x}_{0}})$, we arrive at
\begin{equation*}
R_{K,\max }\left( \mathbf{x}_{0}\right) \leq \delta =O\left( \left( \frac{%
\log N}{N}\right) ^{\frac{1}{d}}\right) ,\text{ \ }D_{K,\max }\leq
C_{K}\delta =O\left( \left( \frac{\log N}{N}\right) ^{\frac{1}{d}}\right) ,
\end{equation*}%
with probability higher than $1-\frac{1}{N}$.

Subsequently, we compute the expectation and the standard deviation of $%
D_{K,\max }({\mathbf{x}_{0}})$. Let $p_{0}(z)$\ be the density function of
the random variable $Z:=D_{K,\max }({\mathbf{x}_{0}})$. Let $a_{\max }$ be
the diameter of the manifold $M$, that is, $a_{\max }=\max_{\mathbf{x},%
\mathbf{y}\in M}\left\Vert \mathbf{x}-\mathbf{y}\right\Vert $. Then, the
expectation can be bounded by%
\begin{eqnarray*}
ED_{K,\max }({\mathbf{x}_{0}}) &=&\int_{0}^{a_{\max
}}zp_{0}(z)dz=\int_{0}^{C_{K}\delta }zp_{0}(z)dz+\int_{C_{K}\delta
}^{a_{\max }}zp_{0}(z)dz \\
&\leq &C_{K}\delta +a_{\max }P_{\mathbf{X}_{M}\sim Q}(D_{K,\max }({\mathbf{x}%
_{0}})>C_{K}\delta )\leq C_{K}\delta +a_{\max }\exp (-q_{\min
}C_{d}(N-K)\delta ^{d}).
\end{eqnarray*}%
Taking $\delta =\left( \frac{\log N}{dq_{\min }C_{d}(N-K)}\right) ^{1/d}$,
then the expectation becomes%
\begin{equation*}
ED_{K,\max }({\mathbf{x}_{0}})\leq C_{K}\left( \frac{\log N}{dq_{\min
}C_{d}(N-K)}\right) ^{\frac{1}{d}}+a_{\max }N^{-\frac{1}{d}}=O\left( \left(
\frac{\log N}{N}\right) ^{\frac{1}{d}}\right) .
\end{equation*}%
The standard deviation can be calculated similarly,
\begin{equation*}
\sigma \left( D_{K,\max }({\mathbf{x}_{i}})\right) \leq \left( ED_{K,\max
}^{2}({\mathbf{x}_{0}})\right) ^{1/2}\leq \left[ C_{K}^{2}\delta
^{2}+a_{\max }^{2}\exp (-q_{\min }C_{d}(N-K)\delta ^{d})\right] ^{1/2}.
\end{equation*}%
Taking here $\delta =\left( \frac{2\log N}{dq_{\min }C_{d}(N-K)}\right)
^{1/d}$, then the standard deviation becomes%
\begin{equation*}
\sigma \left( D_{K,\max }({\mathbf{x}_{i}})\right) \leq \left[
C_{K}^{2}\left( \frac{2\log N}{dq_{\min }C_{d}(N-K)}\right) ^{\frac{2}{d}%
}+a_{\max }^{2}N^{-\frac{2}{d}}\right] ^{1/2}=O\left( \left( \frac{\log N}{N}%
\right) ^{\frac{1}{d}}\right) .
\end{equation*}
The proof is complete.

\bibliographystyle{abbrv}
\bibliography{arxiv_kme_cleaned}

\end{document}